\documentclass[10pt]{article}

\usepackage[left=2.7cm,bottom=2.7cm,right=2.7cm,top=2.7cm]{geometry}

\usepackage{amsmath,amsfonts}
\usepackage{array}
\usepackage[caption=false,font=normalsize,labelfont=sf,textfont=sf]{subfig}
\usepackage{textcomp}
\usepackage{stfloats}
\usepackage{url}
\usepackage{verbatim}
\usepackage{graphicx}
\usepackage{mathtools}
\usepackage{algorithm}
\usepackage{algorithmic}
\usepackage{cases}
\usepackage{enumerate}
\usepackage{tabularx}
\usepackage{multirow}
\usepackage{bm}
\usepackage{algorithm}
\usepackage{authblk}
\usepackage{algorithmic}
\usepackage{indentfirst}
\usepackage{setspace}
\usepackage{color}
\usepackage[utf8]{inputenc} % allow utf-8 input
\usepackage[T1]{fontenc}    % use 8-bit T1 fonts
\usepackage{hyperref}       % hyperlinks
\usepackage{url}            % simple URL typesetting
\usepackage{booktabs}       % professional-quality tables
\usepackage{amsfonts}       % blackboard math symbols
\usepackage{nicefrac}       % compact symbols for 1/2, etc.
\usepackage{microtype}      % microtypography
\usepackage{lipsum}		% Can be removed after putting your text content
\usepackage{graphicx}
\usepackage{doi}
\usepackage{amssymb}                  
\usepackage{mathrsfs} 
\usepackage[resetlabels]{multibib}
\usepackage[bottom]{footmisc}
\newtheorem{definition}{Definition}
\newtheorem{lemma}{Lemma}
\newtheorem{theorem}{Theorem}

\newtheorem{proposition}{Proposition}[section]
\newtheorem{proof}{Proof}
\newtheorem{remark}{Remark}

\usepackage[utf8]{inputenc}
\usepackage{natbib}

\usepackage{tikz}
\usetikzlibrary{patterns}
\usetikzlibrary{decorations.pathreplacing,calligraphy}
\usetikzlibrary{patterns.meta}
\usetikzlibrary{arrows.meta}
\tikzdeclarepattern{
  name=mylines,
  parameters={
      \pgfkeysvalueof{/pgf/pattern keys/size},
      \pgfkeysvalueof{/pgf/pattern keys/angle},
      \pgfkeysvalueof{/pgf/pattern keys/line width},
      \pgfkeysvalueof{/pgf/pattern keys/pat},
  },
  bounding box={
    (0,-0.5*\pgfkeysvalueof{/pgf/pattern keys/line width}) and
    (\pgfkeysvalueof{/pgf/pattern keys/size},
0.5*\pgfkeysvalueof{/pgf/pattern keys/line width})},
  tile size={(\pgfkeysvalueof{/pgf/pattern keys/size},
\pgfkeysvalueof{/pgf/pattern keys/size})},
  tile transformation={rotate=\pgfkeysvalueof{/pgf/pattern keys/angle}},
  defaults={
    size/.initial=5pt,
    angle/.initial=45,
    line width/.initial=.4pt,
    pat/.initial=dashed,
  },
  code={
      \draw [line width=\pgfkeysvalueof{/pgf/pattern keys/line width},\pgfkeysvalueof{/pgf/pattern keys/pat}]
        (0,0) -- (\pgfkeysvalueof{/pgf/pattern keys/size},0);
  },
}
\usepackage{graphicx} % Required for inserting images

\begin{document}

\title{Exact recovery in the double sparse model: sufficient and necessary signal conditions}

\author[1]{Shixiang Liu}
\author[2]{Zhifan Li}
\author[1]{Yanhang Zhang}
\author[1, 3]{Jianxin Yin}

\affil[1]{\footnotesize School of Statistics, Renmin University of China}
\affil[2]{\footnotesize Beijing Institute of Mathematical Sciences and Applications}
\affil[3]{\footnotesize Center for Applied Statistics and School of Statistics, Renmin University of China}

\date{}
\maketitle \sloppy

\begin{abstract}
The double sparse linear model, which has both group-wise and element-wise sparsity in regression coefficients, has attracted lots of attention recently. 
This paper establishes the sufficient and necessary relationship between the exact support recovery and the optimal minimum signal conditions in the double sparse model. 
Specifically, sharply under the proposed signal conditions, a two-stage double sparse iterative hard thresholding procedure achieves exact support recovery with a suitably chosen threshold parameter.
Also, this procedure maintains asymptotic normality aligning with an OLS estimator given true support, hence holding the oracle properties.
Conversely, we prove that no method can achieve exact support recovery if these signal conditions are violated. 
This fills a critical gap in the minimax optimality theory on support recovery of the double sparse model.
Finally, numerical experiments are provided to support our theoretical findings.
\end{abstract}


\section{Introduction}\label{sec:intro}
We consider the double sparse linear model under simultaneous group sparsity and element sparsity:
\begin{equation}\label{orireg}
    Y = X \beta^*  + \sigma \xi \in \mathbb R^n,
\end{equation}
where $Y \in \mathbb{R}^n$ is the response variable, $X \in \mathbb R^{n\times p}$ is the design matrix and $\beta^* \in \mathbb{R}^p$ is the coefficient vector. The noise level $\sigma$ and the 1-subgaussian random vector $\xi \in \mathbb{R}^n$ (independent of $X$) together constitute the random term. 
We assume the $p$ covariates can be partitioned into $m$ non-overlapping groups $\{G_j\}_{j=1}^m$ with equal group size $d$ (for groups with different sizes, we take $d=\max\{ |G_j| \}_{j=1}^m$ and the results are not affected), such that $\sum_{j=1}^m |G_j|=m \times d=p$. 
Following the definition in \cite{CZ22}, we assume that the coefficient vector $\beta^*$ belongs to the double sparse parameter space 
%Assume that $\beta^* \in \mathbb R ^p$ can be divided into $m$ non-overlapping groups $\{G_j\}_{j=1}^m$ with an equal group size $d$, i.e., $\sum_{j=1}^m |G_j|=m\times d=p$.
%We define the double sparse parameter space follows \cite{CZ22} as 
\begin{equation}\label{eq:oriset}
\beta^* \in  {\Theta}(s,s_0): = \left\{ \beta \in \mathbb R^p \left|~  \sum_{j=1}^m \mathbf1(\beta_{G_j}\ne \mathbf 0_d)\le s, ~ \|\beta\|_0 \le s s_0 \right.\right\},
\end{equation}
where $\| \beta\|_0$ denotes the number of nonzero components of $\beta$, and $\mathbf 1 (\cdot)$ denotes the indicator function.
The double sparse structure in \eqref{eq:oriset} implies that the number of support groups does not exceed $s$, and the total number of support elements does not exceed $ss_0$.
And we claim that $\beta$ is $(s,s_0)$-sparse if $\beta \in  {\Theta}(s,s_0)$.

This double sparsity is particularly valuable in applications such as change-point detection \cite{Zhang15change, hao22twt}, three-dimensional image formation \cite{rs14163945}, multi-task learning \cite{abra2024sgam}, among many others.
Theoretically, most existing studies \citep{CZ22, LZ22, zhang2024minimax, LZ23} focused on the minimax optimal estimation of the signal vector $\beta^*$ based on $\ell_2$ risk. 
However, the performance of support recovery (variable selection) and the asymptotic distribution in the double sparse model remain insufficiently explored.
This paper addresses the theoretical gap by thoroughly analyzing the influence of signal strength on exact support recovery and statistical inference in model \eqref{orireg}.
From a non-asymptotic perspective, we establish the necessary and sufficient minimum signal conditions for achieving exact recovery in the double sparse space.

\subsection{Related work}
% 提纲
% 双稀疏问题是一个具有实际意义的问题，其算法与理论都吸引了很多人研究
% 在理论上，cai与Li都提出了双稀疏模型参数估计的minimax下界，并分别给出了rateoptimal的估计方法。zhang等人基于迭代硬阈值算子，提出了一个tuning-free的双稀疏IHT算法，并能够保证估计结果仍然是双稀疏的。
% 然而，正如fan等所说，一个好的估计算法需要具有oracle properties，也即在适当的假设下满足$\text{supp}(\hat \beta) = \text{supp}( \beta^*) $,以及渐近正态性。
% 目前有关双稀疏模型的估计算法并没有被证实具有这两种oracle properties，因此，对双稀疏模型的深入分析仍是有必要的。
% 对于只有element-wise sparse（假设\| \beta^*\|_0 \le s）的回归模型,\citep{NM20}证明了IHT-type algorithm可以保证minimax估计的optimality，并且可以实现变量选择的一致性，因此我们将基于zhang等人双稀疏迭代硬阈值算法来进行，通过提出一个final procedure，本文证实IHT-type algorithm在双稀疏空间中具有oracle properties。

\paragraph{Double sparse model}
The double sparse structure is of practical significance, attracting extensive research in both algorithmic development \cite{liang2023sparsegl, Breheny15} and theoretical exploration \cite{CZ22, LZ22, LZ23}.
Theoretically, \cite{CZ22, LZ22} have established minimax lower bounds for signal estimation in the double sparse model \eqref{orireg} as 
\begin{equation}\label{eq:gener}
 \inf_{\hat \beta} \sup_{\beta^* \in {\Theta}(s,s_0)} \mathbf E \left( \| \hat \beta - \beta^*\|_2^2 \right) \gtrsim \frac{\sigma^2}n \big( s \log({m}/{s}) + ss_0 \log(d/s_0) \big),
\end{equation}
where the infimum is taken over all possible estimators $\hat \beta$ based on $(Y,X) \in \mathbb R^{n} \times \mathbb R^{n \times p}$, and $\gtrsim $ denotes the inequality up to an absolute constant.
\cite{CZ22, LZ23} provided theoretical analyses of the Sparse Group Lasso (SGLasso), which is introduced by \cite{simon13SGL}. 
This algorithm combines the Lasso penalty \cite{tib96lasso} with the group Lasso penalty \cite{yuan2006model}, integrating both element-wise and group-wise sparsity.
\cite{zhang2024minimax} proposed a tuning-free Double Sparse Iterative Hard Thresholding (DSIHT) algorithm and showed that the obtained estimator is minimax optimal.
However, as \cite{FanLi01,fan04ncp, Zou01122006} pointed out, under proper conditions, a good estimator should possess exact recovery and asymptotic normality, i.e., 
\begin{equation}\label{eq: op}
\text{supp}(\hat{\beta}) := \text{supp}(\beta^*),\quad \sqrt n\left( \hat{\beta}_{S^*} - \beta^*_{S^*} \right) \overset{d} \to N(0 , \Sigma^*),
\end{equation}
where $\Sigma^*$ is the covariance matrix on the true support.
In general, these two properties in \eqref{eq: op} are called the \emph{oracle properties}. To the best of our knowledge, the oracle properties of the double sparse model are currently absent in existing works \cite{CZ22, LZ22, zhang2024minimax, LZ23}. Hence, further analysis of these properties in the double sparse model is necessary.
In the element-wise sparse linear model (assuming only $\|\beta^*\|_0 \leq s$), \cite{NM20} proposed an IHT-type algorithm that ensures minimax optimality and achieves exact support recovery.
This finding motivates us to investigate whether the IHT-type algorithm possesses oracle properties in the double sparse setting.

% 变量选择问题一直是统计领域的一个重点话题，在近期有很多工作聚焦于其minimax最优性。
% 具体而言，minimax性质旨在给出一个separate rate $\lambda$, 并证明当最小信号强度$\min_{j\in S} |\beta^*_j| > C_1\lambda$时，存在某些算法可以实现变量选择一致性；而当$\min_{j\in S} |\beta^*_j| < C_2\lambda$时，任何算法都不能保证这种一致性，其中$C_1>C_2$是两个正常数。
% 对于一个element-wise sparse的p维高斯向量$\xi \sim N(\beta, \sigam^2 I_n)$，在\| \beta^*\|_0 \le s时，bm等人给出了变量选择的minimax separate rate，也即提出了实现变量选择所需要的信号强度的充要条件。
% a，b，c等文章也对此现象做了深入的研究，充分说明了最小信号条件对变量选择的必要性。
% \citep{NM19}揭示了变量选择与信号估计之间的交互联系，并从minimax下界的角度揭示了变量选择对信号估计的增益。
% 然而，这些工作都是基于个element-wise sparsity展开的，并不能适用于双稀疏空间。因此有关双稀疏空间的变量选择研究是具有价值的

\paragraph{Support recovery}
The minimax optimality of exact recovery, i.e., variable selection consistency, is a heatedly discussed topic \cite{WJM07rec, WW10it, butucea17, CM18, butucea23group}. 
Specifically, it corresponds to a minimax separate rate $\mu$, where there exist some algorithms that can achieve exact support recovery if the minimum signal strength $ \min_{j \in S^*} |\beta^*_j| > C_1 \mu$, and no algorithm can ensure such selection consistency if $ \min_{j \in S^*} |\beta^*_j| < C_2\mu$, with $C_1 \geq C_2 > 0$ being two constants. In a $p$-dimensional Gaussian sequence model $X \sim N(\beta^*, \sigma^2 I_n)$ with only element-wise sparsity $\|\beta^*\|_0 \leq s$, \cite{CM18} identified the minimax separate rate for exact recovery as $\sqrt{2 \sigma^2 \log p}$, indicating that exact recovery is impossible if $ \min_{j: \beta^*_j \neq 0} |\beta_j|^* < \sqrt{2 \sigma^2 \log p}$.
Recent studies \citep{Zhengstoev20, Belitser22UQ, castillo24sharp} also explored this issue, emphasizing the necessity of minimum signal conditions for exact recovery. Additionally, \cite{MN19} revealed how support recovery enhances signal estimation from the perspective of the minimax estimation lower bound, highlighting the interaction between these two.
However, these studies concentrate on element-wise sparsity, and their conclusions cannot be directly extended to the double sparse space \eqref{eq:oriset}. 
Therefore, a detailed investigation into the exact recovery for the double sparse model is necessary.

%Consequently, research on variable selection in the double sparse model holds a unique value.

In summary, based on the existing studies, we propose the following questions:
\begin{quote}
In the double sparse model \eqref{orireg}, can the IHT-type algorithm achieve exact support recovery? 
If so, are the minimum signal conditions required for this recovery minimax optimal?
Can these findings facilitate the oracle properties \eqref{eq: op}?
\end{quote}
%\emph{Can the signal strength influence the minimax lower bound in the double sparse model \eqref{orireg}?
%What level of signal strength is required to achieve such improvement?
%What algorithm can break through the ordinary minimax rate of the double sparse model and achieve the oracle rate $\sigma^2 ss_0 / n$? 
%Can these findings facilitate exact support recovery at both the element-wise and group-wise levels?}
This paper provides a theoretical analysis of the above questions and gives affirmative answers.
%Although iterative analysis often yields support recovery \citep{yuan18grad, vas2019implicit, NM20, fan23implicit}, to the best of our knowledge, support recovery results in the double sparse model \eqref{orireg} have not yet been established.
%While hard thresholding methods can offer debiased estimators and support recovery \citep{yuan18grad}, support recovery results in the double sparse model \eqref{orireg} have not been established to our knowledge.

\subsection{Main contribution}\label{subsec:contribute}
This paper establishes the sufficient and necessary relationship between the exact support recovery and the optimal minimum signal conditions in the double sparse model. 
We demonstrate the oracle properties of the double sparse IHT procedure, further enhancing its theoretical foundation.
Specifically, the main contributions of this paper are summarized as follows:  

\begin{enumerate}
\item  Theoretically, we show that the double sparse iterative hard thresholding procedure possesses oracle properties \eqref{eq: op}, guaranteeing that its output achieves both exact support recovery and oracle asymptotic normality, under our proposed minimum signal conditions. 
This result confirms the sufficiency of these signal conditions for achieving exact recovery.
 
\item We analyze the minimax lower bounds for selection error based on Hamming risk. 
Results show that no method can achieve exact support recovery if the proposed minimum signal conditions are violated. 
This result confirms the necessity of these signal conditions and highlights that the double sparse IHT procedure achieves minimax rate-optimality in support recovery.
\end{enumerate}

The innovation of this paper lies in analyzing how signal strength influences the exact recovery and statistical inference of the double sparse linear model \eqref{orireg}. 
From a non-asymptotic framework, we show that a two-stage DSIHT estimator converges to the oracle estimator under rate-optimal signal conditions, emphasizing the theoretical superiority over convex panelized estimators such as the sparse group Lasso. 
Table \ref{table: contri} provides a comparison between double sparse IHT and sparse group Lasso algorithms.

\begin{table*}[t]
\caption{Comparison of Sparse group lasso and Double sparse IHT algorithms in the double sparse model.}
\label{table: contri}
\resizebox{\linewidth}{!}{\begin{tabular}{@{}cccc@{}}
\hline
  & \textbf{Minimax lower bound} & \textbf{Sparse group Lasso} & \textbf{Double sparse IHT} \\  \hline
\textbf{Signal estimation} & {Equation \eqref{eq:gener}}, Cai et al. \cite{CZ22}& {Rate-optimal}, Cai et al. \cite{CZ22} & {Rate-optimal}, Zhang et al.\cite{LZ23} \\
\textbf{Exact recovery} & {Theorems \ref{LB2priors} and \ref{th: LB group}}, ours & Unknown & Theorem \ref{T8.5}, ours  \\
\textbf{Asymptotic Normality}& NA & De-sparsified normality, \cite{CZ22} &  Theorem \ref{T9}, ours \\
&&(under more stringent conditions) & Distributed as oracle estimator \\
\hline
\end{tabular}
}
\end{table*}

\begin{comment}
{
\begin{table*}[t]
\caption{Comparison of support recovery and asymptotic distribution in the double sparse model.}
\label{table: contri}
\begin{tabular}{@{}ccc@{}}
\hline
\textbf{Methods} & \textbf{Support recovery} & \textbf{Asymptotic Normality} \\
\hline
Cai et al. \cite{CZ22} &  NA & De-sparsified normality\\ && (under more stringent conditions)\\ 
Zhang et al. \cite{zhang2024minimax} & {FDR+FNR control } & NA   \\
Ours &  Minimax optimal exact recovery & Distributed as oracle estimator\\
\hline
\end{tabular}
\end{table*}
}
\end{comment}

Additionally, compared to the element-wise sparse model or the group-wise sparse model, the double sparse model provides richer information about the structure of $\beta^*$, thus implying the potential for more accurate outcomes (i.e., smaller sample complexity \citep*{CZ22}). 
However, analyzing the model with simultaneous group-wise and element-wise sparsity is far more difficult than simply combining the two separate types of sparsity. 
Specifically, to achieve exact support recovery, we need to analyze the signal conditions from both the group and element perspectives, examining how each signal condition reduces the estimation (and selection) error. 
This problem is more challenging than addressing either the element-wise \citep{CM18} or group-wise \citep{KM11group, butucea23group} cases individually.

\begin{comment}
\begin{table}[t!] %*** 
\caption{Comparison of variable selection and asymptotic distribution in double sparse model.} 
\label{table: contri0}\par
\resizebox{\linewidth}{!}{\begin{tabular}{|cccc|} \hline %***5truept
\textbf{Papers} & \textbf{Variable selection} & \textbf{Asymptotic distribution} \\ \hline
\cite{CZ22} &  NA & Normal distribution with a larger variance\\  \hline
\cite{zhang2024minimax} & FDR+FNR control & NA \\
\hline
Ours &  Optimal consistency & Oracle distribution \\ \hline
\end{tabular}}
\end{table}

\begin{table}[t!] %***
\newcommand{\tabincell}[2]{\begin{tabular}{@{}#1@{}}#2\end{tabular}}
\caption{Comparison of minimax lower bound, adaptability and recovery in double sparse model.} 
\label{table: contri}\par
\resizebox{\linewidth}{!}{\begin{tabular}{|cccc|} \hline %***5truept
\textbf{Papers} & \textbf{Minimax lower bound} & \tabincell{c}{\textbf{Adaptability to}\\ \textbf{signal strength}} & \textbf{Exact recovery} \\ \hline
\cite{CZ22} &  minimax rate \eqref{eq:gener} & NA & {noiseless case}\\  \hline
\cite{zhang2024minimax} & minimax rate \eqref{eq:gener} & NA & NA \\
\hline
Ours &  \tabincell{c}{minimax rate \eqref{eq:gener} +\\ oracle rate (Theorem \ref{LB2priors})} & \checkmark &  noise case \\ \hline
\end{tabular}}
\end{table}
\end{comment}

\subsection{Organization} 
The rest of the paper is organized as follows:
Section \ref{sec:intro} establishes the problem and notations used throughout the paper.
Section \ref{sec:scaled} introduces the double sparse IHT Algorithm and proves its final estimator has oracle properties under suitable signal conditions.
Section \ref{sec:lower} establishes the minimax lower bound for support recovery based on Hamming loss in model \eqref{orireg}, showing that the minimum signal conditions required in Section \ref{sec:scaled} are minimax rate-optimal.
Section \ref{sec:num} presents numerical experiments to confirm our theoretical findings.
Section \ref{sec:con} contains the conclusion and possible extensions of our study.
Appendix \ref{appA}, \ref{lowerproof}, \ref{appC} and \ref{appD} provide the proof of our results.
%All proofs are available in Supplementary Documents.

\subsection{Notations and preliminaries}\label{sec: pre}
For the given sequences $a_n$ and $b_n$, we say that $a_n = O(b_n)$ when $a_n \le Cb_n$ for some constant $C>0$, while $a_n = o(b_n)$ corresponds to ${a_n}/{b_n} \rightarrow 0$ as $n \to \infty$.
We write $a_n \asymp b_n$ if $a_n = O(b_n)$ and $b_n  = O(a_n)$.
Let $[m]$ denote the set $\{1,2,\ldots,m\}$.
Let $x \vee y= \max\{x,y\} $, and $x \wedge y = \min \{x,y \}$.
%For any set $S$ with cardinality $|S|$,
%let $ X_S = ( X_j, j \in S) \in \mathbb{R}^{n \times |S|}$.
For a vector $\beta$, let $\|\beta \|_2$ denote its Euclidean norm.
For a set $S$, let $|S|$ denote its cardinality.
%and $\|\beta \|_0$ denotes the number of nonzero components of $\beta$. 
%For a matrix $ A$, let $\| A\|_2$ denote its spectral norm and $\| A\|_F$ denotes its Frobenius norm. 
%Let $\mathbf I_p \in \mathbb R^{p \times p}$ denote the identity matrix.
We use $C_i$ to denote absolute constants, whose actual values vary from time to time.
Let $\mathbf{0}_d$ denote the $d$-dimensional zero vector.

We next introduce some more specific notations related to the double sparse model.
We use the double index $(i,j)$ to locate the $i$-th variable in the $j$-th group $G_j$ for $i \in [d], j \in [m]$ in the original parameter vector.
Under the given group structure, each element's location in a vector corresponds to a unique location with the double index; therefore, we will use these two notations interchangeably without further declaration.
%Assume $\| X_{(ij)} \|_2  = \sqrt{n}$, where $X_{(ij)}\in \mathbb R^n$ denotes the corresponding observation vector of the variable $(i,j)$, for all $(i,j) \in [d] \times [m]$.
%and $X_{k,ij}$ denotes the $k$-th observation of this variable.
%Besides, we separate all variables into $m$ non-overlapping groups $G_1, G_2, \ldots, G_m$ with $|G_j| \le d, \forall j \in [m]$. 
For a fixed vector $\beta \in \mathbb R^p$, we denote by $\text{supp}(\beta)  := \{ (i,j) \in [d] \times [m]: \beta_{ij} \ne 0 \}$ the support set of $\beta$, and $G^*(\beta): = \{j \in [m]: \beta_{G_j}  \ne \mathbf 0_d \}$ the group index set of true support groups.
We refer to the coefficient vector $\beta$ as $(s,s_0)$-double sparse if $\beta \in {\Theta}(s,s_0)$. 
Denote by $\mathcal S (s,s_0):=\left\{\text{supp}(\beta):~ \beta \in{\Theta}(s,s_0) \right\} $ the space consisting of all the support set of ${\Theta}(s,s_0)$, that is, if $\beta \in{\Theta}(s,s_0)$, we say the support of $\beta$ belongs to $\mathcal S (s,s_0) $.
Furthermore, we denote by $\beta_S \in \mathbb R^{|S|}$ the subvector of $\beta$ indexed by the set $S$, where $S$ can be any subset of the index space $[d] \times [m]$.
{ Notably, our results can also be considered as $n \rightarrow \infty$ when all other parameters of the problem, i.e., $d, m, s$, and $s_0$, depend on $n$ in such a way that $d = d(n) \rightarrow \infty$.
For brevity, the dependence of these parameters on $n$ will be further omitted in the notation.}
%, and use $S_{G^*} (\beta):= \textstyle \bigcup_{j \in G^*(\beta)} G_j$ to denote the index set consisting of all entries in those true support groups, therefore $S^*(\beta) \subseteq  S_{G^*}(\beta)$.
%We use $\mathcal S (s,s_0) $ to denote the space consisting of all the support index sets of ($s,s_0$)-sparse vector. % and define $s_j(\beta) := \|\beta_{G_j}\|_0$ for all $j \in G^*(\beta) $.

%To avoid the appearance of lengthy expressions, we introduce a function of $s$ and $s_0$ as $ \Delta(s,s_0):= (1/{s_0} )\log({{\rm e} m}/{s}) + \log( {{\rm e} d}/{s_0})$, therefore the ordinary minimax rate in Theorem \ref{LB2priors} is simplified as $ \sigma^2 ss_0 \Delta(s,s_0)/n$.

To facilitate our technical derivation, the design matrix is standardized as $\| X_{(ij)} \|_2  = \sqrt{n}$, where $X_{(ij)}\in \mathbb R^n$ denotes the corresponding observation vector of the variable $(i,j)$, for all $(i,j) \in [d] \times [m]$.
Furthermore, we introduce a fundamental assumption for the design matrix $X$, termed the Double Sparse Restricted Isometry Property (DSRIP).
This assumption is originally introduced in \cite{LZ22} and is an extension of the standard RIP \citep{candes2006robust} into the double sparse space.
\begin{definition}[DSRIP condition]\label{df2}
We say that $X \in \mathbb R^{n\times p}$ satisfies the Double Sparse Restricted Isometry Property $DSRIP(as,bs_0, \delta)$ with a constant $0 < \delta < 1$, if and only if 
\begin{equation}
    n(1-\delta)\|u\|_2^2  
\leq \left\|X_{S} ~ u\right\|_2^2 
\leq n(1+\delta)\|u\|_2^2
\end{equation}
holds for all $ S \in \mathcal S(as,bs_0)$ and $ u \in  \mathbb{R}^{|S|}\setminus \{\mathbf 0_{|S|} \}$, where $a,b>0$ are two constants, and $X_S \in \mathbb R^{n \times |S|}$ is the design matrix of the variables indexed by $S$.
\end{definition}
The DSRIP condition is less strict compared to the RIP condition when considering the double sparsity. 
Specifically, taking $a=b=1$, DSRIP requires an isometry property for all $(s, s_0)$-double sparse vectors. In contrast, RIP requires the satisfaction of all $ss_0$-sparse vectors.
The analyses in this paper are based on the fixed design matrix $X$ satisfying the DSRIP-type condition.

\section{The double sparse IHT algorithm and oracle properties}\label{sec:scaled}
The Iterative Hard Thresholding (IHT) algorithm is an effective method that plays a significant role in compressed sensing \cite{BLUMENSATH2009265, jain2014iterative, yuan18grad, liubarber19}. 
This section illustrates the theoretical properties of the IHT-type algorithm regarding exact support recovery and asymptotic distribution in the double sparse model.
First, Section \ref{3.1} introduces the DSIHT Algorithm \ref{IHT} proposed by \cite{zhang2024minimax}, illustrating that it does not address the analysis of oracle properties. 
Section \ref{3.2} then introduces a two-stage DSIHT Algorithm \ref{scaledIHT} that refines the output of Algorithm \ref{IHT}. 
We show that the final estimator from Algorithm \ref{scaledIHT} achieves a sharper estimation rate and exhibits oracle properties under suitable minimum signal conditions, thereby filling a theoretical gap in the double sparse regression model.
%This section introduces a two-stage double sparse IHT algorithm, which is a novel variant of the standard IHT algorithm \cite{BLUMENSATH2009265, jain2014iterative, yuan18grad, liubarber19}.
%In Section \ref{3.1}, we introduce the double sparse IHT Algorithm \ref{IHT}, which is proposed by \cite*{zhang2024minimax}.
%In Section \ref{3.2}, we propose the second-stage Algorithm \ref{scaledIHT}. This essential algorithm ensures that our final estimator can attain more precise estimates and achieve support recovery under optimal signal conditions.

The analyses in this section are based on the fixed design matrix $X$ satisfying the DSRIP condition (see Definition \ref{df2}).
For clarity, in this section, we set the coefficient vector $\beta^* \in {\Theta}(s,s_0)$ in model \eqref{orireg} as fixed.
We denote its support group index set as $G^*=G^*(\beta^*)$ and its support set as $S^*= S^*(\beta^*)$.

\subsection{One-stage algorithm: minimax optimal estimation}\label{3.1}
First, we introduce the double sparse thresholding operator $\mathcal T_{\lambda,s_0}:\mathbb R^p \to \mathbb R^p$ proposed by \cite{zhang2024minimax}. 
This operator comprises the following two steps:
\begin{enumerate}
    \item The element-wise operator:
        \begin{equation}\label{element}
             \left\{ \mathcal T_\lambda^{(1)}(\beta)\right\}_{ij} = \beta_{ij} \mathbf 1(|\beta_{ij}| \ge \lambda) \in \mathbb R.
        \end{equation}
    \item The group-wise operator:
        \begin{equation}\label{group}
             \left\{ \mathcal T_{\lambda,s_0}^{(2)}(\beta)\right\}_{G_j} 
= \beta_{G_j} \mathbf1 \left( \left\| \beta_{G_j} \right\|_2^2 \ge s_0 \lambda^2  \right) \in \mathbb R^{d}.
\end{equation}
\end{enumerate}
Then, the double sparse hard thresholding operator can be described as $\mathcal T_{\lambda,s_0}:= \mathcal T_{\lambda,s_0}^{(2)} \circ \mathcal T_{\lambda}^{(1)}$. 
Specifically, for each group $G_j$, if $\sum_{i=1}^{d} \left\{ \hat \beta_{ij}^2 \mathbf 1(|\hat \beta_{ij}| \ge \lambda) \right\} \ge s_0\lambda^2$, group $G_j$ will be estimated as a support group and $\left\{ \mathcal T_{\lambda,s_0}(\hat\beta) \right\}_{ij} = \hat\beta_{ij} \mathbf 1(|\hat\beta_{ij}| \ge \lambda)$ for each $ (i,j) \in G_j$.
Conversely, if $\sum_{i=1}^{d} \left\{ \hat\beta_{ij}^2 \mathbf 1(|\hat\beta_{ij}| \ge \lambda) \right\} < s_0\lambda^2$, group $G_j$ will be estimated as a non-support group, therefore $\left\{ \mathcal T_{\lambda,s_0}(\hat\beta) \right\}_{G_j} = \mathbf 0_d$. 

\begin{algorithm}[htbp]
\caption{DSIHT (Double Sparse IHT)}\label{IHT} % with known $s, s_0$ and $\delta$.}
\begin{algorithmic}[1]
    \REQUIRE $X,\ Y,\ \{G_j\}^m_{j=1},\ \kappa,\ \delta,\ \lambda_{(\infty)},\ s_0, \ \sigma$
    \STATE Initialize $t=0$, $\lambda_{(0)} = \frac{ \|X^\top Y /n\|_\infty + \sqrt{10\sigma^2 (\log (dm))/n}}{ \sqrt2 \kappa}$ and $\hat \beta^0 = \mathbf 0_p$ 
    \WHILE {$t \le  \left \lceil \log ( \lambda_{(\infty)} / \lambda_{(0)} )/ {\log \kappa} \right \rceil,$}
    \STATE ${\hat \beta}^{t+1} = \mathcal{T}_{\lambda_{(t)}, s_0}\left({\hat \beta}^{t} + \frac1n X^\top(Y-X{\hat \beta}^{t})\right)$
    \STATE $\lambda_{(t+1)} = \left(\kappa \lambda_{(t)}\right) \vee \lambda_{(\infty)}$
    \STATE $t = t+1$
    \ENDWHILE
    \ENSURE $\hat \beta^{t}$
  \end{algorithmic}
\end{algorithm}

Leveraging the operator $\mathcal T_{\lambda,s_0}$, \cite{zhang2024minimax} introduced the Double Sparse Iterative Hard Thresholding (DSIHT) Algorithm \ref{IHT}, a gradient-descent-based procedure enforcing both group-wise and element-wise sparsity. 
At iteration $t$, the threshold parameter follows $\lambda_{(t)} = \max\{\kappa^t \lambda_{(0)},~\lambda_{(\infty)}\}$, where $\kappa\in(0,1)$ governs a decay from the initial value $\lambda_{(0)}$ to the floor $\lambda_{(\infty)}$.
We next show that, by selecting $\lambda_{(\infty)}$ appropriately, this dynamic regularization strategy guarantees that the output of Algorithm \ref{IHT} processes a double sparse structure, and its $\ell_2$ error is effectively bounded.

\begin{comment}
\begin{figure}
\begin{algorithm}[H]
\caption{DSIHT (Double Sparse IHT)}\label{IHT} % with known $s, s_0$ and $\delta$.}
\begin{algorithmic}[1]
    \REQUIRE $X,\ Y,\ \{G_j\}^m_{j=1},\ \kappa,\ \delta,\ \lambda_{(\infty)},\ s_0, \ \sigma$
    \STATE Initialize $t=0$, $\lambda_{(0)} = \frac{ \|X^\top Y /n\|_\infty + \sqrt{10\sigma^2 (\log (dm))/n}}{ \sqrt2 \kappa}$ and $\hat \beta^0 = \mathbf 0_p$ 
    \WHILE {$t \le  \left \lceil \log ( \lambda_{(\infty)} / \lambda_{(0)} )/ {\log \kappa} \right \rceil,$}
    \STATE ${\hat \beta}^{t+1} = \mathcal{T}_{\lambda_{(t)}, s_0}\left({\hat \beta}^{t} + \frac1n X^\top(Y-X{\hat \beta}^{t})\right)$
    \STATE $\lambda_{(t+1)} = \left(\kappa \lambda_{(t)}\right) \vee \lambda_{(\infty)}$
    \STATE $t = t+1$
    \ENDWHILE
    \ENSURE $\hat \beta^{t}$
  \end{algorithmic}
\end{algorithm}
\end{figure}
\end{comment}

For ease of display, we denote by $\{\hat \beta^t\}$ the estimation sequence obtained from Algorithm \ref{IHT}.
Define the element-wise decoder $\eta_{ij}(\beta) = \mathbf 1 (\beta_{ij} \ne 0)$ and the group-wise decoder $\left(\eta_G\right)_j(\beta) = \mathbf1( \beta_{G_j} \ne \mathbf 0_d)$, for every $\beta \in \mathbb R^p $ with group structure $G_1, \cdots G_m$, and every $(i,j)\in[d] \times [m]$. 
%For the estimator $\tilde\beta^t$, define $\tilde\eta^t_{ij} = \mathbf 1(\tilde\beta^t_{i} \ne 0)$ and $(\tilde\eta^t_G)_j = \mathbf 1( \tilde\beta^t_{G_j} \ne \mathbf 0_d)$, for each $i \in [d]$ and $j \in [m]$.
Using these two decoders, we can characterize the support recovery error both element-wise and group-wise in terms of the corresponding Hamming losses.
{ Additionally, define
$$
    C_\lambda = C_\lambda(\kappa, \delta) := \sqrt{40} \times \frac{\kappa + (\sqrt3 -1)\delta }{\kappa - \delta}
    ,\quad A = A(\kappa, \delta ) := \frac{8 \delta^2}{(\kappa - \delta)^2},
$$
and 
$$
\Delta(s,s_0) : = (1/s_0) \cdot \log(em/s) + \log(ed/s_0).
$$
}

%The next theorem demonstrates that $\hat \beta^t$ maintains a double sparse structure, and its $\ell_2$ error is effectively bounded.

{
\begin{theorem}[Estimation upper bound]\label{upperbound}
    Assume that the design matrix $X$ satisfies DSRIP$\left((1+2A)s,\frac{1+4A}{1+2A}s_0, \delta \right)$ (see Definition \ref{df2}) with $\delta \in (0,1)$.
    Assume that $ss_0 \Delta(s,s_0)= O(n)$ and $\kappa \in (\delta,1)$.
    Then, by taking $\lambda_{(\infty)} = C_\lambda \sigma \cdot \sqrt{\Delta(s,s_0)/n}$, with a probability greater than $1-\exp\left\{-(A \wedge 1) ss_0\Delta(s,s_0)/3 \right\}$, for every $ t \ge 0$ we have the following properties:
    \begin{enumerate}
        \item The estimator $\hat \beta^t$ achieves sparse group selection as
        $$
         \sum_{j=1}^m  \left| \left(\eta_G\right)_j( \hat\beta^t ) - \left(\eta_G\right)_j\left( \beta^* \right) \right|\le (A+1) s.
        $$
        \item The estimator $\hat \beta^t$ achieves sparse element selection as
        $$
         \sum_{j=1}^m \sum_{i=1}^{d} \left| \eta_{ij}( \hat\beta^t ) - \eta_{ij}\left( \beta^* \right) \right|\le (2A+1) ss_0.
        $$
        \item The estimator $\hat \beta^t$ achieves an upper bound as 
        $$
        \|\hat \beta^{t} - \beta^* \|_2  \le \left( 1- \frac{\sqrt{10}}{C_\lambda}\right) \frac{2\sqrt2 \kappa}{\kappa-\delta} \cdot \sqrt{ss_0} \lambda_{(t)} .
        $$
    \end{enumerate}
\end{theorem}
}

Although \cite{zhang2024minimax} established a result akin to Theorem \ref{upperbound}, we restate it here for clarity and to support our subsequent analysis in Section \ref{3.2}:
Theorem \ref{upperbound} proves that the support set of $\hat{\beta}^t$ remains double sparse, that is, $\hat{S}^t \in \mathcal{S}\left( (1+A)s,\frac{1+2A}{1+A}s_0 \right)$ holds for all $t \ge 0$, where the definition of $\mathcal{S}$ follows from Section \ref{sec: pre}. 
By terminating the iterations at $t = \left \lceil \log ( \lambda_{(\infty)} / \lambda_{(0)} )/ {\log \kappa} \right \rceil $, the output of Algorithm \ref{IHT} satisfies
\begin{equation}\label{eq: 1st rate}
\|\hat{\beta}^{t} - \beta^* \|_2^2 \le 
80\cdot \frac{\left[ \kappa + (2\sqrt3 -1)\delta \right]^2 \kappa^2}{(\kappa- \delta)^4}\cdot \frac{ \sigma^2\left[ s \log({em}/{s}) + ss_0 \log(ed/s_0) \right]}n,
\end{equation}
which meets the rate in minimax lower bound \eqref{eq:gener}, demonstrating that Algorithm \ref{IHT} is rate-optimal. 
{Moreover, Theorem \ref{upperbound} allows the DSRIP parameter $\delta$ to range freely over $(0,1)$, and provides an explicit characterization of how both $\kappa$ and $\delta$ influence the sparse pattern and estimation error of $\hat{\beta}^{t}$.
This refined and quantitative insight into the dynamic regularization path goes beyond the earlier studies \cite{LZ22, zhang2024minimax}.
}

\subsection{Two-stage algorithm: oracle properties}\label{3.2}
While the first-stage Algorithm \ref{IHT} yields an estimator that achieves the minimax estimation rate, it tends to omit some true support variables and fall short in terms of estimation accuracy in practice.
Furthermore, Theorem \ref{upperbound} does not ensure exact recovery results, nor does it provide the asymptotic distribution of its estimator.
Next, we will delve deeper into the analysis of the aforementioned issues.

{The first goal in this subsection is to obtain an estimator that achieves exact support recovery.
To realize this, we propose the two-stage DSIHT Algorithm \ref{scaledIHT} with a fixed thresholding parameter $\mu$, which will be clarified in Proposition \ref{T8}.
Specifically, Algorithm \ref{scaledIHT} uses the output of Algorithm \ref{IHT} as the initial input $\tilde\beta^{0}$, and then performs iterative updates in the second stage, producing the estimation sequence $\{\tilde\beta^{t}\}_{t\ge0}$. 
%as
%\begin{equation}
%\mu := \frac{2\sigma}{\sqrt n}  \left\{  \sqrt{ \frac6{s_0} \log ({\rm e} m ) +  6\log({{\rm e} d}/{s_0})  }+ \sqrt{ 3 \log(ss_0)  } \right\}.
%= \frac{2\sigma}{\sqrt n} \left( \sqrt{ 6 \Delta(1,s_0) } + \sqrt{ 3 \log(ss_0) } \right).
%\end{equation}
The essence of Algorithm \ref{scaledIHT} lies in the utilization of the double sparse thresholding operator $\mathcal T_{\mu,s_0}$, which is derived similarly to \eqref{element} and \eqref{group}.

\begin{algorithm}[htbp]
\caption{Two-stage DSIHT}\label{scaledIHT}  % with known $s, s_0$ and $\delta$.}
\begin{algorithmic}[1]
    \REQUIRE $X,\ Y,\ \{G_j\}^m_{j=1},\ \kappa,\ \delta,\ \lambda_{(\infty)},\ \mu,\ s_0, \ \sigma$ 
    \STATE Initialize $t=0$ % and $\mu_2 = \frac{2\sigma}{\sqrt n}  \left\{  \sqrt{ \frac6{s_0} \log ({\rm e} m ) +  6\log({{\rm e} d}/{s_0})  }+ \sqrt{ 3 \log(ss_0)  } \right\}$
    \STATE  $\tilde \beta^0 =$ DSIHT($X,\ Y,\ \{G_j\}^m_{j=1},\ \kappa,\ \delta,\ \lambda_{(\infty)},\ s_0, \ \sigma$ ) \hfill // first-stage Algorithm \ref{IHT} 
    \WHILE {$t\le C \log n,\ $}
    \STATE ${\tilde \beta}^{t+1} = \mathcal{T}_{\mu, s_0}\left({\tilde \beta}^{t} + \frac1n X^\top(Y-X{\tilde \beta}^{t})\right)$
    \hfill // second-stage iteration with a fixed $\mu$
    \STATE $t = t+1$
    \ENDWHILE
    \ENSURE $\tilde \beta^{t}$ 
  \end{algorithmic}
\end{algorithm}

\begin{comment}
\begin{figure}
\begin{algorithm}[H]
\caption{Two-stage DSIHT}\label{scaledIHT}  % with known $s, s_0$ and $\delta$.}
\begin{algorithmic}[1]
    \REQUIRE $X,\ Y,\ \{G_j\}^m_{j=1},\ \kappa,\ \delta,\ \lambda_{(\infty)},\ \mu,\ s_0, \ \sigma$ 
    \STATE Initialize $t=0$ % and $\mu_2 = \frac{2\sigma}{\sqrt n}  \left\{  \sqrt{ \frac6{s_0} \log ({\rm e} m ) +  6\log({{\rm e} d}/{s_0})  }+ \sqrt{ 3 \log(ss_0)  } \right\}$
    \STATE  $\tilde \beta^0 =$ DSIHT($X,\ Y,\ \{G_j\}^m_{j=1},\ \kappa,\ \delta,\ \lambda_{(\infty)},\ s_0, \ \sigma$ ) \hfill // first-stage Algorithm \ref{IHT} 
    \WHILE {$t\le C \log n,\ $}
    \STATE ${\tilde \beta}^{t+1} = \mathcal{T}_{\mu, s_0}\left({\tilde \beta}^{t} + \frac1n X^\top(Y-X{\tilde \beta}^{t})\right)$
    \hfill // second-stage iteration with a fixed $\mu$
    \STATE $t = t+1$
    \ENDWHILE
    \ENSURE $\tilde \beta^{t}$ 
  \end{algorithmic}
\end{algorithm}
\end{figure}     
\end{comment}

For the analysis of exact recovery, let the oracle estimator $\tilde\beta^*$ be defined by
$$
\tilde \beta^*_{S^*} = (X_{S^*}^\top X_{S^*})^{-1} X_{S^*}^\top Y,
\qquad
\tilde\beta^*_{(S^*)^c} =\mathbf0,
$$
where $X_{S^*}\in\mathbb R^{n\times|S^*|}$ contains the columns of $X$ indexed by the true support $S^*$, and $(S^*)^c$ is its complement.
The next proposition quantifies the deviation $\tilde\beta^t - \tilde\beta^*$ produced by Algorithm \ref{scaledIHT} under suitable element-wise and group-wise minimum signal conditions.
Recall $\Delta(1,s_0) = (1/s_0) \cdot \log(em ) + \log(ed/s_0)$.
}

\begin{comment}
we first introduce the notions of element‑wise and group‑wise selection errors, measured in terms of the Hamming loss.
%gpted
We use the Hamming loss to analyze the consistency of the output $\tilde \beta^t$.
To clarify this, define the element-wise decoder $\eta^*_{ij} = \mathbf 1 (\beta^*_{ij} \ne 0)$ and the group-wise decoder $(\eta^*_G)_j = \mathbf1( \beta^*_{G_j} \ne \mathbf 0_d)$. 
For the estimator $\tilde\beta^t$, define $\tilde\eta^t_{ij} = \mathbf 1(\tilde\beta^t_{i} \ne 0)$ and $(\tilde\eta^t_G)_j = \mathbf 1( \tilde\beta^t_{G_j} \ne \mathbf 0_d)$, for each $i \in [d]$ and $j \in [m]$.
Therefore the element selection Hamming loss can be expressed as 
\begin{equation}\label{eq: eham}
 \sum_{j=1}^m \sum_{i=1}^{d} \mathbf 1\left(  \tilde \eta_{ij}^t \ne \eta_{ij}^* \right)=  \sum_{j=1}^m \sum_{i=1}^{d} | \tilde \eta_{ij}^t - \eta_{ij}^* |,
\end{equation}
and the group selection Hamming loss can be expressed as 
\begin{equation}\label{eq: gham}
\sum_{j=1}^m \mathbf1 \left( (\tilde \eta^t_G)_j \ne (\eta^*_G)_j \right) = \sum_{j=1}^m\left| (\tilde \eta^t_G)_j -(\eta^*_G)_j\right|.
\end{equation}
\end{comment}

\begin{proposition}[Convergence to $\tilde \beta^*$]\label{T8}
{Assume that the design matrix $X$ satisfies DSRIP$\left((1+2A)s,\frac{1+4A}{1+2A}s_0, \delta \right)$ condition with $\delta \in (0,1)$, and $ss_0 \Delta(s,s_0)= O(n)$.
We choose the fixed thresholding parameter $\mu$ as 
\begin{equation}\label{eq: mu}
    \mu =\max\left\{ \frac{\kappa C_\lambda}{\delta}, ~
    \sqrt{40+ \frac{120\delta^2}{(1-\delta)^2}}\right\}\cdot \sqrt{\frac{ \sigma^2}n \left\{ \frac{\log(em)}{s_0} + \log(esd) \right\} } ,
\end{equation}
and assume that both the element-wise and the group-wise minimum signal conditions 
\begin{equation}\label{eq: betamin}
\begin{aligned}
    &\min_{(i,j)\in S^*}|\beta^*_{ij}| \ge \left( 2 + \frac{\sqrt6 \delta}{1-\delta} \right) \mu,\\
    &\min_{j \in G^*}\|\beta^*_{G_j}\|_2 \ge 
     \left( 2 + \frac{\sqrt6 \delta}{1-\delta} \right) \sqrt{s_0} \mu 
\end{aligned}
\end{equation}
hold. 
Then, with probability greater than $1-O\left( e^{-\frac13 \left\{ \Delta(1,s_0)+ \log(ss_0) \right\} } \right)$, we have 
\begin{equation}\label{eq:convergence} 
\|\tilde \beta^t - \tilde \beta^* \|_2 \le\left\{ \sqrt{5/6} + \left(1-\sqrt{5/6} \right) \delta \right\}^t \cdot \|\tilde \beta^0 - \tilde \beta^* \|_2, \quad \text{for every } t \ge 0.
\end{equation}
}
\end{proposition}
%gpted
With the thresholding parameter $\mu$ in \eqref{eq: mu} and the minimum signal conditions \eqref{eq: betamin}, we deduce that the solution sequence $ \{\tilde \beta^t \}$ converges to the oracle estimator $\tilde \beta^*$ from a non-asymptotic perspective.
This result quantifies both the optimization error and the computational efficiency of our algorithm, and serves as the key point for the exact support recovery and sharper error bound, as shown below.

{
\begin{theorem}[Exact support recovery]\label{T8.5}
Assume all conditions in Proposition \ref{T8} hold. 
Then, with probability greater than $1-O\left( e^{-\frac13 \left\{ \Delta(1,s_0)+ \log(ss_0) \right\} } \right)$, for all $ t \ge C_\delta \log n$, $\tilde \beta^t$ achieves the exact support recovery at both the group-wise and element-wise levels
$$
 \sum_{j=1}^m  \left| \left(\eta_G\right)_j( \tilde\beta^t ) - \left(\eta_G\right)_j\left( \beta^* \right) \right| =0,
 \quad\sum_{j=1}^m \sum_{i=1}^{d} \left| \eta_{ij}( \tilde\beta^t ) - \eta_{ij}\left( \beta^* \right) \right|=0.
$$
With probability greater than $1-\epsilon-O\left( e^{-\frac13 \left\{ \Delta(1,s_0)+ \log(ss_0) \right\} } \right)$, it also achieves a sharper error bound as 
\begin{equation}\label{eq: sharper}
\left\| \tilde\beta^t - \beta^* \right\|_2 \le \sqrt{\frac{3\sigma^2 }{1-\delta}} \left( \sqrt{\frac{ss_0}{n}} + \sqrt{\frac{\log(1/\epsilon)}{n}} \right).
\end{equation}
\end{theorem}

Theorem \ref{T8.5} demonstrates that, under suitable element-wise and group-wise signal conditions, the output of Algorithm \ref{scaledIHT} can exactly recover the support set at both the group-wise and element-wise levels.
Section \ref{sec:lower} further demonstrates that our signal conditions \eqref{eq: betamin} are necessary in terms of rate. 
Therefore, we show that the double sparse IHT procedure achieves minimax optimality in both signal estimation \cite{zhang2024minimax} and support set recovery. 

Furthermore, under signal conditions \eqref{eq: betamin}, $\tilde\beta^t$ achieves the same rate of estimating $\beta^*$ as if its support were known.
This is a key advantage of our IHT-based method, which is generally difficult to achieve with convex procedures such as the sparse group Lasso \citep{CZ22}.

%The proof of Proposition \ref{T8} is provided in Appendix \ref{appA}.
}

\begin{remark}[Initial estimator of Algorithm \ref{scaledIHT}]\label{remark: general}
    Although Algorithm \ref{scaledIHT} uses the output of Algorithm \ref{IHT} as the initial estimator $\tilde \beta^0$, in theory, it can be replaced with any estimator that achieves the minimax estimation rate \eqref{eq:gener}. 
    For example, the sparse group Lasso \cite{CZ22} or the sparse group slope \cite{LZ23} can serve as $\tilde\beta^0$, rendering our procedure broadly applicable.
    Therefore, the second-stage iteration in Algorithm \ref{scaledIHT} can be realized as a fast debiasing procedure: under signal condition \eqref{eq: betamin}, it removes the bias induced by regularization and ensures that $\tilde \beta^t$ converges to the oracle estimator. 
   % it ensures that the final estimator $\tilde \beta^t$ possesses oracle properties if signal condition \eqref{eq: betamin} is satisfied. 
\end{remark}

\begin{remark}[A delicate proof technnique for support recovery]
To exactly recover the support set $S^*$, one typically needs to obtain an upper bound of $\| \hat \beta - \beta^*\|_{\infty}$, and in this context, incoherence conditions become inevitable \cite{JMLR:v7:zhao06a}.
Some studies avoided these requirements by leveraging the RIP condition, and using $\| \hat \beta - \beta^*\|_2$ to upper bound the $\ell_{\infty}$ error. 
However, those approaches often impose more stringent minimum signal conditions \cite{yuan18grad, Huang18}.
In Proposition \ref{T8} and Theorem \ref{T8.5}, we do not employ this technique and instead achieve exact recovery by proving the convergence of $\tilde \beta^t$ to the oracle estimator $\tilde \beta^*$.
Consequently, we can establish a tight $\ell_\infty$ error bound by using solely the (double sparse) RIP condition, without the need for the incoherence condition in \cite{CZ22}.
\end{remark}

{
\begin{remark}[Interpretation of the fixed threshold $\mu$]
The fixed threshold $\mu$, as defined in \eqref{eq: mu}, is chosen to dominate the statistical error with high probability: 
\begin{itemize}
    \item Outside the true support groups, according to Lemma \ref{exactOGChi2:1st}, a threshold of order $\sqrt{\frac{\sigma^2}{n} \left\{ \frac{\log m}{s_0} + \log(ed/s_0) \right\}}$ ensures that all statistical errors from non-support groups are excluded with high probability.
    \item  Within the true support groups, according to \eqref{eq: inside support group} in the proof of Proposition \ref{T8}, a threshold of order $\sqrt{\frac{\sigma^2}{n} \log(sd)}$ ensures that all statistical errors from the support groups are excluded with high probability.
\end{itemize}

Therefore, by selecting a threshold
$$
\mu \asymp\sqrt{\frac{\sigma^2}{n} \max \left\{ \frac{\log m}{s_0} + \log(ed/s_0),~  \log(sd)\right\} }
\asymp \sqrt{\frac{\sigma^2}{n} \left\{ \frac{\log m}{s_0} + \log(sd)  \right\} },
$$
we effectively filter out statistical errors arising from sub-Gaussian noise. 
Hence, the second stage could separate the signal from the noise and enforce convergence to the oracle estimator.
\end{remark}
}

{
\begin{remark}[Review the DSRIP constant]
Under our DSRIP framework, we do not impose the stringent requirement like $\delta< 0.11$ as in \cite{zhang2024minimax}; instead, we allow $\delta$ to range freely over $(0,1)$.
This generalization theoretically broadens the applicability of our two-stage DSIHT Algorithm \ref{scaledIHT}.
Moreover, our results explicitly quantify the influence of the DSRIP constant $\delta$: as $\delta$ increases to 1, the estimation error bounds \eqref{eq: 1st rate} and \eqref{eq: sharper} inflate, and the required signal strength \eqref{eq: betamin} for support recovery grows larger, while the convergence speed in \eqref{eq:convergence} (toward the oracle estimator) becomes slower.
These characterizations demonstrate how $\delta$ affects both statistical accuracy and computational efficiency. 

Moreover, by choosing a suitable learning rate in each gradient update step, we show that our DSRIP condition is equivalent to a double sparse Riesz condition, namely, the condition
$$
C_L\|u\|_2^2  \le \frac1n\bigl\|X_{S} ~u\bigr\|_2^2 \le C_U\|u\|_2^2,
 \text{ for every } S\in\mathcal S(C_3s,\,C_4s_0) \text{ and } u \in  \mathbb{R}^{|S|}\setminus \mathbf 0_{|S|},
$$
where $C_U\ge C_L>0$ are two arbitrary constants and $C_3,C_4>0$ depend on $C_U$ and $C_L$.
This equivalence weakens our original DSRIP assumption and ensures that our procedure can possess rate-optimal results under some random-design settings. 
Further technical details and proofs are provided in Appendix \ref{sec: random}.

%进一步地，通过在梯度下降时施加一个适当的学习率，我们说明DSRIP条件实际上等价于双重稀疏结构的sparse Riesz condition，即 
%where $C_U \ge C_L >0$ are two arbitrary constants, and $C_3, C_4>0$ 是两个由$C_U，C_L$确定的常数。这一结论进一步放松了我们的DSRIP假设，并保证我们的算法在在一些random design设定中有良好的理论性质。更多技术性的讨论与证明请参见附录\ref{sec: random}

%在本文的DSRIP条件中，我们并没有像\cite{LZ22, zhang2024minimax}那样严苛地约束参数delta充分小，而是允许其在整个(0,1)区间上取值，换言之我们允许矩阵的限制条件数取任何大于1的常数。这在理论上拓展了我们的两阶段DSIHT算法的适用性。此外，我们的理论结果也与delta有关：delta越大，估计结果的误差界会随之变大，为了精确支撑恢复所需要的信号条件也会变得更大，而算法收敛到oracle估计的速度会变慢。这些结论定量地刻画了协变量之间的相关性如何影响输出结果的理论性质。 
%The condition $\delta\le \epsilon^4 \wedge (2/25)$ in Theorem \ref{T8} is imposed to simplify the proof. 
%Indeed, for every $\delta \in (0,1)$, the conclusions in Theorem \ref{T8} persist if we set $\epsilon \geq C_\delta$ and replace $\mu_2$ with $\mu_{C_\delta'}$, where $C_\delta, C_\delta'$ are two positive constants depending solely on $\delta$.
%Similarly, by employing a larger $\lambda_{(\infty)}$, the upper bound for $\delta$ in Theorem \ref{upperbound} can be extended to 1.
\end{remark}
}

% 渐近正态性
Given that Proposition \ref{T8} establishes the convergence of estimator $\tilde \beta^t$ to the oracle $\tilde \beta^*$, it is expected that $\tilde \beta^t_{S^*}$ is asymptotically normally distributed as $\tilde \beta^*_{S^*}$. 
For simplicity, we denote by $c_\xi$ the variance of $\xi_k$ for each $k \in [n]$, and %Consequently, the covariance matrix of $\xi \in \mathbb{R}^n$ is given by $c_\xi I_n$.
{
denote by $B_{S^*}$ an upper bound on the row-wise $\ell_2$-norm of $X_{S^*} \in \mathbb R^{n \times |S^*|}$, i.e., $B_{S^*} := \max_{i\in[n]}\|X_{S^*}^{(i)}\|_2 $, where $X_{S^*}^{(i)} \in \mathbb R^{|S^*|}$ is the $i$-th observation of the covariates indexed by $S^*$.
}

\begin{theorem}[Asymptotic Normality]\label{T9}
{
Assume that all conditions in Proposition \ref{T8} hold and $ B_{S^*}^3= o_p(\sqrt n)$.
}
Then, for each fixed $K >0$ and each matrix $A \in \mathbb R^{K\times |S^*|}$, as $n, d, m\to \infty$, we have 
$$
\sqrt n A\left(\tilde \beta^t_{S^*} - \beta^*_{S^*} \right)
\to N \left( \mathbf0,~ c_\xi \sigma^2 A \left( \frac1n X_{S^*}^\top X_{S^*}\right)^{-1} A^\top\right).
$$
\end{theorem}

Theorem \ref{T9} illustrates that the final estimator $\tilde{\beta}^t$ on the true support has an asymptotic distribution identical to that of the oracle estimator $\tilde \beta^*_{S^*}$, potentially providing a solid foundation for statistical inference. 
Notably, Theorem \ref{T9} remains valid whether $ss_0$ is fixed or diverging, provided $ B_{S^*}^3= o_p(\sqrt n)$.

{
\begin{remark}[The rate of $B_{S^*}$]
In a fixed-design setting, the rate of $B_{S^*}$ is generally hard to examine. 
And hence we turn to the random-design setting for a better understanding. 
Assume that the $i$-th observation $X^{(i)}\stackrel{d}{=}\Sigma^{1/2}Z^{(i)} \in \mathbb R^{p}$, where $\Sigma \in \mathbb{R}^{p \times p}$ is the population covariance and $Z^{(1)},\cdots, Z^{(n)} \in \mathbb R^p$ are i.i.d. centered 1-sub-Gaussian random vectors such that $\mathbf E ( Z^{(i)} Z^{(i)\top} ) = I_p$.
In Appendix \ref{subsec: B}, we prove $B_{S^*} \lesssim \sqrt{ss_0 + \log n}$ with a probability greater than $1-1/n$, and therefore 
$$
\left\{ B_{S^*}^3 = o_p(\sqrt n)\right\} \Leftarrow \left\{ ss_0 = o_p(n^{1/3})\right\},
$$
where the right hand side gives a more intuitive sufficient condition on the design.
%which clarifies the statistical interpretation of our technical condition.

\begin{comment}
在fix design中，B_{S^*}的具体速率通常是难以确定的，但是我们可以在random design setting中更好地理解其速率。
For each $i \in [n]$, consider the vector $X^{(i)} \overset{d}{=} \Sigma^{1/2} Z^{(i)} \in \mathbb{R}^{p}$, where $\Sigma \in \mathbb{R}^{p \times p}$ is the population covariance matrix and $Z^{(i)} \in \mathbb R^p$ is a standard sub-Gaussian vector with i.i.d. entries having variance 1 and $\psi_2$-norm at most a constant $K$. 
Assume the submatrix $\Sigma_{S^*, S^*} \in \mathbb{R}^{|S^*| \times |S^*|}$ has bounded spectral norm as $\Lambda_{\max}(\Sigma_{S^*, S^*} )  \leq C_{S^*}$, where $C_{S^*}$ is an absolute constant.
Then, 在 Appendix \ref{sec: random}， 我们证明了$B_{S^*} \lesssim \sqrt{ss_0 + \log n}$, and therefore 
$$
\left\{ B_{S^*}^3 = o_p(\sqrt n)\right\} \Leftarrow \left\{ ss_0 = o_p(n^{1/3})\right\},
$$
which 帮助我们更好地理解了该条件背后的统计意义。 

 Notably, 定理2.3 holds both when $ss_0$ is fixed and when $ss_0$ diverges, as long as $ B_{S^*}^3= o_p(\sqrt n)$. 
\end{comment}
\end{remark}
}

Consequently, by integrating Theorems \ref{T8.5} and \ref{T9}, we conclude that, under proper signal conditions, the two-stage DSIHT Algorithm \ref{scaledIHT} performs as well as if the true support $S^*$ were known in advance, thereby exhibiting the oracle properties \eqref{eq: op}.
%Our algorithm-driven estimator, though not necessarily globally or even locally optimal, enjoys computational efficiency and minimax optimality in support set recovery.
%Additionally, Proposition \ref{T8} further shows that, as long as a minimax optimal estimator is employed in the first stage, our second-stage iteration, as a refinement step, can ensure oracle properties under appropriate conditions.
%Our results further show that the IHT-type procedure, when used as a de-biasing step in certain other algorithms, can also retain oracle properties under proper conditions.
These findings enrich the theoretical properties of the double sparse IHT and highlight its superiority over the sparse group Lasso \cite{CZ22}.

%Finally, we emphasize that the analysis in this section is not based on the primal-dual witness (PDW, \cite{Wain09PDW}) framework, therefore, we do not require the stringent incoherence condition as used by \cite{CZ22}.
%Instead, we utilize a mild double sparse RIP condition.

%Hence, by following the stopping criterion $t = \left \lceil \log \left( 484 ss_0 \Delta(s,s_0) \right) \right \rceil$, the estimator $\tilde \beta^t$ achieves exact support recovery and ensures sign consistency.
%In contrast to Theorem 6 in \cite{CZ22}, our conclusion imposes weaker conditions on the order of minimum signal strength.

\section{Minimax lower bounds of exact support recovery}\label{sec:lower}
In the previous section, we proved that the two-stage DSIHT Algorithm \ref{scaledIHT} attains exact support recovery at both the group and element levels under the following element-wise and group-wise minimum signal conditions:
\begin{equation}\label{eq: uppercondition}
\begin{aligned}
    &\min_{(i,j) \in S^*} |\beta^*_{ij}| \geq C \sqrt{\frac{\sigma^2}{n} \left( \frac{\log m}{s_0} + \log(sd) \right)}, \\
    &\min_{j \in G^*} \|\beta^*_{G_j}\|_2 \geq C \sqrt{\frac{\sigma^2}{n} \big(  \log m + s_0 \log(sd) \big)}.
\end{aligned}
\end{equation}
In this section, we investigate the necessity of these two conditions. 
{Concretely, we establish minimax lower bounds (under Hamming risk) to show that, if either the element-wise or the group-wise signal condition in \eqref{eq: uppercondition} fails, then no procedure can simultaneously recover the support set.
Thus, the conditions in \eqref{eq: uppercondition} are necessary for exact recovery in the double sparse model.}
For simplicity, we assume that $\xi_{k} \sim N(0,1)$ independently for every $ k \in [n]$.

%%%%% 250812 Yin's idea %%%%%
\subsection{How element-wise signal strength affects recovery}
To clarify the role of minimum signal strength, we analyze the double sparse space in two parts, first focusing on element-wise signal strength and then on group-wise signal strength.
Consider a subspace of $\Theta(s,s_0)$ \eqref{eq:oriset} as 
\begin{equation}\label{omega element}  
{\Theta}_e(s,s_0,a): = \left\{ \beta \in {\Theta}(s,s_0) \left|~ \min_{(i,j) \in \text{supp}(\beta)}| \beta_{ij} | \ge a  \right. \right\},  
\end{equation}  
where $a > 0$ is a parameter that quantifies the element-wise minimum signal strength.
We then construct two subspaces of ${\Theta}_e(s,s_0,a)$ as 
$$
\begin{aligned}
\Theta_{e,1} &:= \left\{ \beta \in \Theta_e \left(s,s_0, a \right) ~\left|~
\begin{aligned}
    & G^*(\beta) = [s], \\
    & \beta_{ij} = a ~\text{ for every } (i,j) \in \operatorname{supp}(\beta)
\end{aligned} \right.\right\},\\
\Theta_{e,2} &:= \left\{ \beta \in \Theta_e\left(s,s_0, a \right) ~\left|~
\begin{aligned}\beta_{ij} =& a\text{ for every } i \in [s_0], j \in G^*(\beta), \\
\beta_{ij}=& 0\text{ otherwise} 
\end{aligned}
\right.\right\}.
\end{aligned}
$$
Specifically, in ${\Theta}_{e,1}$, the locations of support groups are specified (i.e., the first $s$ groups), while the locations of the support elements are undetermined. 
In ${\Theta}_{e,2}$, the locations of the support elements within each support group are specified (i.e., the first $s_0$ elements), while the locations of the support groups are undetermined.
The subspaces ${\Theta}_{e,1}$ and ${\Theta}_{e,2}$ provide a concrete understanding of the compositional structure of the double sparse space.
The subplots (a) and (b) in Figure \ref{fig:examples} provide brief examples of these subspaces.
\begin{figure}[htbp]
\centering
  \subfloat[One example in $\Theta_{e,1}$.]%
        {\begin{tikzpicture}[scale = 0.5]
    \draw[color=black,
        pattern={mylines[size= 5pt,line width=2pt,angle=45,pat=thick]},
        pattern color=black] (0,3) rectangle (1,6);
    \draw[color=black,
        pattern={mylines[size= 5pt,line width=2pt,angle=45,pat=thick]},
        pattern color=black] (1,1) rectangle (2,5);
    \draw[color=black,
        pattern={mylines[size= 5pt,line width=2pt,angle=45,pat=thick]},
        pattern color=black] (2,1) rectangle (3,2);
    \draw[color=black,
        pattern={mylines[size= 5pt,line width=2pt,angle=45,pat=thick]},
        pattern color=black] (2,5) rectangle (3,6);
      % \node[] at (2.1,-1.5) {$G^*$};
      % \node[] at (-0.5,4.5) {$s_0$};
      \node[] at (0.5,-0.5) {$G_1$};
      \node[] at (1.5,-0.5) {$G_2$};
      \node[] at (2.5,-0.5) {$G_3$};
      \node[] at (3.5,-0.5) {$G_4$};
      \node[] at (4.5,-0.5) {$G_5$};
      \node[] at (5.5,-0.5) {$G_6$};
      \node[] at (6.5,-0.5) {$G_7$};
      \node[] at (7.5,-0.5) {$G_8$};
      \draw[step=1,color=gray] (0,0) grid (8,6); %8*6的格子
      \node[] at (4,7) {$\text{supp}(\beta_1)$ with $\beta_1 \in {\Theta}_{e,1}$};
      \end{tikzpicture} }
\hspace{1em}
 \subfloat[One example in $\Theta_{e,2}$.]%
         {\begin{tikzpicture}[scale = 0.5]
    \draw[color=black,
        pattern={mylines[size= 5pt,line width=2pt,angle=45,pat=thick]},
        pattern color=black] (1,3) rectangle (2,6);
    \draw[color=black,
        pattern={mylines[size= 5pt,line width=2pt,angle=45,pat=thick]},
        pattern color=black] (4,3) rectangle (5,6);
    \draw[color=black,
        pattern={mylines[size= 5pt,line width=2pt,angle=45,pat=thick]},
        pattern color=black] (6,3) rectangle (7,6);
      % \node[] at (2.1,-1.5) {$G^*$};
      % \node[] at (-0.5,4.5) {$s_0$};
      \node[] at (0.5,-0.5) {$G_1$};
      \node[] at (1.5,-0.5) {$G_2$};
      \node[] at (2.5,-0.5) {$G_3$};
      \node[] at (3.5,-0.5) {$G_4$};
      \node[] at (4.5,-0.5) {$G_5$};
      \node[] at (5.5,-0.5) {$G_6$};
      \node[] at (6.5,-0.5) {$G_7$};
      \node[] at (7.5,-0.5) {$G_8$};
      \draw[step=1,color=gray] (0,0) grid (8,6); %8*6的格子
      \node[] at (4,7) {$\text{supp}(\beta_2)$ with $\beta_2 \in {\Theta}_{e,2}$};
      \end{tikzpicture} }
\caption{To fully comprehend the subspaces ${\Theta}_{e,1}, {\Theta}_{e,2}$, we take two examples $\beta_1,\beta_2$ from them respectively and use the black solid regions to represent their support sets (take $m=8$, $d=6$ and $s=s_0=3$). We reshape the group structure as $6\times 8$ matrices with each column representing a group. 
In ${\Theta}_{e,1}$ (subfigure (a)), the support groups are the first three groups.
%, while detailed location for support entries are unknown.
In ${\Theta}_{e,2}$ (subfigure (b)), for each support group, only the first three entries are support entries.}\label{fig:examples}
\end{figure}

{
For ease of display, here we rewrite the element-wise decoder $\eta_{ij}^* = \eta_{ij}(\beta^*) = \mathbf 1 (\beta_{ij}^* \ne 0)$ and the group-wise decoder $(\eta_{G}^*)_j = \left(\eta_G\right)_j(\beta^*) = \mathbf1( \beta_{G_j}^* \ne \mathbf 0_d)$.
%Given these settings, we have $\beta^* = a \eta^*$ for every $\beta^* \in {\Theta}_1 \cup {\Theta}_2$, where recall $\eta^* = \{ \mathbf{1}(\beta^*_{ij} \ne 0)\}_{(i,j)\in [d]\times[m]} \in \{0,1\}^p$.
The Hamming losses 
$$ \sum_{(i,j)\in [d]\times[m]} |\hat{\eta}_{ij}  - \eta^*_{ij}|, \text{ and } \sum_{j \in [m]} |(\hat{\eta}_G)_{j}  - ( \eta_G^*)_{j}|
$$ 
are then employed to measure selection errors across different subspaces, where
$$
\hat \eta = \{ \hat \eta_{ij}\}_{(i,j)\in [d]\times[m]} \in \{0,1\}^{d\times m}, \text{ and } \hat\eta_G = \{ (\hat \eta_G)_j \}_{j\in[m]} \in \{0,1\}^{m}
$$ 
are the element-wise and group-wise selectors aiming to recover the true support pattern.}

The following theorem characterizes the minimax lower bounds of selection errors in ${\Theta}_{e,1}$ and ${\Theta}_{e,2}$, respectively, demonstrating the complexity of support recovery in the double sparse space $\Theta_e\left(s,s_0, a \right)$. 

%From the perspective of \emph{model selection}, we have to consider the loss due to variable misidentification. This typically involves analyzing the complexity of the target space and constructing its covering number \citep*{Raskutti2011minimax}.

\begin{theorem}[Minimax lower bounds of selection errors] \label{allselectorD1D2}
Assume that the double sparse model \eqref{orireg} holds with $s<m$ and $s_0<d$. 
Then, for every $ s_0' \in (0,s_0)$, we have an element-wise selection lower bound 
\begin{equation}\label{allselectorD1}
\begin{aligned}
&\inf_{\hat \eta\in \{0,1\}^{d\times m} }~ \sup_{\beta^* \in \Theta_e (s,s_0, a)} ~ {\mathbf{E}}_{Y \sim P_{\beta^*} }\left\{ \sum_{(i,j)\in [d]\times[m]} |\hat{\eta}_{ij}(Y,X) - \eta^*_{ij}| \right\}\\
\ge & \inf_{\hat \eta\in \{0,1\}^{d\times m} } \sup_{\beta^* \in {\Theta}_{e,1}}  {\mathbf{E}}_{Y \sim P_{\beta^*} }\left\{ \sum_{(i,j)\in [d]\times[m]} |\hat{\eta}_{ij}(Y,X) - \eta^*_{ij}| \right\}\\
\ge& \frac{ss_0'}{2s_0}\left\{ (d-s_0) {\Phi}\left(-\frac{a \sqrt n}{2\sigma}-\frac{\sigma\log(d/s_0-1)}{a\sqrt n} \right) \right.\\ 
 & \qquad~~ \left. + s_0 {\Phi}\left(-\frac{a \sqrt n}{2\sigma}+\frac{\sigma \log(d/s_0-1)}{a\sqrt n} \right) \right\} \\
    & -  2s( s_0+s_0') \exp\left(-\frac{3s(s_0-s_0')^2}{2(s_0+2s_0')}\right) .
    %- 4s^2s_0 \exp\left(-\frac{15(s_0'-s_0/10)^2}{40s_0' - s_0}\right).
\end{aligned}
\end{equation}
Additionally, for every $s' \in (0,s)$, we have a group-wise selection lower bound 
%in each support group, the observation vectors of support variables are orthogonal. 
\begin{equation}\label{allselectorD2}
\begin{aligned}
&\inf_{\hat \eta_G \in \{0,1\}^{ m} }~ \sup_{\beta^* \in \Theta_e (s,s_0, a )} ~ {\mathbf{E}}_{Y \sim P_{\beta^*} } \left\{ \sum_{j\in [m]} \left| (\hat \eta_G)_j(Y,X) - (\eta^*_G)_j \right| \right\}\\
\ge & \inf_{\hat \eta_G \in \{0,1\}^m} \sup_{\beta^* \in  {\Theta}_{e,2}} 
    \mathbf{E}_{Y\sim P_{\beta^*} } \left\{ \sum_{j\in [m]} \left| (\hat \eta_G)_j(Y,X) - (\eta^*_G)_j \right| \right\} \\
\ge& \frac{ s'}{2s}\left\{  (m-s) {\Phi}\left(-\frac{a \|\sum_{i\in[s_0]}X_{(ij)}\|_2 }{2\sigma} -
     \frac{\sigma \log(m/s-1)}{a \|\sum_{i\in[s_0]}X_{(ij)}\|_2 } \right) \right.  \\
     & \qquad  \left. +s {\Phi}\left(-\frac{a \|\sum_{i\in[s_0]}X_{(ij)}\|_2 }{2\sigma} + 
     \frac{\sigma \log(m/s-1)}{a \|\sum_{i\in[s_0]}X_{(ij)}\|_2 } \right) \right\} \\
    &- 2(s+s')\exp\left(-\frac{3(s-s')^2}{2(s+2s')} \right).
\end{aligned}
%\inf_{\hat \eta} \sup_{\beta^* \in {\Theta}_2} \underset{Y \sim P_{\beta^*} }{\mathbf{E}} \|\hat \eta - \eta^* \|_2^2 
%\ge \frac{s_0s'}{2s} \psi(m, s, a \sqrt{s_0}, \sigma)
%    - 4ss_0\exp\left(-\frac{(s-s')^2}{2s} \right) ,
\end{equation}
where ${\Phi}(\cdot)$ is the cumulative distribution function of standard normal distribution, % and $\inf_{\hat \eta}$ represents the infimum taken over all possible selector $\hat \eta \in \{ 0,1\}^p$ based on $(Y, X) \in \mathbb R^{n} \times \mathbb R^{n \times p}$, 
and ${\mathbf{E}}_{Y \sim P_{\beta^*} }$ represents the expectation for $Y \sim  N( X\beta^*, \sigma^2 \mathbf I_n)$, as we assume the design matrix $X$ is fixed.
\end{theorem}

Theorem \ref{allselectorD1D2} quantifies the influence of the element-wise signal strength $a$ on support recovery in these two subspaces. 
{To avoid confusion with the notation introduced in Section \ref{sec:scaled}, we clarify the distinction in selectors: In Section \ref{sec:scaled}, each element-wise selector $\eta_{ij}(\tilde \beta^t) = \mathbf 1(\tilde\beta^t_{ij} \ne 0)$ is defined via the estimator $\tilde\beta^t$.
By contrast, in Theorem \ref{allselectorD1D2} we derive lower bounds over arbitrary selectors, not limited to estimator-based selectors, so we employ the more general notation $\hat\eta_{ij}(Y,X)$ to emphasize its dependence only on the observed data $(Y,X)$, rather than on a particular estimator.}
The proof of Theorem \ref{allselectorD1D2} is provided in Appendix \ref{lowerproof}.
Building on Theorem \ref{allselectorD1D2}, we next give two more intuitive minimax lower bounds when $a$ is sufficiently small.
%account of the role of $a$ and show that, when $a$ is sufficiently small, exact support recovery in the double sparse space $\Theta (s,s_0, a)$ is impossible.

{
\begin{theorem}[Necessity of element-wise minimum signal strength]\label{LB2priors}
Assume that the double sparse model \eqref{orireg} holds with $25 \le s< m/2$, $ s_0<d/2$ and $ss_0 \ge 54$. 
Assume that the design matrix $X$ satisfies $DSRIP(s,s_0, \delta)$ with arbitrary $\delta \in (0,1)$. 
If the minimum signal strength $a$ satisfies
\begin{equation}\label{eq: sep}
a^2 \le \frac{\sigma^2}{10n} \left( \frac{\log(m-s)}{s_0(1+\delta)} + \log(sd-ss_0)\right),
\end{equation}
then we have
\begin{equation}\label{eq: elementwise LB 1}
\inf_{\hat \eta\in  \{0,1\}^{d \times m} } \sup_{\beta^* \in {\Theta}_{e}(s,s_0,a)} ~\mathbf{E}_{Y \sim P_{\beta^*} } \left\{ \sum_{(i,j)\in [d]\times[m]} |\hat{\eta}_{ij}  - \eta^*_{ij}| \right\} 
\ge \frac{(ss_0)^{4/5}}{10},
\end{equation}
or
\begin{equation}\label{eq: elementwise LB 2}
\inf_{\hat \eta_G \in \{0,1\}^{ m} }~ \sup_{\beta^* \in \Theta_e (s,s_0, a )} ~ {\mathbf{E}}_{Y \sim P_{\beta^*} } \left\{ \sum_{j\in [m]} \left| (\hat \eta_G)_j  - (\eta^*_G)_j \right| \right\} 
\ge \frac{s^{7/10}}{20}
\end{equation}
\end{theorem}
}

Theorem \ref{LB2priors} shows that if \eqref{eq: sep} holds, then exact support recovery at both element-wise and group-wise levels is impossible: at least one of the two selection tasks must fail.
Moreover, by combining \eqref{eq: sep} with \eqref{eq: uppercondition} we establish the element-wise minimax separation rate $\sqrt{\frac{\sigma^2}n \left( \frac1{s_0} \log m +  \log(sd) \right)}$ for the double-sparse model.
By contrast, in the purely element-sparse parameter space $\Theta(ss_0):=\{\beta\in\mathbb{R}^p~|~\|\beta\|_0\le s s_0\}$ with $p=dm$, the element-wise minimax separation rate required for exact support recovery is $\sqrt{\frac{\sigma^2}{n}\big(\log m + \log d\big)}$ \cite{WJM07rec, CM18}, which is larger than the rate above.  
This gap demonstrates that the double sparse structure facilitates support recovery from the perspective of element-wise signal strength.

{
Additionally, we stress that our focus is on non-asymptotic minimum signal rates rather than on the asymptotic phase-transition phenomenon of support recovery.
Consequently, the numerical constants in Theorem \ref{LB2priors} are conservative and not optimized for sharpness.
The same remark applies to the results in the next subsection.
}

{
\begin{remark}
[Interpretation of the element-wise separation rate]
We can comprehend the composition of the minimax separation rate \eqref{eq: sep} from two perspectives:
\begin{itemize}
    \item Parameter $s_0$ controls the average number of nonzero entries in each group, and determines the minimum signal strength $\frac{\sigma^2}{n}\frac{\log m}{s_0}$ required for group-wise exact recovery. This term arises from the Hamming loss of group selection in subspace $\Theta_{e,2}$.

    \item Parameters $s$ and $d$ control the total number of entries in the active groups, and determine the minimum signal strength $\frac{\sigma^2}{n}\log(sd)$ required for element-wise exact recovery. This term arises from the Hamming loss of element selection in subspace $\Theta_{e,1}$.
\end{itemize}
Taking the maximum of these two thresholds yields $a^2 \gtrsim \frac{\sigma^2}n \left\{ \frac{\log m}{s_0} + \log(sd)\right\}$, which provides the lower bound of element-wise signal strength in \eqref{eq: sep}.
\end{remark}
}

%此外，由于我们追求的是非渐近的信号速率，而不是渐近的支撑恢复相变现象，因此我们使用的常数们并不是最紧的，这一点同样适用于下一小节。
%相反，在仅元素稀疏的参数空间$\Theta_{ss_0}:=\{\beta \in \mathbb R^p: \| \beta\|_0 \le ss_0 \}$中，其为了精确支撑恢复所需要的minimax separation rate为$\sqrt{\frac{\sigma^2}n \left( \log m + \log d \right)}$, 要高于本文所提的速率，这 from the perspective of element-wise signal strength说明了双稀疏结构对支撑恢复的正向影响。 
 
%It also confirms that our two-stage DSIHT Algorithm \ref{scaledIHT} achieves minimax rate-optimal support recovery performance, from the perspective of element-wise signal strength.

\subsection{How group-wise signal strength affects recovery}
We next analyze how the group-wise signal strength influences the support recovery in the double sparse space $\Theta(s,s_0)$ \eqref{eq:oriset}.
Define 
\begin{equation}\label{omega group}  
{\Theta}_g(s,s_0,b): = \left\{ \beta \in {\Theta}(s,s_0) \left|~\min_{j \in G^*(\beta)} \| \beta_{G_j} \|_2 \ge b \right. \right\},  
\end{equation}  
where $b > 0$ is a parameter that quantifies the group-wise minimum signal strength.
Similar to Theorems \ref{allselectorD1D2} and \ref{LB2priors}, we next demonstrate a necessary group-wise signal strength $b$ for the exact support recovery task.

{
\begin{theorem}[Necessity of group-wise minimum signal strength]\label{th: LB group}
Assume that the double sparse model \eqref{orireg} holds with $25 \le s< m/2$, $ s_0<d/2$, $ss_0 \ge 87$, and $s \le 0.061 s_0^{-1/6} \exp(0.1563s_0)$. 
Assume that the design matrix $X$ satisfies $DSRIP(s,s_0, \delta)$ with arbitrary $\delta \in (0,1)$. 
If $b$ satisfies
\begin{equation}\label{eq: groupmin violate}
b^2 \le \frac{\sigma^2}{10n} \left( \frac{\log(m-s)}{(1+\delta)} +\frac{s_0}{20} \log(sd-ss_0)\right),
\end{equation}
then we have
\begin{equation}\label{eq: elementwise LB 3}
\inf_{\hat \eta\in  \{0,1\}^{d \times m} } \sup_{\beta^* \in {\Theta}_{g}(s,s_0, b)}~ \mathbf{E}_{Y \sim P_{\beta^*} } \left\{ \sum_{(i,j)\in [d]\times[m]} |\hat{\eta}_{ij}  - \eta^*_{ij}| \right\}
\ge \frac{(ss_0)^{4/5}}{100},
\end{equation}
or
\begin{equation}\label{eq: elementwise LB 4}
\inf_{\hat \eta_G \in \{0,1\}^{ m} }~ \sup_{\beta^* \in \Theta_g (s,s_0, b )} ~ {\mathbf{E}}_{Y \sim P_{\beta^*} } \left\{ \sum_{j\in [m]} \left| (\hat \eta_G)_j  - (\eta^*_G)_j \right| \right\} 
\ge \frac{s^{7/10}}{20} .
\end{equation}
\end{theorem}
}

{
Theorem \ref{th: LB group} establishes the group-wise minimax separation rate (measured by the group-wise $\ell_2$ norm) as $\sqrt{\frac{\sigma^2}n \big(   \log m + s_0 \log(sd) \big)}$.
Combining Theorems \ref{LB2priors} and \ref{th: LB group}, we show that both the element-wise and the group-wise minimum signal conditions, as illustrated in \eqref{eq: uppercondition}, are necessary for exact support recovery in the double sparse model: omitting either condition precludes simultaneous support recovery at both element and group levels.}
These results confirm that our two-stage DSIHT Algorithm \ref{scaledIHT} attains minimax rate-optimal support recovery performance, thereby demonstrating its theoretical advantage over existing procedures \cite{CZ22, zhang2024minimax}.

{
\begin{remark}[Minimal sample size for exact recovery]
Consider the case that the element-wise and the group-wise minimum signal strengths $\min_{(i,j)\in S^*}|\beta^*_{ij}| =a$ and $ \min_{j \in G^*}\|\beta^*_{G_j}\|_2 =g$ are given at firsthand.
Then, beyond the optimal sample size $n \gtrsim ss_0 \Delta(s,s_0)$ (required for the DSRIP condition), exact support recovery in the double sparse model further requires
$$
\begin{aligned}
n \gtrsim& \max\left\{\frac{\sigma^2}{a^2}\left( \frac{\log m}{s_0} + \log(sd) \right),
~ \frac{\sigma^2}{g^2}\left(\log m + s_0\log(sd) \right)  \right\}\\
\asymp & \frac{\sigma^2}{\min(a^2,~ g^2/s_0)}\left( \frac{\log m}{s_0} + \log(sd) \right).
\end{aligned}
$$
Thus, the element-wise signal $a^2$ and the group-average signal $g^2/s_0$ jointly influence the sample complexity. 
The number of support groups $s$ affects the scale of element-wise selection via the $\log(sd)$ term, while the average group sparsity $s_0$ plays a dual role: 
\begin{itemize}
    \item If $s_0 < g^2/a^2$, then $s_0$ mainly affects the group-selection complexity (through the $(\log m)/s_0$ term).
    
    \item If $s_0 \ge g^2/a^2$, then $s_0$ mainly affects the within-group selection complexity (through the $s_0\log(sd)$ term).
\end{itemize}
\end{remark}
}

\section{Numerical experiments}\label{sec:num}
In this section, we investigate the finite-sample properties of our proposed algorithm. All simulations in this section are computed using R and executed on a personal laptop with an AMD Ryzen 7 5800H processor operating at 3.20 GHz and 16.00GB of RAM.

{
Eight estimators are considered in the estimation: 
The DSIHT method denotes the one-stage Algorithm \ref{IHT} (proposed by \cite{zhang2024minimax}) implemented via the R package $\mathtt{ADSIHT}$, and the Two-stage DSIHT (TS-DSIHT) method is our two-stage procedure (Algorithm \ref{scaledIHT}) that refines the output of DSIHT.
The sparse group Lasso (SGLasso) method \cite{simon13SGL} is fitted using the R package $\mathtt{sparsegl}$ \citep{liang2023sparsegl} and its hyperparameters are selected via 5-fold cross-validation.
The debiased sparse group Lasso (Debiased-SGLasso) method uses a debiasing technique to refine the SGLasso estimates, following equations (22)-(23) in \cite{CZ22}.
The composite MCP (CMCP) method \cite{huang2012selective} is implemented via the R package $\mathtt{grpreg}$ \cite{Breheny15bio},  tuning by 5-fold cross-validation.
Without loss of generality, we also consider the element-wise IHT (IHT-element) and the element-wise Lasso (Lasso-element) as methods under investigation.
Finally, we include the Oracle method, i.e., the OLS estimator fitted on the true support, as a baseline for comparison.

%Eight methods are considered in the estimation: The DSIHT method follows from Algorithm \ref{IHT} (proposed by \cite{zhang2024minimax}) with the adaptive R package $\mathtt{ADSIHT}$. The Two-stage DSIHT (TS-DSIHT) method follows from Algorithm \ref{scaledIHT}. The sparse group Lasso (SGLasso) method \cite{simon13SGL} is computed using the R package $\mathtt{sparsegl}$ \cite{liang2023sparsegl}, with hyperparameters selected according to \cite{CZ22} with 5-fold cross-validation. The debiased sparse group Lasso (Debiased-SGLasso) Method is computed 基于\cite{CZ22}的（22）式和（23）式，which 对SGLasso method的结果进行了修偏。 The composite MCP (CMCP) method \cite{huang2012selective} is computed using the R package $\mathtt{grpreg}$ \cite{Breheny15bio} tuning by 5-fold cross-validation. The Oracle method follows from the OLS estimator based only on the support set.  Without loss of generality, we also consider the element-wise IHT (IHT-element) and the element-wise Lasso (Lasso-element) as methods under investigation.

We now detail the adaptive implementation of the second stage iteration in Algorithm \ref{scaledIHT}.
Based on the double sparse operator $\mathcal T_{\mu,s_0}$, it requires determining appropriate element-wise thresholds $\mu_e$ and group-wise thresholds $\mu_g$. To select these values adaptively, we consider the set
\begin{equation}\label{eq: grids}
\mathcal G_L = \left\{ \left(n^{-\frac{\ell_1-1}{2(L-1)}}, n^{-\frac{\ell_2-1}{2(L-1)}} \right): 1 \le \ell_2 < \ell_1\le L\right\} ,
\end{equation}
which provides a range of candidate pairs $(\mu_e, \mu_g)$ based on a geometric series.
We then use 5-fold cross-validation to select the optimal pair $(\hat{\mu}_e, {\hat \mu}_g)$ that minimizes prediction error and then refit the second stage iteration on the full data using this choice.
This method is denoted as TS-DSIHT-CV with taking $L =10$. 
For comparison, we also consider a non-adaptive version, denoted as TS-DSIHT-True, which assumes that the true parameters $(\sigma,s,s_0)$ are known.
%我们接下来介绍如何自适应实现算法2的第二阶段迭代。根据双稀疏算子$\mathcal T_{\mu,s_0}$的定义,我们需要确定合适的元素阈值$\mu_e = \mu$与组阈值$\mu_g = \sqrt{s_0} \mu$. 定义格点集合
%其基于几何级数提供了一系列潜在的$(\mu_e, \mu_g)$取值。随后我们使用五折交叉验证，基于最小化预测误差来确定最优的阈值组合$(\hat{\mu}_e, {\hat \mu}_g)$, 并将之应用到全体数据上。以上方法简记为 TS-DSIHT-CV。 作为对比，我们记TS-DSIHT_True为已知参数$(\sigma,s,s_0)$时的两步算法。
}

\subsection{Simulation 1: the influence of signal strength}\label{num1}
We first introduce our model setting.
%First, we analyze the influence of signal strength on the estimation and support recovery.
Assume that data are generated from the model $y = X \beta^* + \sigma \xi$, where $\sigma=1$ and $\beta^* \in \mathbb R^p$ has $m=50$ separable groups with equal group size $d=40$. 
We randomly draw the design matrix $X \in \mathbb R^{n \times p}$ from $N(0,1)$ independently and standardize by columns to ensure $\| X_{(ij)} \|_2^2 = n$ for each $(i,j) \in [d] \times [m]$, where $n = 300$. 
We set $s = s_0 = 5$, and 
\begin{equation}\label{eq: k}
\begin{aligned}
\beta_{ij}^* \sim 
\text{Unif}  \left(  k\sqrt{\frac{\sigma^2}n\left( \frac{\log(m/s)}{s_0} + \log(d/s_0)\right)} , 2k \sqrt{\frac{\sigma^2}n\left( \frac{\log(m/s)}{s_0} + \log(d/s_0)\right)} \right )
\end{aligned}
\end{equation}
for all $ (i,j) \in S^*$ independently, where $k \in [1,4]$ reflecting the signal strength.

%In the previous work, \cite{CZ22} compared sparse group Lasso with Lasso and group Lasso, and \cite{LZ22} compared double sparse IHT with IHT and group IHT.
%Their results demonstrate that, under the double sparse space those estimators with only a single sparse penalty cannot achieve minimax optimality in signal estimation. 

We evaluate the estimation performance through the $\ell_2$ error rate $\|\hat \beta -\beta^* \|_2/ \|\beta^* \|_2$, and evaluate the support recovery performance using Hamming loss and Matthews Correlation Coefficient \cite{MCC75} at both the group-wise and element-wise levels (shown as Hamming Element, Hamming Group, MCC Element and MCC Group). We also compare the average computation time of each method.

%All simulations in this section are computed using R and executed on a personal laptop equipped with an AMD Ryzen 7 5800H processor operating at 3.20 GHz and 16.00GB of RAM.
%The median run-times for the DSIHT, Enhanced DSIHT, SGLasso are $41.25,~ 108.04, ~118.49 ~ms$, respectively.
%execution time required for the simulations in this section {are} within one hour.

\subsubsection{Estimation and support recovery}\label{sec: simu1}
\begin{figure}[t]
\centering
\includegraphics[scale=0.6]{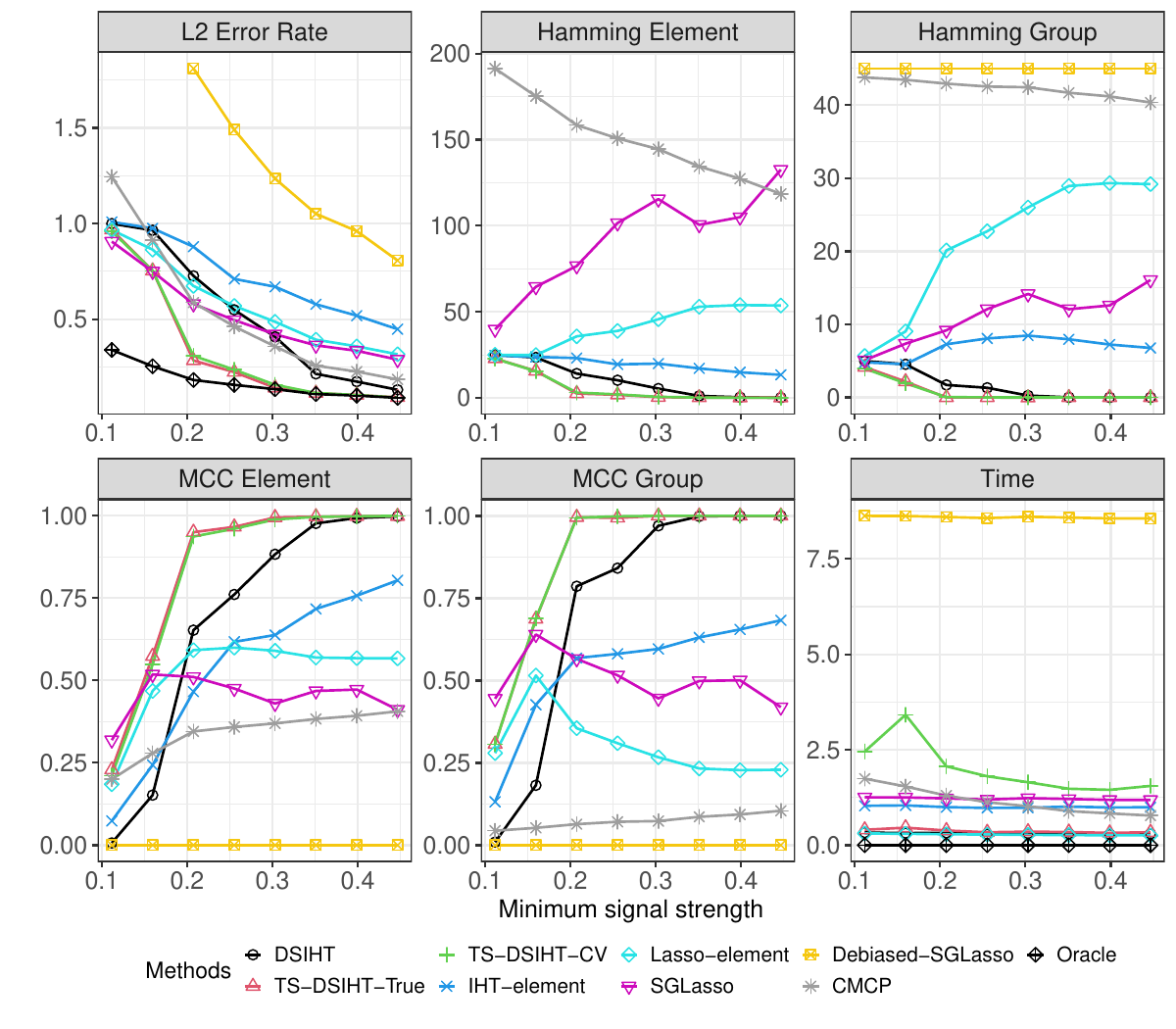}
\caption{Performance metrics with increasing signal strength.
    The x-axis represents the minimum signal strength. %, i.e., $k \sigma\sqrt{ \Delta(s,s_0)/n }$ with $k \in [1,4]$.
    Each point is averaged from 300 Monte Carlo simulations.
    The Oracle method (fitted on the true support $S^*$) has its MCCs always at 1 and Hamming loss always at 0, hence we ignore them.
    The Debiased-SGLasso produces a desparsified estimator with relatively large $\ell_2$ errors, consequently, its extreme values are omitted from the L2 Error Rate plot. Additionally, its element-wise Hamming loss is always at $2000-25=1975$, so it is omitted from the Hamming Element plot.
    }\label{fig:fig1} 
\end{figure}

Here, we analyze the influence of signal strength on the estimation and support recovery.
From Figure \ref{fig:fig1}, it is evident that when the signal strength is less than $0.20$, all high-dimensional methods perform poorly, showing substantial deviations from the Oracle estimates.
{As signal strength increases, both TS-DSIHT-True (using the true $(\sigma,s,s_0)$) and the data-driven TS-DSIHT-CV converge toward the Oracle estimator and attain nearly identical estimation accuracy.
This indicates that cross-validation can identify suitable threshold pairs $(\hat{\mu}_e, {\hat \mu}_g)$, allowing our output to adaptively achieve the oracle estimation rate when key parameters are unknown in practice.
Moreover, both methods recover the true support $S^*$ at both the group-wise and element-wise levels once the signal strength reaches roughly $0.30$.
Therefore, in the following simulations, we abbreviate the cross-validated procedure TS-DSIHT-CV as TS-DSIHT, referring to the adaptive version of Algorithm \ref{scaledIHT}.
%因此，我们在接下来的模拟中将TS-DSIHT-CV方法简写为TS-DSIHT，用来代表自适应版本的算法2.
%As the signal strength increases, 基于已知真参数$(\sigma, s, s_0)$的TS-DSIHT-True和数据驱动的TS-DSIHT-CV方法都能够converges towards the Oracle method, 且二者的估计精度几乎一致。这说明在关键参数未知时，我们仍然可以在实际中得到具有信号自适应性的估计结果。此外，这两种方法都能recover the support set $S^*$ at both the group-wise and element-wise levels with a signal strength of around $0.30$.
The DSIHT method consistently requires stronger signals for similar performance.
Although the SGLasso and the CMCP methods show decreasing estimation error rate with stronger signals, %when tuned by cross-validation,
both methods remain substantially biased to the Oracle estimation and exhibit poor support recovery performance. 
The Debiased-SGLasso method removes shrinkage bias and yields asymptotic normality, but it no longer produces sparse estimates, which leads to high $\ell_2$ errors and prevents reliable support recovery.

The total runtime of our full two-stage DSIHT method is larger than that of one-stage methods such as SGLasso and CMCP. However, the runtime of our second-stage refinement itself is comparable to those methods.
Moreover, the two-stage Debiased-SGLasso directly requires a precision matrix estimation step (see (23) in \cite{CZ22}) that our method avoids, making our method substantially faster than it.

%compared with the two-stage Debiased-SGLasso, our method is substantially faster because Debiased-SGLasso requires an additional precision-matrix estimation step (see (23) in \cite{CZ22}) that our method avoids.
%在计算时间方面，我们的最终方法是两步的，并且需要基于cv调参，因此所需总计算时间相较于SGLasso和CMCP等方法要更多，但是其第二次迭代的单独用时与SGLasso和CMCP等方法是相近的。此外，相较于同为两步算法的Debiased SGLasso，我们的算法耗时是显著低的，这是因为Debiased SGLasso 需要计算进行一个图模型估计，而我们的方法规避掉了这一点。
}
%Furthermore, 尽管SGLasso和CMCP方法基于交叉验证得到的估计误差随着信号的增强而有所下降，但二者距离oracle估计的精度仍有很大差距，且其支撑恢复效果很不理想。此外，Debiased-SGLasso方法可以修偏并实现渐近正态性，但是其估计结果不再具有稀疏性，导致L2误差变大且不具有支撑恢复效果。
%even when signals are sufficiently large, the SGLasso estimators still exhibit significant bias and fail to recover the support set at the element level.

\subsubsection{Asymptotic property}\label{simu: asy} 
To assess the asymptotic normality, we set $ k = 2.7$ in equation \eqref{eq: k} and plot the histograms for $\sqrt n \left(\hat\beta_{(1,1)} - \beta_{(1,1)}^* \right)$ and $\sqrt n \sum_{i=1}^5\left(\hat\beta_{(i,3)} - \beta_{(i,3)}^* \right)$ across sample sizes $n = 300, 600, 900, 1200$, where we denote by $\beta_{i,j}$ the coefficient of the $i$-th covariate in the $j$-th group.
{We also consider using our second-stage iteration in Algorithm \ref{scaledIHT} (lines 3-6) as a refinement procedure (with 5-fold cross-validation) for the sparse group Lasso estimator, and refer to it as the SGLasso-DSIHT method. 

\begin{figure}[t]
\centering
\includegraphics[scale=0.6]{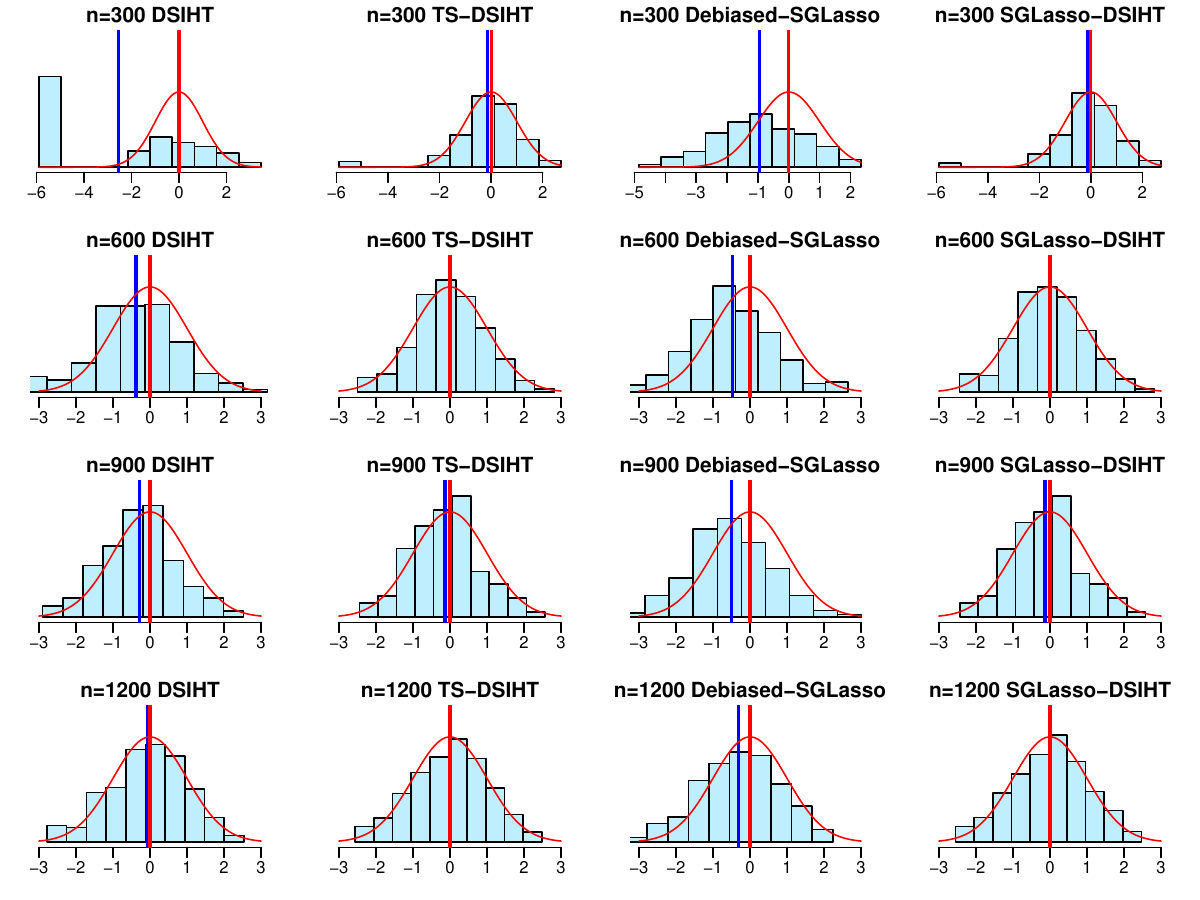}
\caption{The histograms of $\sqrt n \left(\hat\beta_{(1,1)} - \beta_{(1,1)}^* \right)$ with 300 Monte Carlo simulations conducted for each sample size. Blue vertical lines represent the sample means, and red vertical lines represent the population means.
}\label{fig21}
\end{figure}

\begin{figure}[t]
\centering
\includegraphics[scale=0.6]{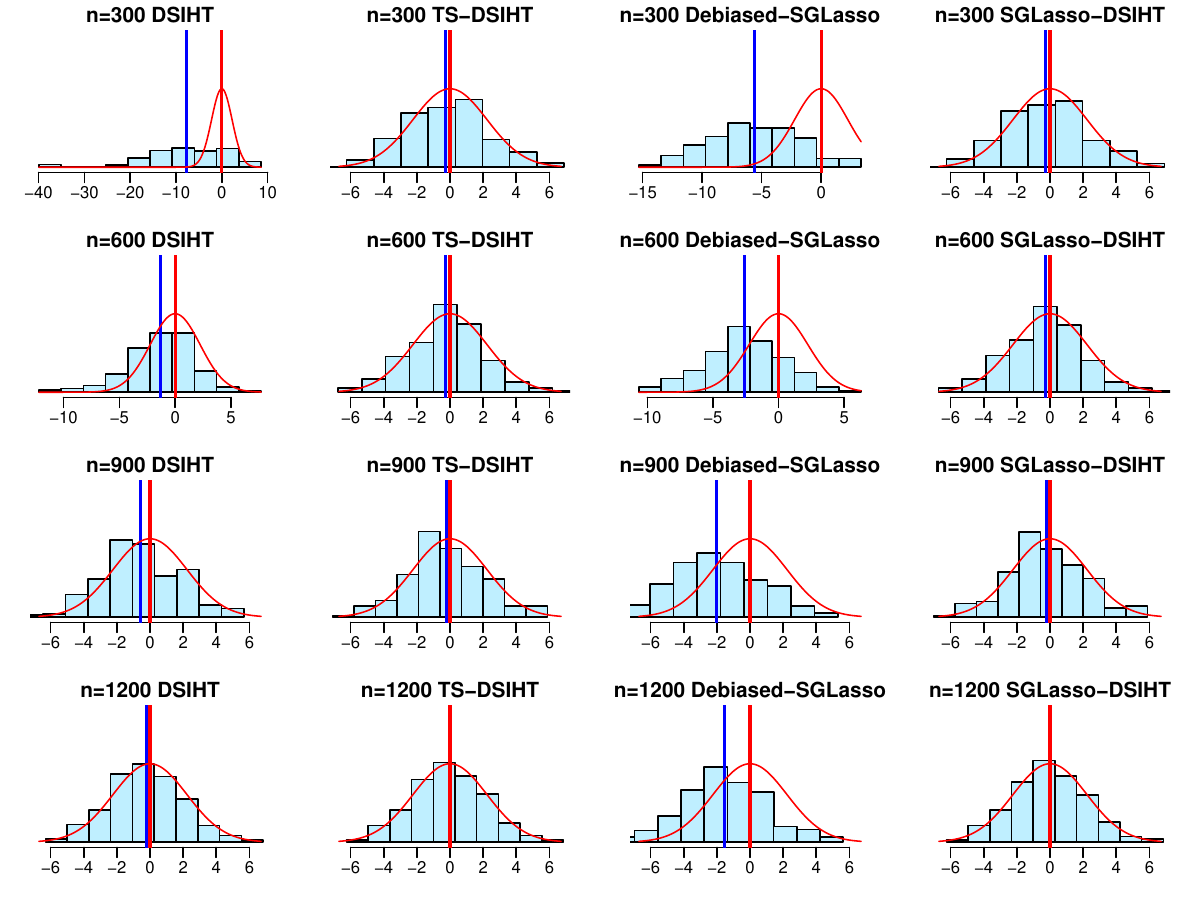}
\caption{The histograms of $\sqrt n  \sum_{i=1}^5\left(\hat\beta_{(i,3)} - \beta_{(i,3)}^* \right)$ with 300 Monte Carlo simulations conducted for each sample size. Blue vertical lines represent the sample means, and red vertical lines represent the population means.
}\label{fig22}
\end{figure}

Figure \ref{fig21} and \ref{fig22} compare the histograms of the DSIHT, TS-DSIHT, Debiased-SGLasso, and SGLasso-DSIHT methods, superimposed with the Gaussian density curves (red curves). 
The results show that as the sample size $n$ increases, all four methods exhibit asymptotic normality. 
However, the TS-DSIHT and SGLasso-DSIHT methods produce distributions that are closer to the normal shape than those of the DSIHT and Debiased-SGLasso methods. 
This indicates that (i) our proposed second-stage iteration improves asymptotic properties, and (ii) the procedure could be broadly applicable as a refinement step to enhance the statistical behavior of other methods. 
We provide further evidence for point (ii) in the next experiment.
}
%The results indicate that as the sample size $n$ increases，这四种方法都有渐近正态性质，但是TS-DSIHT与SGLasso-DSIHT方法表现很相近，且相较于DSIHT与Debiased-SGLasso要更接近正态分布形状。这说明（i）本文提出的第二阶段迭代算法的渐近性质更良好；（ii）该算法具有普适性，可以作为一种refinement算法，来增强其他算法的统计性质。我们将在下一小节针对第二点结论进行更详细的验证。

\begin{comment}
    \begin{figure}[t]
  \subfloat[Histogram of $\sqrt n \left(\hat\beta_{(2,1)} - \beta_{(2,1)}^* \right)$.]%
        { \includegraphics[scale=0.45]{pic/simuasy1re.pdf}}
 \hspace{0.5em}% 
  \subfloat[Histogram of $\sqrt n  \sum_{i=1}^5\left(\hat\beta_{(i,1)} - \beta_{(i,1)}^* \right)$.]%
         {\includegraphics[scale=0.45]{pic/simuasy2.pdf}}
\caption{The histograms of some specific parameters with 300 Monte Carlo simulations conducted for each sample size. Blue vertical lines represent the sample means, and red vertical lines represent the population means.
}\label{fig2}
\end{figure}
\end{comment}

\subsubsection{Universality of the second-stage iteration}

{
Building upon the analysis in Remark \ref{remark: general} and the simulation results from Section \ref{simu: asy}, we now verify the universality of our second-stage iteration in Algorithm \ref{scaledIHT}.
We use the SGLasso and CMCP estimators, respectively, as a warm start for our second-stage algorithm, and then obtain the refined estimators, which we term SGLasso-DSIHT and CMCP-DSIHT.
We adopt the same parameter settings as in Section \ref{sec: simu1}, and the simulation results are presented in Figure \ref{fig15}.
The SGLasso-DSIHT and CMCP-DSIHT estimators demonstrate substantially better performance in both estimation and support recovery compared to their initial estimators.
As the signal strength increases, these refined estimators, along with TS-DSIHT, all achieve the oracle estimation rate and exact support recovery.
This indicates that our proposed second-stage iteration is effective and has a general applicability as a refinement method to debias or enhance other estimators.

%基于Remark \ref{remark: general}的分析和上一小节的模拟结果，我们在这里验证第二阶段迭代算法的普适性。我们分别在SGLasso和CMCP的估计量后面接一个算法2中的second-stage iteration，将得到的refinement估计分别称作SGLasso-DSIHT和CMCP-DSIHT。我们继承Section \ref{sec:simu1}的参数设定，模拟结果如图\ref{fig15}所示。
%SGLasso-DSIHT和CMCP-DSIHT的估计与支撑恢复效果都明显好于其各自的初始估计。此外，随着信号逐渐增强，这些refined estimators都能够像TS-DSIHT那样，达到oracle estimation rate并实现精确支撑回复。这说明真正起到帮助的是本文提出的第二部迭代算法，且其作为一种修偏/refinement 方法，具有general的普适性。
}
 
\begin{figure}[t]
\centering
\includegraphics[scale=0.6]{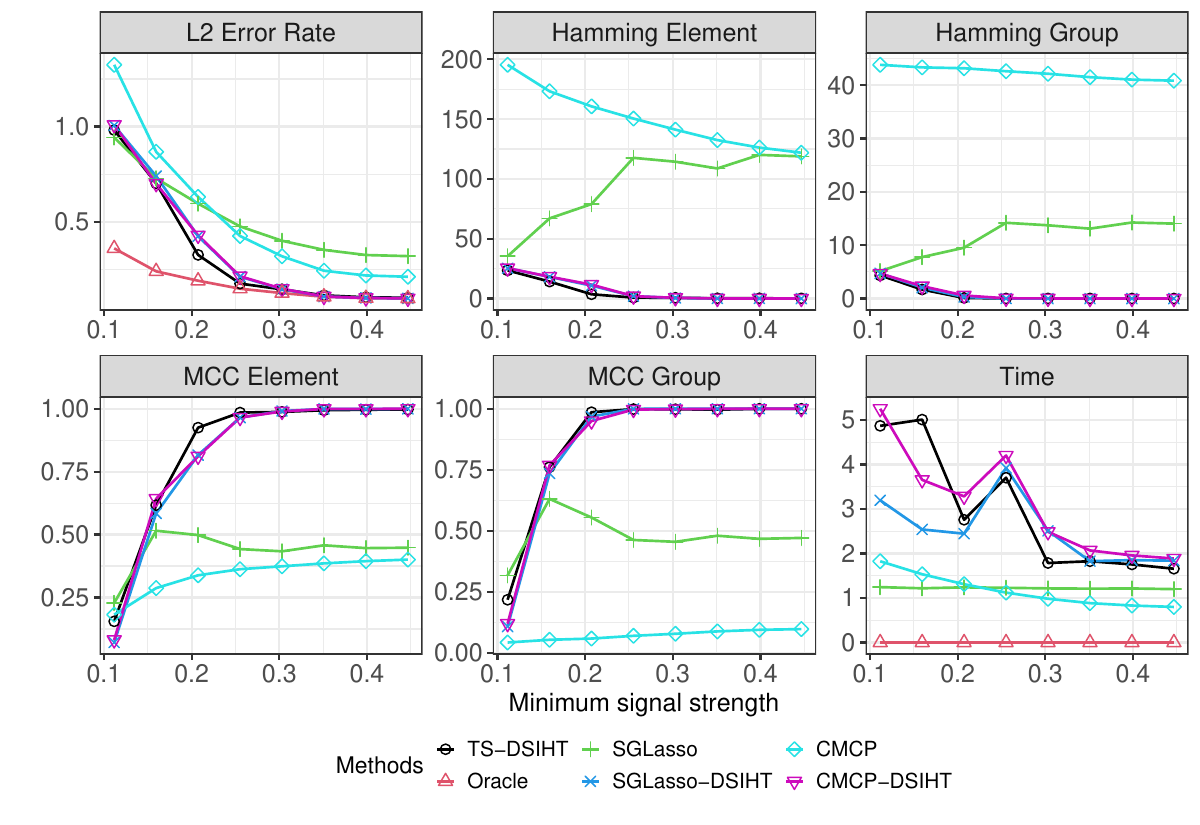}
\caption{Refinement effect of the second-stage DSIHT with increasing signal strength. Each point is averaged from 300 Monte Carlo simulations.
In any method that incorporates the second-stage DSIHT iteration, this specific part relies on the grid points in \eqref{eq: grids} and employs 5-fold cross-validation for data-driven estimation.
}\label{fig15}
\end{figure}

\subsection{Simulation 2: the influence of sparsity levels}\label{num2}
This subsection analyzes the influence of sparsity level $s$ and $s_0$ in the estimation and support recovery.
We assume $n=500$ and $ss_0=48$, and the remaining parameters $d=40,~ m=50, ~\sigma=1$ are identical to those in Section \ref{num1}.
We consider the double sparse spaces under the following three cases, respectively:
\begin{itemize}
    \item \textbf{Case A.} $s_0=4,~ s=12$.
    \item \textbf{Case B.} $s_0=8,~ s=6$.
    \item \textbf{Case C.} $s_0=12,~ s=4$.
\end{itemize}

To mitigate confounding effects, we assume equal signal strength across support sets in both cases, i.e., $\beta_{ij}^*=a$ for all $(i,j)$ in the support set of each case.
We now analyze the performance of TS-DSIHT (Algorithm \ref{scaledIHT}) under these three cases (TS-DSIHT A, TS-DSIHT B, and TS-DSIHT C), and assess their deviation from the corresponding Oracle estimators (Oracle A, Oracle B, and Oracle C).
The $\ell_2$ error rate $\|\hat \beta -\beta^* \|_2/\| \beta^*\|_2$, element-wise and group-wise Hamming loss, and Matthews Correlation Coefficients are used to measure the effectiveness.

With the fixed $ss_0=48$, Figure \ref{fig:fig2} shows that a larger $s_0$ (Case C) leads to a weaker elemental signal strength required for the element-wise recovery (see subplots ``Hamming Element '' and ``MCC Element''), consistent with the element-wise minimum signal condition in Proposition \ref{T8}, i.e., of the signal rate 
$\frac{\sigma}{\sqrt n} \sqrt{\frac{\log m}{s_0} + \log(sd)}.$
However, a larger $s_0$ also results in a stronger group signal strength required for the group recovery, consistent with the group-wise minimum signal condition in Proposition \ref{T8}, i.e., of the rate $ \frac{\sigma}{\sqrt n} \sqrt{ \log m + s_0\log(sd) }.$
Therefore, Case A is favored in support group recovery (see subplots ``Hamming Group '' and ``MCC Group''). 
Hence, as the intra-group sparsity $s_0$ increases, achieving equivalent support recovery outcomes necessitates a more stringent requirement on group strength, while the requirement on element-wise strength becomes more relaxed.

\begin{figure}[t]
\centering
\includegraphics[scale=0.6]{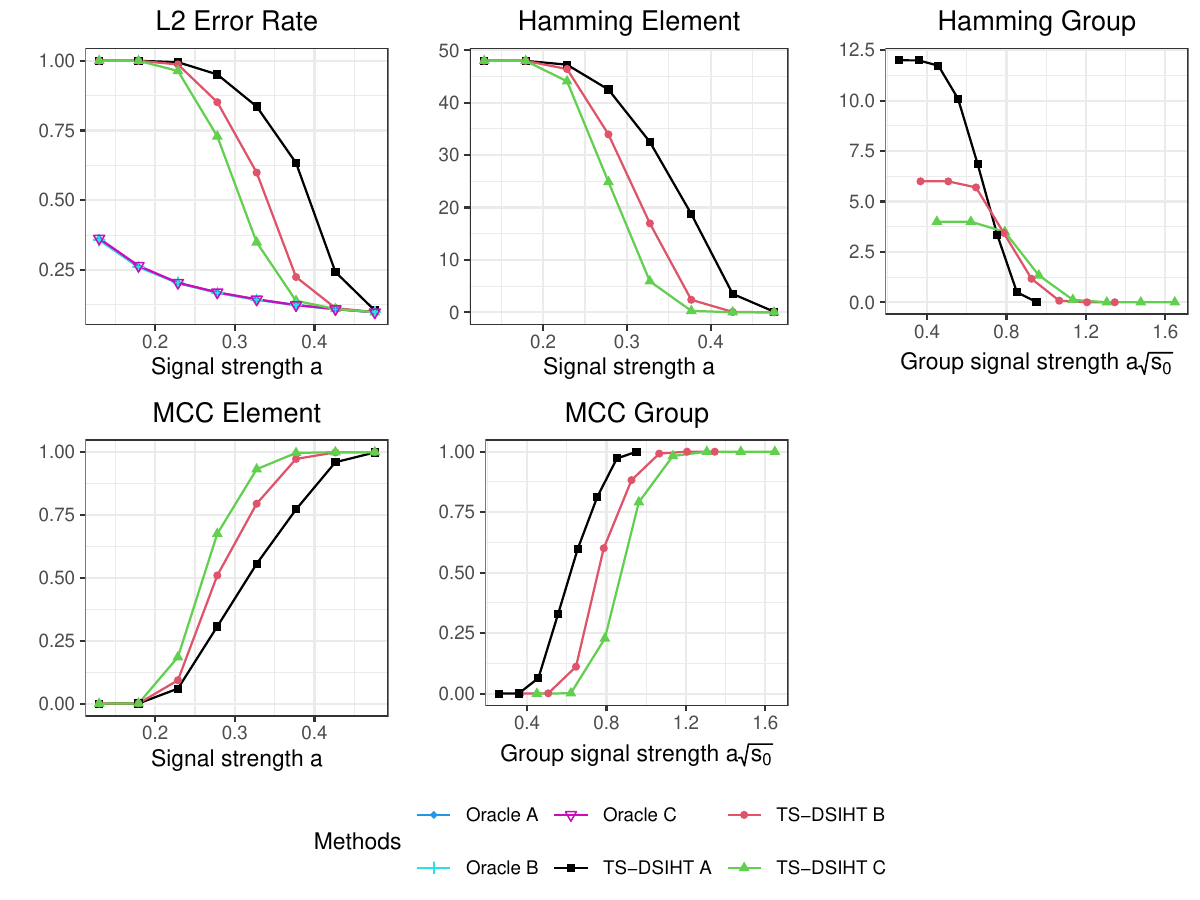}
\caption {Performance metrics with different sparsity levels.
    The x-axis in the first two subplots represents the element signal strength $a$, and in the third subplot, it represents the group signal strength $a\sqrt{s_0}$.
    Each point is averaged from 300 Monte Carlo simulations.}\label{fig:fig2}
\end{figure}

\subsection{{Real data analysis}}
We finally evaluate the methods on a supermarket sales dataset \cite{wang2009} containing $n=464$ daily observations and 6,398 product-level sales volume.
The response is the daily customer count. 
We standardize all variables and keep the 300 predictors with the largest absolute Pearson correlation with the response.
Each selected predictor is then expanded into an 8-term B-spline basis, yielding a double sparse design with $m=300$ and $d=8$ that can capture flexible, potentially nonlinear effects of product sales.
 
For performance assessment, we performed 100 random 80/20 train/test splits (371 observations for training and 93 for testing in each split). 
Each fitted model is measured on the test set using the mean squared prediction error (MSPE). 
Table \ref{tab: real} reports the mean and standard deviation (over 100 random splits) of MSPE, the number of selected groups, and the number of selected elements.
TS-DSIHT achieves the lowest average MSPE and selects a substantially smaller model than most of the competing methods. 
These results underscore the practical advantages of our method: better predictive accuracy together with a relatively interpretable model.

\begin{table*}[htbp]
\caption{The computational results of each method based on 100 random partitions. The standard deviation is shown in parentheses. }
\label{tab: real}
\resizebox{\linewidth}{!}{\begin{tabular}{@{}cccc@{}}
\hline
\textbf{Method}  & \textbf{MSPE} & \textbf{Number of selected groups} & \textbf{Number of selected elements} \\  \hline
\textbf{TS-DSIHT} & {\bf0.468 (0.105)} & 3.860 (3.315) & 21.240 (15.588) \\ 
\textbf{DSIHT} &  0.503 (0.121) & 2.290 (1.066)  & 8.510 (4.596) \\  
\textbf{SGLasso} & 0.483 (0.100) & 9.930 (2.6789) & 42.530 (11.788) \\ 
\textbf{Debiased-SGLasso} &  0.914 (0.183) & 300 (0) & 2400 (0) \\ 
\textbf{CMCP} & 0.507 (0.121) &50.920 (5.134) & 63.020 (6.734) \\  
\textbf{IHT-element} &  0.782 (0.182) & 223.980 (124.509) & 225.240 (125.172)\\ 
\textbf{Lasso-element} & 0.478 (0.093) & 14.270 (4.194) &  17.840 (5.773) \\ \hline
\end{tabular}
}
\footnotesize\textit{Note:} Debiased-SGLasso yields a dense solution and therefore selects all 300 groups and 2,400 variables in every split.
\end{table*}

\section{Conclusion and discussion}\label{sec:con}

In this paper, we propose the specific minimum signal conditions that serve as both the sufficient and necessary conditions for exact recovery in the double sparse model.
From the sufficiency perspective, we demonstrate that a two-stage DSIHT algorithm achieves oracle properties under these minimum signal conditions, ensuring exact recovery and asymptotic normality.
From the necessity perspective, we show through minimax lower bounds based on Hamming risk that no algorithm can achieve exact recovery if these signal conditions are violated.
Our work fills a critical gap in the minimax optimality theory on support recovery of the double sparse model. 
We also establish the minimax rate optimality of the double sparse IHT procedure for exact recovery.
However, some constants in this paper are not exact, meaning that practical implementations of the algorithm do not need to strictly adhere to them.

Additionally, Proposition \ref{T8} and Theorem \ref{T9} reveal an advantage of the IHT-type estimator over convex penalized estimators like the Lasso: 
The IHT-type algorithm permits precise control over estimation error in each iteration, facilitating an in-depth analysis of the relationship between the minimum signal condition and support recovery performance. 
We believe this intuitive non-convex approach could be effective in a broader range of models, including {multi-attribute} graphical models, {sparse plus} low-rank matrix regression, and other general problems under double sparse structure. We leave these meaningful problems for future work.

%% if your bibliography is in bibtex format, uncomment commands:
\bibliographystyle{unsrtnat} % Style BST file (imsart-number.bst or imsart-nameyear.bst)
\bibliography{bibfile.bib}       % Bibliography file (usually '*.bib')

\begin{thebibliography}{49}
\providecommand{\natexlab}[1]{#1}
\providecommand{\url}[1]{\texttt{#1}}
\expandafter\ifx\csname urlstyle\endcsname\relax
  \providecommand{\doi}[1]{doi: #1}\else
  \providecommand{\doi}{doi: \begingroup \urlstyle{rm}\Url}\fi

\bibitem[Cai et~al.(2022)Cai, Zhang, and Zhou]{CZ22}
T~Tony Cai, Anru~R Zhang, and Yuchen Zhou.
\newblock Sparse group lasso: Optimal sample complexity, convergence rate, and
  statistical inference.
\newblock \emph{IEEE Transactions on Information Theory}, 68\penalty0
  (9):\penalty0 5975--6002, 2022.

\bibitem[Zhang et~al.(2015)Zhang, Geng, and Lai]{Zhang15change}
Bingwen Zhang, Jun Geng, and Lifeng Lai.
\newblock Multiple change-points estimation in linear regression models via
  sparse group lasso.
\newblock \emph{IEEE Transactions on Signal Processing}, 63\penalty0
  (9):\penalty0 2209--2224, 2015.
\newblock \doi{10.1109/TSP.2015.2411220}.

\bibitem[Teng and Zhang(2022)]{hao22twt}
Hao~Yang Teng and Zhengjun Zhang.
\newblock Two-way truncated linear regression models with extremely
  thresholding penalization.
\newblock \emph{Journal of the American Statistical Association}, 0\penalty0
  (0):\penalty0 1--17, 2022.
\newblock \doi{10.1080/01621459.2022.2147074}.
\newblock URL \url{https://doi.org/10.1080/01621459.2022.2147074}.

\bibitem[Yang and Zhu(2022)]{rs14163945}
Weixing Yang and Daiyin Zhu.
\newblock Multi-circular sar three-dimensional image formation via group
  sparsity in adjacent sub-apertures.
\newblock \emph{Remote Sensing}, 14\penalty0 (16):\penalty0 3945--3965, 2022.
\newblock ISSN 2072-4292.
\newblock \doi{10.3390/rs14163945}.
\newblock URL \url{https://www.mdpi.com/2072-4292/14/16/3945}.

\bibitem[Abramovich(2024)]{abra2024sgam}
Felix Abramovich.
\newblock {Classification by sparse generalized additive models}.
\newblock \emph{Electronic Journal of Statistics}, 18\penalty0 (1):\penalty0
  2021 -- 2041, 2024.
\newblock \doi{10.1214/24-EJS2246}.
\newblock URL \url{https://doi.org/10.1214/24-EJS2246}.

\bibitem[Li et~al.(2024)Li, Zhang, and Yin]{LZ22}
Zhifan Li, Yanhang Zhang, and Jianxin Yin.
\newblock Estimating double sparse structures over $\ell_u(\ell_q)$-balls:
  Minimax rates and phase transition.
\newblock \emph{IEEE Transactions on Information Theory}, 70\penalty0
  (10):\penalty0 7066--7088, 2024.
\newblock \doi{10.1109/TIT.2024.3451512}.

\bibitem[Zhang et~al.(2024)Zhang, Li, Liu, and Yin]{zhang2024minimax}
Yanhang Zhang, Zhifan Li, Shixiang Liu, and Jianxin Yin.
\newblock A minimax optimal approach to high-dimensional double sparse linear
  regression.
\newblock \emph{Journal of Machine Learning Research}, 25\penalty0
  (369):\penalty0 1--66, 2024.

\bibitem[Li et~al.(2023)Li, Zhang, and Yin]{LZ23}
Zhifan Li, Yanhang Zhang, and Jianxin Yin.
\newblock Sharp minimax optimality of lasso and slope under double sparsity
  assumption.
\newblock \emph{arXiv preprint arXiv:2308.09548}, 2023.

\bibitem[Liang et~al.(2024)Liang, Cohen, Sólon~Heinsfeld, Pestilli, and
  McDonald]{liang2023sparsegl}
Xiaoxuan Liang, Aaron Cohen, Anibal Sólon~Heinsfeld, Franco Pestilli, and
  Daniel~J. McDonald.
\newblock sparsegl: An r package for estimating sparse group lasso.
\newblock \emph{Journal of Statistical Software}, 110\penalty0 (6):\penalty0
  1–23, 2024.
\newblock \doi{10.18637/jss.v110.i06}.
\newblock URL
  \url{https://www.jstatsoft.org/index.php/jss/article/view/v110i06}.

\bibitem[Breheny(2015{\natexlab{a}})]{Breheny15}
Patrick Breheny.
\newblock The group exponential lasso for bi-level variable selection.
\newblock \emph{Biometrics}, 71\penalty0 (3):\penalty0 731--740, 03
  2015{\natexlab{a}}.
\newblock ISSN 0006-341X.
\newblock \doi{10.1111/biom.12300}.
\newblock URL \url{https://doi.org/10.1111/biom.12300}.

\bibitem[Simon et~al.(2013)Simon, Friedman, Hastie, and Tibshirani]{simon13SGL}
Noah Simon, Jerome Friedman, Trevor Hastie, and Robert Tibshirani.
\newblock A sparse-group lasso.
\newblock \emph{Journal of Computational and Graphical Statistics}, 22\penalty0
  (2):\penalty0 231--245, 2013.
\newblock \doi{10.1080/10618600.2012.681250}.
\newblock URL \url{https://doi.org/10.1080/10618600.2012.681250}.

\bibitem[Tibshirani(1996)]{tib96lasso}
Robert Tibshirani.
\newblock Regression shrinkage and selection via the lasso.
\newblock \emph{Journal of the Royal Statistical Society Series B: Statistical
  Methodology}, 58\penalty0 (1):\penalty0 267--288, 1996.

\bibitem[Yuan and Lin(2006)]{yuan2006model}
Ming Yuan and Yi~Lin.
\newblock Model selection and estimation in regression with grouped variables.
\newblock \emph{Journal of the Royal Statistical Society Series B: Statistical
  Methodology}, 68\penalty0 (1):\penalty0 49--67, 2006.

\bibitem[Fan and Li(2001)]{FanLi01}
Jianqing Fan and Runze Li.
\newblock Variable selection via nonconcave penalized likelihood and its oracle
  properties.
\newblock \emph{Journal of the American Statistical Association}, 96\penalty0
  (456):\penalty0 1348--1360, 2001.
\newblock \doi{10.1198/016214501753382273}.
\newblock URL \url{https://doi.org/10.1198/016214501753382273}.

\bibitem[Fan and Peng(2004)]{fan04ncp}
Jianqing Fan and Heng Peng.
\newblock {Nonconcave penalized likelihood with a diverging number of
  parameters}.
\newblock \emph{The Annals of Statistics}, 32\penalty0 (3):\penalty0 928 --
  961, 2004.
\newblock \doi{10.1214/009053604000000256}.
\newblock URL \url{https://doi.org/10.1214/009053604000000256}.

\bibitem[Zou(2006)]{Zou01122006}
Hui Zou.
\newblock The adaptive lasso and its oracle properties.
\newblock \emph{Journal of the American Statistical Association}, 101\penalty0
  (476):\penalty0 1418--1429, 2006.
\newblock \doi{10.1198/016214506000000735}.
\newblock URL \url{https://doi.org/10.1198/016214506000000735}.

\bibitem[Ndaoud(2020)]{NM20}
Mohamed Ndaoud.
\newblock Scaled minimax optimality in high-dimensional linear regression: A
  non-convex algorithmic regularization approach.
\newblock \emph{arXiv preprint arXiv:2008.12236}, 2020.

\bibitem[Wainwright(2007)]{WJM07rec}
Martin Wainwright.
\newblock Information-theoretic bounds on sparsity recovery in the
  high-dimensional and noisy setting.
\newblock In \emph{2007 IEEE International Symposium on Information Theory},
  pages 961--965, 2007.
\newblock \doi{10.1109/ISIT.2007.4557348}.

\bibitem[Wang et~al.(2010)Wang, Wainwright, and Ramchandran]{WW10it}
Wei Wang, Martin~J. Wainwright, and Kannan Ramchandran.
\newblock Information-theoretic limits on sparse signal recovery: Dense versus
  sparse measurement matrices.
\newblock \emph{IEEE Transactions on Information Theory}, 56\penalty0
  (6):\penalty0 2967--2979, 2010.
\newblock \doi{10.1109/TIT.2010.2046199}.

\bibitem[Butucea and Stepanova(2017)]{butucea17}
Cristina Butucea and Natalia Stepanova.
\newblock {Adaptive variable selection in nonparametric sparse additive
  models}.
\newblock \emph{Electronic Journal of Statistics}, 11\penalty0 (1):\penalty0
  2321 -- 2357, 2017.
\newblock \doi{10.1214/17-EJS1275}.
\newblock URL \url{https://doi.org/10.1214/17-EJS1275}.

\bibitem[Butucea et~al.(2018)Butucea, Ndaoud, Stepanova, and Tsybakov]{CM18}
Cristina Butucea, Mohamed Ndaoud, Natalia~A. Stepanova, and Alexandre~B.
  Tsybakov.
\newblock {Variable selection with Hamming loss}.
\newblock \emph{The Annals of Statistics}, 46\penalty0 (5):\penalty0 1837 --
  1875, 2018.
\newblock \doi{10.1214/17-AOS1572}.
\newblock URL \url{https://doi.org/10.1214/17-AOS1572}.

\bibitem[Butucea et~al.(2023)Butucea, Mammen, Ndaoud, and
  Tsybakov]{butucea23group}
Cristina Butucea, Enno Mammen, Mohamed Ndaoud, and Alexandre~B. Tsybakov.
\newblock {Variable selection, monotone likelihood ratio and group sparsity}.
\newblock \emph{The Annals of Statistics}, 51\penalty0 (1):\penalty0 312 --
  333, 2023.
\newblock \doi{10.1214/22-AOS2251}.
\newblock URL \url{https://doi.org/10.1214/22-AOS2251}.

\bibitem[Gao and Stoev(2020)]{Zhengstoev20}
Zheng Gao and Stilian Stoev.
\newblock {Fundamental limits of exact support recovery in high dimensions}.
\newblock \emph{Bernoulli}, 26\penalty0 (4):\penalty0 2605 -- 2638, 2020.
\newblock \doi{10.3150/20-BEJ1197}.
\newblock URL \url{https://doi.org/10.3150/20-BEJ1197}.

\bibitem[Belitser and Nurushev(2022)]{Belitser22UQ}
Eduard Belitser and Nurzhan Nurushev.
\newblock {Uncertainty quantification for robust variable selection and
  multiple testing}.
\newblock \emph{Electronic Journal of Statistics}, 16\penalty0 (2):\penalty0
  5955 -- 5979, 2022.
\newblock \doi{10.1214/22-EJS2088}.
\newblock URL \url{https://doi.org/10.1214/22-EJS2088}.

\bibitem[Abraham et~al.(2024)Abraham, Castillo, and Roquain]{castillo24sharp}
Kweku Abraham, Isma{\"e}l Castillo, and {\'E}tienne Roquain.
\newblock {Sharp multiple testing boundary for sparse sequences}.
\newblock \emph{The Annals of Statistics}, 52\penalty0 (4):\penalty0 1564 --
  1591, 2024.
\newblock \doi{10.1214/24-AOS2404}.
\newblock URL \url{https://doi.org/10.1214/24-AOS2404}.

\bibitem[Ndaoud(2019)]{MN19}
Mohamed Ndaoud.
\newblock Interplay of minimax estimation and minimax support recovery under
  sparsity.
\newblock In Aurélien Garivier and Satyen Kale, editors, \emph{Proceedings of
  the 30th International Conference on Algorithmic Learning Theory}, volume~98
  of \emph{Proceedings of Machine Learning Research}, pages 647--668. PMLR,
  2019.
\newblock URL \url{https://proceedings.mlr.press/v98/ndaoud19a.html}.

\bibitem[Lounici et~al.(2011)Lounici, Pontil, van~de Geer, and
  Tsybakov]{KM11group}
Karim Lounici, Massimiliano Pontil, Sara van~de Geer, and Alexandre~B.
  Tsybakov.
\newblock {Oracle inequalities and optimal inference under group sparsity}.
\newblock \emph{The Annals of Statistics}, 39\penalty0 (4):\penalty0 2164 --
  2204, 2011.
\newblock \doi{10.1214/11-AOS896}.
\newblock URL \url{https://doi.org/10.1214/11-AOS896}.

\bibitem[Cand{\`e}s et~al.(2006)Cand{\`e}s, Romberg, and Tao]{candes2006robust}
Emmanuel~J Cand{\`e}s, Justin Romberg, and Terence Tao.
\newblock Robust uncertainty principles: Exact signal reconstruction from
  highly incomplete frequency information.
\newblock \emph{IEEE Transactions on Information Theory}, 52\penalty0
  (2):\penalty0 489--509, 2006.

\bibitem[Blumensath and Davies(2009)]{BLUMENSATH2009265}
Thomas Blumensath and Mike~E. Davies.
\newblock Iterative hard thresholding for compressed sensing.
\newblock \emph{Applied and Computational Harmonic Analysis}, 27\penalty0
  (3):\penalty0 265--274, 2009.
\newblock ISSN 1063-5203.
\newblock \doi{https://doi.org/10.1016/j.acha.2009.04.002}.
\newblock URL
  \url{https://www.sciencedirect.com/science/article/pii/S1063520309000384}.

\bibitem[Jain et~al.(2014)Jain, Tewari, and Kar]{jain2014iterative}
Prateek Jain, Ambuj Tewari, and Purushottam Kar.
\newblock On iterative hard thresholding methods for high-dimensional
  m-estimation.
\newblock In Z.~Ghahramani, M.~Welling, C.~Cortes, N.~Lawrence, and K.Q.
  Weinberger, editors, \emph{Advances in Neural Information Processing
  Systems}, volume~27. Curran Associates, Inc., 2014.
\newblock URL
  \url{https://proceedings.neurips.cc/paper_files/paper/2014/file/218a0aefd1d1a4be65601cc6ddc1520e-Paper.pdf}.

\bibitem[Yuan et~al.(2018)Yuan, Li, and Zhang]{yuan18grad}
Xiao-Tong Yuan, Ping Li, and Tong Zhang.
\newblock Gradient hard thresholding pursuit.
\newblock \emph{Journal of Machine Learning Research}, 18\penalty0
  (166):\penalty0 1--43, 2018.
\newblock URL \url{http://jmlr.org/papers/v18/14-415.html}.

\bibitem[Liu and Foygel~Barber(2019)]{liubarber19}
Haoyang Liu and Rina Foygel~Barber.
\newblock Between hard and soft thresholding: optimal iterative thresholding
  algorithms.
\newblock \emph{Information and Inference: A Journal of the IMA}, 9\penalty0
  (4):\penalty0 899--933, 12 2019.
\newblock ISSN 2049-8772.
\newblock \doi{10.1093/imaiai/iaz027}.
\newblock URL \url{https://doi.org/10.1093/imaiai/iaz027}.

\bibitem[Zhao and Yu(2006)]{JMLR:v7:zhao06a}
Peng Zhao and Bin Yu.
\newblock On model selection consistency of lasso.
\newblock \emph{Journal of Machine Learning Research}, 7\penalty0
  (90):\penalty0 2541--2563, 2006.
\newblock URL \url{http://jmlr.org/papers/v7/zhao06a.html}.

\bibitem[Huang et~al.(2018)Huang, Jiao, Liu, and Lu]{Huang18}
Jian Huang, Yuling Jiao, Yanyan Liu, and Xiliang Lu.
\newblock A constructive approach to $l_0$ penalized regression.
\newblock \emph{Journal of Machine Learning Research}, 19\penalty0
  (10):\penalty0 1--37, 2018.
\newblock URL \url{http://jmlr.org/papers/v19/17-194.html}.

\bibitem[Huang et~al.(2012)Huang, Breheny, and Ma]{huang2012selective}
Jian Huang, Patrick Breheny, and Shuangge Ma.
\newblock A selective review of group selection in high-dimensional models.
\newblock \emph{Statistical science: a review journal of the Institute of
  Mathematical Statistics}, 27\penalty0 (4):\penalty0 10--1214, 2012.

\bibitem[Breheny(2015{\natexlab{b}})]{Breheny15bio}
Patrick Breheny.
\newblock The group exponential lasso for bi-level variable selection.
\newblock \emph{Biometrics}, 71\penalty0 (3):\penalty0 731--740,
  2015{\natexlab{b}}.
\newblock \doi{https://doi.org/10.1111/biom.12300}.
\newblock URL \url{https://onlinelibrary.wiley.com/doi/abs/10.1111/biom.12300}.

\bibitem[Matthews(1975)]{MCC75}
B.W. Matthews.
\newblock Comparison of the predicted and observed secondary structure of t4
  phage lysozyme.
\newblock \emph{Biochimica et Biophysica Acta (BBA) - Protein Structure},
  405\penalty0 (2):\penalty0 442--451, 1975.
\newblock ISSN 0005-2795.
\newblock \doi{https://doi.org/10.1016/0005-2795(75)90109-9}.
\newblock URL
  \url{https://www.sciencedirect.com/science/article/pii/0005279575901099}.

\bibitem[Wang(2009)]{wang2009}
Hansheng Wang.
\newblock Forward regression for ultra-high dimensional variable screening.
\newblock \emph{Journal of the American Statistical Association}, 104\penalty0
  (488):\penalty0 1512--1524, 2009.
\newblock \doi{10.1198/jasa.2008.tm08516}.

\bibitem[Hsu et~al.(2012)Hsu, Kakade, and Zhang]{HK12}
Daniel Hsu, Sham Kakade, and Tong Zhang.
\newblock {A tail inequality for quadratic forms of subgaussian random
  vectors}.
\newblock \emph{Electronic Communications in Probability}, 17:\penalty0 1 -- 6,
  2012.
\newblock \doi{10.1214/ECP.v17-2079}.
\newblock URL \url{https://doi.org/10.1214/ECP.v17-2079}.

\bibitem[Bickel et~al.(2009)Bickel, Ritov, and
  Tsybakov]{BickelTsy09lassodantzig}
Peter~J. Bickel, Ya’acov Ritov, and Alexandre~B. Tsybakov.
\newblock {Simultaneous analysis of Lasso and Dantzig selector}.
\newblock \emph{The Annals of Statistics}, 37\penalty0 (4):\penalty0 1705 --
  1732, 2009.
\newblock \doi{10.1214/08-AOS620}.
\newblock URL \url{https://doi.org/10.1214/08-AOS620}.

\bibitem[Vershynin(2010)]{vershynin2010introduction}
Roman Vershynin.
\newblock Introduction to the non-asymptotic analysis of random matrices.
\newblock \emph{arXiv preprint arXiv:1011.3027}, 2010.

\bibitem[Raskutti et~al.(2010)Raskutti, Wainwright, and
  Yu]{JMLR:v11:raskutti10a}
Garvesh Raskutti, Martin~J. Wainwright, and Bin Yu.
\newblock Restricted eigenvalue properties for correlated gaussian designs.
\newblock \emph{Journal of Machine Learning Research}, 11\penalty0
  (78):\penalty0 2241--2259, 2010.
\newblock URL \url{http://jmlr.org/papers/v11/raskutti10a.html}.

\bibitem[Fan et~al.(2014)Fan, Xue, and Zou]{fan2014fcp}
Jianqing Fan, Lingzhou Xue, and Hui Zou.
\newblock {Strong oracle optimality of folded concave penalized estimation}.
\newblock \emph{The Annals of Statistics}, 42\penalty0 (3):\penalty0 819 --
  849, 2014.
\newblock \doi{10.1214/13-AOS1198}.
\newblock URL \url{https://doi.org/10.1214/13-AOS1198}.

\bibitem[Zhao et~al.(2022)Zhao, Yang, and He]{zhao2022high}
Peng Zhao, Yun Yang, and Qiao-Chu He.
\newblock High-dimensional linear regression via implicit regularization.
\newblock \emph{Biometrika}, 109\penalty0 (4):\penalty0 1033--1046, 2022.

\bibitem[Ndaoud and Tsybakov(2020)]{Ndaoud2020NCS}
Mohamed Ndaoud and Alexandre~B. Tsybakov.
\newblock Optimal variable selection and adaptive noisy compressed sensing.
\newblock \emph{IEEE Transactions on Information Theory}, 66\penalty0
  (4):\penalty0 2517--2532, 2020.
\newblock \doi{10.1109/TIT.2020.2965738}.

\bibitem[Roy et~al.(2025)Roy, Tewari, and Zhu]{ARW2025Zhu}
Saptarshi Roy, Ambuj Tewari, and Ziwei Zhu.
\newblock {High-dimensional variable selection with heterogeneous signals: A
  precise asymptotic perspective}.
\newblock \emph{Bernoulli}, 31\penalty0 (2):\penalty0 1206 -- 1229, 2025.
\newblock \doi{10.3150/24-BEJ1767}.
\newblock URL \url{https://doi.org/10.3150/24-BEJ1767}.

\bibitem[Rudelson and Vershynin(2013)]{vershynin13HW}
Mark Rudelson and Roman Vershynin.
\newblock {Hanson-Wright inequality and sub-gaussian concentration}.
\newblock \emph{Electronic Communications in Probability}, 18\penalty0
  (none):\penalty0 1 -- 9, 2013.
\newblock \doi{10.1214/ECP.v18-2865}.
\newblock URL \url{https://doi.org/10.1214/ECP.v18-2865}.

\bibitem[Lehmann and Casella(2006)]{LE06}
Erich~L Lehmann and George Casella.
\newblock \emph{Theory of point estimation}.
\newblock Springer Science \& Business Media, 2006.

\bibitem[Mohri et~al.(2018)Mohri, Rostamizadeh, and Talwalkar]{mohri}
Mehryar Mohri, Afshin Rostamizadeh, and Ameet Talwalkar.
\newblock \emph{Foundations of machine learning}.
\newblock MIT press, 2018.

\end{thebibliography}

%%%%%%%%%%%%%%%%%%%%%%%%%%%%%%%%%%%%%%%%%%%%%%
%% Example with multiple Appendixes:        %%
%%%%%%%%%%%%%%%%%%%%%%%%%%%%%%%%%%%%%%%%%%%%%%
\clearpage

\begin{appendix}
 
\section{Proof of the oracle properties}\label{appA}
This appendix focuses on the proofs of the upper bound.
Firstly, we introduce some abbreviations.

\subsection{Some abbreviations and preliminaries} 
Table \ref{sphericcase} introduces some useful abbreviations in the following proof.
For ease of display, recall that we use the double index to locate the entry in $\beta^*$, i.e., use $\beta_{ij}^*$ to represent the $i$-th signal in the $j$-th group $G_j$.

\begin{table*}[htbp]
\caption{Symbols and their meanings.}
\label{sphericcase}
\resizebox{\linewidth}{!}{ \begin{tabular}{@{}cc@{}}
\hline
Symbol   & Meaning \\
\hline
$\mathcal S (as,bs_0) $ &The space consisting of all the support sets of ($as, bs_0$)-sparse vectors. \\  
    $S^*=\left\{(i,j): \beta^*_{ij} \ne 0 \right\}$& The true support set of $\beta^*$.\\  
    $G^*=\left\{j: \beta^*_{G_j} \ne \mathbf 0_{|G_j|} \right\}$& The group index set of the true support groups.\\  
    $S_{G^*} = \bigcup_{j \in G^*} G_j$& The index set of all support groups.\\ 
    $s_j = \| \beta^*_{G_j}\|_0 = |S^* \cap G_j|$ & The sparsity level, i.e., the support number in group $G_j$.  \\  
    $\tilde S^{t}$& The support set of the estimation $\tilde \beta^t$ (in the $t$-th iteration) by Algorithm 2.\\  
    $\tilde S_{OG}^t= (S_{G^*})^c \cap \tilde S^t$& The discovered set in the falsely discovered groups in the $t$-th iteration (by Algorithm 2).\\  
    $\tilde S_{IG}^t= S_{G^*} \cap (S^*)^c \cap \tilde S^t$&The falsely discovered set in the true support groups in the $t$-th iteration (by Algorithm 2).\\  
    $\tilde G_{OG}^{t}$& The group index set of the falsely discovered group in the $t$-th iteration (by Algorithm 2).\\  
    $ \Delta(s,s_0) := \frac1{s_0} \log({{\rm e} m}/{s}) + \log( {{\rm e} d}/{s_0})$ & A useful abbreviation associated with ordinary minimax rate.\\  
    %{\color{blue}$\mu_a =  a\sigma \left( \sqrt{ \frac{6}{n} \Delta(1,s_0) } + \sqrt{ \frac{3}{n} \log(ss_0) } \right),~ a\in\{1,2\}$} & The thresholding parameter associated with Algorithm 2 (exact recovery pursuit).\\  
\hline
\end{tabular} } 
\end{table*}

To simplify the notations, we use the double index $(i,j)$ to locate the $i$-th variable in the $j$-th group $G_j$.
In the beginning, we introduce an essential definition $\tilde H^{t+1} := \tilde \beta^{t } + \frac1nX^\top \left(Y-X\tilde \beta^{t} \right) \in \mathbb R^p$.
It can be decomposed as follows: 
\begin{equation}\label{eq: grad learning}
\begin{aligned}
\tilde H^{t+1} &=  \tilde \beta^{t } + \frac1nX^\top \left(Y-X\tilde \beta^{t } \right) \\
&=  \tilde \beta^{t } + \frac1nX^\top \left(X \tilde \beta^* + \sigma \tilde \xi -X\tilde \beta^{t } \right) \\
&= \tilde \beta^* +  {\Phi} (\tilde\beta^*- \tilde \beta^t) + \tilde \Xi,
\end{aligned}
\end{equation}
where ${\Phi} = \frac1nX^\top X -\mathbf I_p \in \mathbb R^{p \times p}$.
$\tilde \beta^*$ is the oracle estimator, that is, $\tilde \beta_{S^*}^* = (X_{S^*}^\top X_{S^*})^{-1}X_{S^*}^\top Y \in \mathbb R^{|S^*|}$, and $\tilde \beta_{(S^*)^c}^*=\mathbf 0$. 
And $\sigma \tilde \xi = Y - X \tilde \beta^*$.
Define 
$$\tilde \Xi := \frac\sigma n X^\top \left(\mathbf I_n - X_{S^*} (X_{S^*}^\top X_{S^*})^{-1}X_{S^*}^\top \right) \xi \in \mathbb R^p.$$
By sub-Gaussian property of $\xi$, we conclude that $\tilde \Xi_{ij} =0$ holds for all $ (i,j) \in S^* $, and $\mathbf E \exp \left( \lambda \tilde \Xi_{ij} \right) \le \exp(\lambda^2 \sigma^2/(2n) )$ holds for all $\lambda \in \mathbb R$ and $(i,j) \in (S^*)^c$. 
Besides, by $\tilde \beta_{S^*}^* - \beta^*_{S^*}
= \sigma (X_{S^*}^\top X_{S^*})^{-1}X_{S^*}^\top \xi$, we have $ \mathbf E \exp \left( \lambda (\tilde \beta_{ij}^* - \beta^*_{ij}) \right) \le \exp\left(\frac{\lambda^2 \sigma^2}{2 (1-\delta)n} \right)$ holds for all $\lambda \in \mathbb R$ and $ (i,j) \in S^*$.

For ease of display, we denote ${\Phi}_{(k,j)} \in \mathbb R^{1 \times p}$ as the row of ${\Phi}$ corresponding to the $k$-th variable in the $j$-th group, i.e., ${\Phi}_{(k,j)}$ is the $\big((j-1)d+k\big)$-th row vector of ${\Phi} \in \mathbb R^{p \times p}$.
Plus, we show a technique that will be frequently used in the following proof. 
{
Recall $A = A(\kappa, \delta ) := \frac{8 \delta^2}{(\kappa - \delta)^2}$.
For every $\tilde\beta^* - \tilde\beta^{t}$ and $S$ satisfying  $ S \cup \text{supp}(\tilde\beta^* - \tilde\beta^{t}) \in \mathcal S \left( (1+2A)s, \frac{1+4A}{1+2A} s_0 \right)$, based on DSRIP$\left( (1+2A)s, \frac{1+4A}{1+2A} s_0, \delta \right)$ condition (recall Definition \ref{df2}), we have}
\begin{equation}\label{tech}
\begin{aligned}
\sum_{(k,j) \in S } \langle{\Phi}_{(k,j)}^\top, \tilde\beta^* - \tilde\beta^{t} \rangle^2
&\le  \sum_{(k,j) \in S'} \langle{\Phi}_{(k,j)}^\top, \tilde\beta^* - \tilde\beta^{t} \rangle^2 \\
&= (\beta^* - \tilde\beta^{t})_{S'}^\top {\Phi}_{S'S'}^\top {\Phi}_{S'S'}(\tilde\beta^* - \tilde\beta^{t})_{S'} \\
&\le \left\| {\Phi}_{S'S'} \right\|_2^2 \cdot \left\|  \tilde\beta^* - \tilde\beta^{t}\right\|_2^2 \\
&\le \delta^2 \left\| \tilde \beta^* - \tilde\beta^{t}\right\|_2^2,
\end{aligned}
\end{equation}
where $S' = S \cup \text{supp}(\tilde\beta^* - \tilde\beta^{t})$.

{In the following subsections, we first prove Proposition \ref{T8}, Theorem \ref{T8.5}, and Theorem \ref{T9}, and then give a brief proof of Theorem \ref{upperbound} because it can be seen as a simpler version of Proposition \ref{T8}.
The final subsection \ref{sec: random} provides some discussion on the random setting.}

\subsection{Proof of Proposition \ref{T8}}\label{T4proof}
In the beginning, we show the relationship between Proposition \ref{T8}, Theorem \ref{upperbound}, and some lemmas in Figure \ref{roadmap2}.

\begin{figure}[t]
\centering
\begin{tikzpicture}[scale = 1]
\node [draw,rectangle](l)at(-0.5,6){ Theorem \ref{T8.5} (properties of $\tilde \beta^t$ for sufficiently large $t$)};
	\node [draw,rectangle](a)at(-0.5,4){ Proposition \ref{T8} (properties of $\tilde \beta^t$)};
	\node [draw,rectangle](b)at(-5,2){ Theorem \ref{upperbound} (properties of $\tilde \beta^0$) };
	\node [draw,rectangle](e)at(-0.5,2){ Lemma \ref{good event} };
	\node [draw,rectangle](f)at(2.5,2){ Lemma \ref{exactOGChi2} };
	\node [draw,rectangle](g)at(5.5,2){ Lemma \ref{exactsupporterror} };
	\node [draw,rectangle](i)at(-0.5,0){ Lemma \ref{exactOGChi2:1st} };
        \node [draw,rectangle](j)at(2.5,0){ Lemma \ref{OGChi2:2nd} };
        \node [draw,rectangle](k)at(5.5,0){ Lemma \ref{OGChi2:3rd} };

	%\draw [->](node cs:name=e,angle=90)|-(n);
        \draw [-{Stealth[length=3mm]}](a)--(l);
        \draw [-{Stealth[length=3mm]}](e)--(a);
        \draw [-{Stealth[length=3mm]}](b)--(a);
        \draw [-{Stealth[length=3mm]}](e)--(b);
        \draw [-{Stealth[length=3mm]}](f)--(a);
        \draw [-{Stealth[length=3mm]}](g)--(a);
        \draw [-{Stealth[length=3mm]}](e)--(i);
        \draw [-{Stealth[length=3mm]}](i)--(f);
        \draw [-{Stealth[length=3mm]}](j)--(f);
        \draw [-{Stealth[length=3mm]}](k)--(f);
\end{tikzpicture}
\caption{A road map to complete the proof of Proposition \ref{T8} and Theorem \ref{T8.5}. The directed edges mean that the tails of the edges are used in the heads of the edges. }\label{roadmap2}
\end{figure}

\paragraph{Preliminary}
Define
\begin{equation}\label{eq: def mu}
\begin{aligned}
    \mu :=\max\left\{ \frac{\kappa C_\lambda}{\delta}, ~
    \sqrt{40+ \frac{120\delta^2}{(1-\delta)^2}}\right\}\cdot \sqrt{\frac{ \sigma^2}n \left\{ \frac{\log(em)}{s_0} + \log(esd) \right\} } ,
\end{aligned}
\end{equation}
where $C_\lambda$ and $\lambda_{(\infty)}$ are related to the parameter setting in Theorem \ref{upperbound}.
The definition of $\mu$ will be used in \eqref{exactIG}, Lemma \ref{exactOGChi2}, and Lemma \ref{exactOGChi2:1st}.

First, we show the probability inequality used in this proof as follows.
\begin{equation}\label{eq:exactprob}
\begin{aligned}
& \mathbf P\left\{\;
\begin{aligned}
& \sup_{S \in\mathcal S(1,s_0)}\left\| \tilde\Xi_S \right\|_2^2 
\le \frac{10\sigma^2 s_0 \Delta(1,s_0)}{n},\\
&\|\tilde\beta^*-\beta^*\|_2^2 
\le \frac{5 \sigma^2 s s_0 \Delta(s,s_0)}{n(1-\delta) },\\
&\max_{(i,j)\in S^*}|\tilde\beta_{ij}^*-\beta_{ij}^*|
\le \frac{\mu}3,\\
&\max_{(i,j)\in S_{G^*}}|\tilde\Xi_{ij}|
\le \underbrace{\sqrt{\frac{10\sigma^2}{n}\left[ \Delta(1,s_0) + \log(ss_0)\right]}}_{=: \mu'}
\end{aligned}
\;\right\}\\
&\quad\quad \ge 1-O\left( e^{-\frac13 \left[ \Delta(1,s_0)+ \log(ss_0) \right] } \right),
\end{aligned}
\end{equation}
\begin{comment}
    \begin{equation}\label{eq:exactprob}
\begin{aligned}
P\left\{ \forall S \in \mathcal S(1,s_0),~\sum_{(i,j) \in S}\tilde \Xi_{ij}^2 
  \le \frac{6\sigma^2 s_0 \Delta(1,s_0)}{n} \right\} &\ge 1- e^{ -\frac18 s_0 \Delta(1,s_0) };\\
P\left\{ \|\tilde\beta^* - \beta^*\|_2^2 = \sum_{(i,j) \in S^*}\left( \tilde\beta_{ij}^* - \beta_{ij}^* \right)^2 \le \frac{6\sigma^2 ss_0}{n(1-\delta)^2} \right\} &\ge 1- \exp\left( -ss_0 \right);\\
P\left\{ \max_{(i,j) \in S^*} |\tilde\beta_{ij}^* - \beta^*_{ij}| 
  \le \sqrt{\frac{3\sigma^2 \log(ss_0)}{n(1-\delta)}}\right\} &\ge 1- 2 (ss_0)^{-1/2}; \\
P\left\{ \max_{(i,j) \in S_{G^*}} |\tilde\Xi_{ij}| 
  \le \sqrt{\frac{3\sigma^2 \log(sd)}{n }}\right\} & \ge 1- 2 (sd)^{-1/2}, \\
\end{aligned}
\end{equation}
\end{comment}
where the first inequality follows from Lemma \ref{good event}, with taking $s'=1, s_0' = s_0$, the secong inequality follows from a similar result as in \eqref{ineq:t}, with $\left\| X_{S^*}^\top X_{S^*}\right\|_2 \le \frac{1}{n(1-\delta)}$ and taking $t = ss_0 \Delta(s,s_0)$.

{
The third and fourth inequalities follow from the sub-Gaussian property, for example, since
\begin{align*}
    \frac\mu3 \ge& \frac13 \sqrt{4+ \frac{12\delta^2}{(1-\delta)^2}}\cdot \sqrt{\frac{10\sigma^2 \left[ \Delta(1,s_0)+ \log(ss_0) \right]}{n} }\\
    \ge & \frac{\sqrt{30}\sigma}{3(1-\delta)} \sqrt{\frac{\left[ \Delta(1,s_0)+ \log(ss_0) \right]}n  },
\end{align*}
we have 
\begin{align*}
P\left( \max_{(i,j)\in S^*}|\tilde\beta_{ij}^*-\beta_{ij}^*|
\ge \frac{\mu}3 \right)
= & P\left( \bigcup_{(i,j) \in S^*} \left\{ |\tilde\beta_{ij}^*-\beta_{ij}^*| \ge  \frac{\mu}3 \right\}\right)\\
\le & 2ss_0 \cdot  \exp\left(-\frac{5 \left[ \Delta(1,s_0)+ \log(ss_0) \right] }{3(1-\delta)} \right)\\
\le & 2\exp\left\{-\frac13 \left[ \Delta(1,s_0)+ \log(ss_0) \right] \right\} .
\end{align*}
}
Additionally, 
\begin{equation}\label{eq: inside support group}
P\left( \max_{(i,j)\in S_{G^*}}|\tilde\Xi_{ij}| \ge \mu'  \right) 
\le  2sd \cdot  \exp\left(- \frac{n{\mu'}^2}{2\sigma^2} \right)
= 2 \exp\left(- \frac{n {\mu'}^2}{2\sigma^2} + \log(sd)\right).
\end{equation}

{
\paragraph{Mathematical induction}
With $A =\frac{8 \delta^2}{(\kappa - \delta)^2}$, we next prove the following results hold for all $t \ge 0$: 
\begin{equation}\label{eq: 2.2dsresult}
\begin{aligned}
    \|\tilde \beta^t - \tilde \beta^* \|_2 \le&~\left[ \sqrt{5/6} + \left(1-\sqrt{5/6} \right) \delta \right]^t \|\tilde \beta^0 - \tilde \beta^* \|_2\quad,\\
\tilde S_{OG}^{t} := (S_{G^*})^c 
 \cap \tilde S^{t} \in& ~\mathcal S (As,s_0) \quad,\\
\tilde S_{IG}^{t} := S_{G^*} \cap \tilde S^{t } \cap (S^*)^c \in& ~\mathcal S (s, A s_0) \quad,
\end{aligned}
\end{equation}
Under the event mentioned in \eqref{eq:exactprob}, we prove \eqref{eq: 2.2dsresult} by using mathematical induction. %$\| \tilde \beta^{t} -\tilde \beta^* \|_2 \le 22\left(\frac34 \right)^t \sqrt{ \frac{ \sigma^2 ss_0 \Delta(s,s_0)}{n} } $ holds for all $t \ge 0$.
We first check \eqref{eq: 2.2dsresult} in $t=0$: 
It is straightforward that the first results hold, and by Theorem \ref{upperbound}, the initial estimator $\tilde \beta^0$ satisfies $\tilde S_{OG}^{t}  \in \mathcal S (As,s_0) $ and $\tilde S_{IG}^{t} \in \mathcal S (s, A s_0)$ with a probability greater than $1-O\left( e^{-\frac13 \left[ \Delta(1,s_0)+ \log(ss_0) \right] } \right)$.
}

\begin{comment}
Then, with probability greater than $1-\exp(-ss_0)$, we have 
\begin{equation*}
\begin{aligned}
    \| \tilde \beta^0 -\tilde \beta^* \|_2
\le& \| \tilde \beta^0 - \beta^* \|_2 + \| \beta^*-\tilde \beta^* \|_2\\
\le& \sqrt{ \frac{315\sigma^2 ss_0 \Delta(s,s_0)}{n} } + \sqrt{\frac{6\sigma^2 ss_0}{n(1-\delta)^2} }
< 22\sqrt{ \frac{ \sigma^2 ss_0 \Delta(s,s_0)}{n} } ,
\end{aligned}
\end{equation*}
where the last inequality follows from $\delta < 1/4$. 
\end{comment}

Then, we assume that \eqref{eq: 2.2dsresult} holds in the $t$-th iteration ($t\ge 0$), and we need to prove that \eqref{eq: 2.2dsresult} still holds in the $(t+1)$-th iteration.
Next two steps prove $ \tilde S_{OG}^{t+1} \in \mathcal S(As,s_0) $ and $ \tilde S_{IG}^{t+1}  \in \mathcal S (s,As_0 )$ in $(t+1)$-th iteration. 
The way is to prove by contradiction.

\textbf{Step 1 (Control the shape of $ \tilde S_{OG}^{t+1}$ ).} 
By using contradiction, at first, we assume the opposite, that is, $ \tilde S_{OG}^{t+1} $ contains more than $As$ groups (i.e., these groups are not true support groups but the coefficients corresponding to them are not estimated as zero in the $(t+1)$-th iteration). 
Then, we can select arbitrary $As$ falsely discovered groups and construct an index set $S_{OG}' \in \mathcal S (As,s_0)$. The details of the construction process are described as follows:

For any falsely discovered group $G_j$ (i.e., $j \notin G^*$), if $\|\tilde \beta_{G_j}^{t+1}  \|_0 \ge s_0$, then choose arbitrarily $s_0$ non-zero entries from $\tilde \beta_{G_j}^{t+1}$ into $S_{OG}'$;
if $\|\tilde \beta_{G_j}^{t+1}  \|_0 < s_0$, then choose all non-zero entries from $\tilde \beta_{G_j}^{t+1}$ into $S_{OG}'$.
Repeat this operation $As$ times for any $As$ falsely discovered groups, and we obtain a $(As, s_0)$-sparse set $S_{OG}'$, i.e., $S_{OG}' \in \mathcal S(As,s_0)$.

Then, based on the element-wise thresholding operator $\mathcal T_{\mu }^{(1)}$ and the group-wise thresholding operator $\mathcal T_{\mu,s_0}^{(2)}$, for any group $G_j$ selected by $S_{OG}'$, it holds that 
$$\| \tilde \beta_{G_j \cap S_{OG}'}^{t+1} \|_2^2 
= \left\| \mathcal T_{\mu,s_0}\left(\tilde H^{t+1} \right)_{G_j \cap S_{OG}'} \right\|_2^2 
\ge s_0 \mu^2.$$
{Based on the decomposition $\tilde H^{t+1}_{ij}
 =  \langle{\Phi}_{(i,j)}^\top,\tilde \beta^*- \tilde \beta^{t}\rangle +  \Xi_{ij}$ (holds for all $ (i,j) \in (S^*)^c$) and triangle inequality, we derive that
\begin{equation}\label{exactOG}
    \begin{aligned}
    \sqrt{ Ass_0 }\mu
    \le& \sqrt{\sum_{(i,j) \in S_{OG}'} \langle{\Phi}_{(i,j)}^\top, \tilde\beta^* - \tilde\beta^{t} \rangle^2 } 
        + \sqrt{\sum_{(i,j) \in S_{OG}'}  \tilde\Xi_{ij}^2 \mathbf1\Big\{ \mathcal T_{\mu,s_0} 
    \big(\tilde H^{t+1}\big)_{ij}\ne 0\Big\} }\\
    \overset{(i)}{\le} & \delta \|\tilde \beta^t - \tilde\beta^* \|_2 +  \frac{1-\delta}{\sqrt3}  \|\tilde \beta^t - \tilde\beta^* \|_2\\
    < & \|\tilde \beta^t - \tilde\beta^* \|_2 \\
     \overset{(ii)}\le& \|\tilde \beta^0- \beta^* \|_2 + \| \beta^* - \tilde\beta^* \|_2\\
  \overset{(iii)}{\le}&\left( 1- \frac{\sqrt{10}}{C_\lambda}\right) \frac{2\sqrt2 \kappa}{\kappa-\delta} \cdot \sqrt{ss_0} \lambda_{(\infty)}
     + \frac{\kappa }{\kappa - \delta} \sqrt{5ss_0} \frac{\lambda_{(\infty)}}{C_\lambda}\\
  <& \sqrt{A ss_0 }\mu,
\end{aligned}
\end{equation}
where inequality (i) follows from \eqref{tech} and Lemma \ref{exactOGChi2}, inequality (ii) follows from the first result in \eqref{eq: 2.2dsresult} in the $t$-th iteration (held by the assumption of induction).
Inequality (iii) follows from Theorem \ref{upperbound} and the second inequality in \eqref{eq:exactprob}, and the final inequality follows the definition of $\mu$, as shown in \eqref{eq: def mu}.
Therefore, we find an absurdity, demonstrating that only fewer than $As$ groups can be falsely discovered in the $(t+1)$-th iteration. 
}

Next, we prove that fewer than $Ass_0$ elements will be falsely discovered outside true groups $G^*$. %(indexed by $\tilde G_{OG}^{t+1}$). 
If not so, we can construct set $S_{OG}'' \in \mathcal S(As,s_0)$, which satisfies that $S_{OG}'' \subseteq \tilde S_{OG}^{t+1}$, $|S_{OG}''|=Ass_0$ and each $\tilde \beta_{ij}^{t+1}$ indexed from $S_{OG}''$ is falsely discovered.
The construction is similar to $S_{OG}'$: 
For any falsely discovered group $G_j$,
if $\|\tilde \beta_{G_j}^{t+1}  \|_0 < s_0$, then choose all non-zero entries from $\tilde \beta_{G_j}^{t+1}$ into $S_{OG}''$.
If $\|\tilde \beta_{G_j}^{t+1}  \|_0 \ge s_0$, then choose at least $s_0$ of its nonzero entries in the set $S_{OG}''$---crucially, we carefully calibrate the exact number of selections across all such groups so that the total cardinality
$\bigl|S_{OG}''\bigr|= A ss_0$.
Therefore, based on the element-wise thresholding operator and Lemma \ref{exactOGChi2}, we have 
\begin{equation}\label{step1.5}
    \begin{aligned}
    \sqrt{ Ass_0 }\mu
    \le& \sqrt{\sum_{(i,j) \in S_{OG}''} \langle{\Phi}_{(i,j)}^\top, \tilde \beta^* - \tilde\beta^{t} \rangle^2 } 
        + \sqrt{\sum_{(i,j) \in S_{OG}''} \Xi_{ij}^2 \mathbf1\Big\{ \mathcal T_{\mu,s_0} 
    \big(\tilde H^{t+1}\big)_{ij}\ne 0\Big\} }\\
    {\le} & \|\tilde \beta^t - \tilde \beta^* \|_2 ~ < \sqrt{Ass_0} \mu.
    \end{aligned}
\end{equation}
Then, we similarly find an absurdity as \eqref{exactOG} again.
Therefore, we prove that $ \tilde S_{OG}^{t+1} = S_{G^*}^c \cap \tilde S^{t+1} \in \mathcal S(As,s_0) $.

\textbf{Step 2 (Control the shape of $ \tilde S_{IG}^{t+1} $ ).} 

By using contradiction, at first, we assume the opposite, that is, more than $Ass_0$ entries in $S_{G^*}$ are falsely discovered. 
Then we can construct an index set $S_{IG}' \subseteq \tilde S_{IG}^{t+1}$ satisfying $|S_{IG}'| = Ass_0$ and $S_{IG}' \in \mathcal S(s, As_0)$, with $|\mathcal T_{\mu,s_0} (\tilde H^{t+1})_{ij}| \ge \mu$ for all $ (i,j) \in S_{IG}'$. 
Therefore, based on the element-wise thresholding operator, we obtain 
\begin{equation}\label{exactIG}
   \begin{aligned}
    \sqrt{Ass_0}\mu 
    \le & \sqrt{\sum_{(i,j) \in S_{IG}'} \langle{\Phi}_{(i,j)}^\top, \tilde\beta^* - \tilde\beta^{t} \rangle^2 } \\
        &+\sqrt{\sum_{(i,j)\in S_{IG}'}\tilde\Xi_{ij}^2 \mathbf1\Big\{ |\tilde\Xi_{ij}+\langle{\Phi}_{(i,j)}^\top, \tilde\beta^* - \tilde\beta^{t} \rangle|\ge \mu \Big\} }\\
    \le & \delta \|\tilde \beta^t - \tilde\beta^* \|_2
        +\sqrt{\sum_{(i,j)\in S_{IG}'}\tilde\Xi_{ij}^2 \mathbf1\Big\{ |\tilde\Xi_{ij}|\ge \mu'\Big\} }\\ 
        &+\sqrt{\sum_{(i,j)\in S_{IG}'}\tilde\Xi_{ij}^2 \mathbf1\Big\{ |\tilde\Xi_{ij}|<\mu',~|\tilde\Xi_{ij}+ \langle{\Phi}_{(i,j)}^\top, \tilde\beta^* - \tilde\beta^{t} \rangle |\ge \mu \Big\} }\\
    \overset{(i)}{\le}&  \delta \|\tilde \beta^t - \tilde\beta^* \|_2 +\sqrt{\sum_{(i,j)\in S_{IG}'}\tilde\Xi_{ij}^2 \mathbf1\Big\{ |\tilde\Xi_{ij}|<\mu' \le \frac{1-\delta}{\delta \sqrt6}|\langle{\Phi}_{(i,j)}^\top, \tilde\beta^* - \tilde\beta^{t} \rangle|\Big\} }\\
    \le &  \|\tilde \beta^t - \tilde\beta^* \|_2 ,
\end{aligned}
\end{equation}
where recall $ \mu' = \sqrt{\frac{10\sigma^2}{n}\left[ \Delta(1,s_0) + \log(ss_0)\right]}$ and $\mu \ge \sqrt{4+ \frac{12\delta^2}{(1-\delta)^2}}\cdot\mu'$, and inequality (i) follows from the forth inequality in \eqref{eq:exactprob}. 
Then, by the last three lines in \eqref{exactOG}, we get a contradiction again, demonstrating $\tilde S_{IG}^{t+1} = S_{G^*} \cap \tilde S^{t+1} \cap (S^*)^c \in \mathcal S (s, A s_0)$.

\textbf{Step 3 ($\ell_2$ error inequality of $\tilde \beta^{t+1}$).} 
Now we prove that the $\ell_2$ error bound in the $(t+1)$-th iteration.
Note that
\begin{equation}
\begin{aligned}
    &\tilde \beta_{ij}^{t+1} - \tilde \beta_{ij}^* \\
    =& - (\underbrace{\tilde\beta_{ij}^* + \langle{\Phi}_{(i,j)}^\top, \tilde\beta^* - \tilde\beta^{t} \rangle + \tilde\Xi_{ij}}_{\tilde H_{ij}^{t+1}}) \cdot \mathbf1 \left( (i,j) \notin \tilde S^{t+1} \right) + \langle{\Phi}_{(i,j)}^\top, \tilde\beta^* - \tilde\beta^{t} \rangle + \tilde\Xi_{ij},
\end{aligned}
\end{equation}
where recall $ \tilde\Xi_{ij} =0$ for all $ (i,j)\in S^*$.
{
We then have
\begin{equation} \label{eq: 2.2 l21}
    \begin{aligned}
& \left\|\tilde \beta ^{t+1} - \tilde\beta ^* \right\|_2 \\
\le& \left[ \sum_{(i,j) \in S^*} \left\{\langle\Phi_{(i,j)}^\top, \tilde\beta^* -\tilde\beta^{t} \rangle - \tilde H_{ij}^{t+1} \cdot \mathbf1 \left( (i,j) \notin \tilde S^{t+1} \right) \right\}^2 \right. \\
&\quad + \left.\sum_{(i,j) \in \tilde S^{t+1}\setminus S^*} 
   \left\{\langle\Phi_{(i,j)}^\top,\tilde \beta^* - \tilde \beta^{t} \rangle+ \tilde\Xi_{ij} \right\}^2 \right]^{1/2}\\
\overset{(i)}\le& \sqrt{ \sum_{(i,j) \in \tilde S^{t+1}\cup S^*} \langle\Phi_{(i,j)}^\top, \tilde \beta^* - \tilde \beta^{t} \rangle^2}\\
& + \left[ \sum_{(i,j) \in  \tilde S_{OG}^{t+1}} \tilde \Xi_{ij}^2 \mathbf1\Big\{ \mathcal T_{\mu,s_0} \big(\tilde H^{t+1}\big)_{ij}\ne 0\Big\}  
   + \sum_{(i,j) \in  \tilde S_{IG}^{t+1}} \tilde \Xi_{ij}^2 \mathbf1\Big\{ \mathcal T_{\mu,s_0} \big(\tilde H^{t+1}\big)_{ij}\ne 0\Big\}\right.\\
& \quad \quad \left. + \sum_{(i,j) \in S^*} \left\{ \tilde H_{ij}^{t+1} \right)^2 \mathbf1 \left( (i,j) \notin \tilde S^{t+1} \right\} \right]^{1/2}\\
\overset{(ii)}\le & \left( 1+ \sqrt{\frac56} \frac{1-\delta}{\delta}\right) \delta \left\| \tilde \beta^{t}- \tilde \beta^* \right\|_2 \\
\le &\left[ \sqrt{5/6} + \left(1-\sqrt{5/6} \right) \delta \right]^{t+1} \left\|\tilde \beta^0 - \tilde \beta^* \right\|_2,
\end{aligned}
\end{equation}
where inequality (i) follows from H\"{o}lder inequality, and inequality (ii) follows from Lemma \ref{exactOGChi2}, \eqref{exactIG}, and Lemma \ref{exactsupporterror}.
The last inequality follows from the first result in \eqref{eq: 2.2dsresult} in the $t$-th iteration, which holds by the assumption of induction.}

Therefore, we prove that all three results in \eqref{eq: 2.2dsresult} still hold in the $(t+1)$-th iteration.
This completes the proof of Proposition \ref{T8}.

\subsection{Proof of Theorem \ref{T8.5}}

We follow the results derived in Proposition \ref{T8}.
By the sample size assumption and inequalities (ii), (iii) in \eqref{exactOG}, we have $\left\| \tilde \beta^0 - \tilde \beta^* \right\|_2 < \sqrt{Ass_0} \mu \le C$ for a sufficient large absolute constant $C>0$, therefore
\begin{equation}\label{eq: decay rate}
\left\|\tilde \beta^{t} - \tilde \beta^* \right\|_2 \le\left[ \sqrt{5/6} + \left(1-\sqrt{5/6} \right) \delta \right]^t\cdot  C\sigma, \quad \text{for every } t \ge 0.
\end{equation}
Then, let $t >  C_{\delta}\log n$, and we conclude
\begin{equation}\label{eq: error in Th2.2}
  \|\tilde \beta^{t} - \tilde \beta^*  \|_2 < \frac{ \sigma }n,
\end{equation}
which leads that 
{
\begin{equation}\label{eq:exactfinal}
\left\|\tilde \beta^t -\beta^* \right\|_\infty
\le \left\|\tilde \beta^t -\tilde \beta^* \right\|_\infty + \left\|\tilde \beta^* - \beta^* \right\|_\infty
%\le \left\|\tilde \beta^t -\tilde \beta^* \right\|_2 + 2 \sqrt{\frac{ \sigma^2 \log(ss_0)}n}
< \mu/2,
\end{equation}
where the last inequality follows from the third inequality in \eqref{eq:exactprob} and $\mu > \frac{6 \sigma}{\sqrt n}$.}

Then, we conclude that:
\begin{enumerate} 
    \item For every $ (i,j) \in (S^*)^c$, by \eqref{eq:exactfinal}, we have $|\tilde \beta^t_{ij} | <\mu/2 < \mu $, which proves that we can not falsely discover any entry by using thresholding parameter $\mu$ (just recall the element-wise operator $\mathcal T_{\mu}^{(1)}$), therefore $\tilde \beta^t_{ij} =0$.

    \item For every $ (i,j) \in S^*$, by \eqref{eq:exactfinal} and the element-wise beta-min condition, we derive that
    \begin{equation*}
        | \tilde \beta_{ij}^t| \ge | \beta_{ij}^*| - \left\|\tilde \beta^t -\beta^* \right\|_\infty
        > 3\mu/2 >0,
    \end{equation*}
    which proves that the whole support set $S^*$ is recovered with sign consistency.
\end{enumerate}

Therefore, we prove that with probability greater than $ 1-O\left( e^{-\frac13 \left[ \Delta(1,s_0)+ \log(ss_0) \right] } \right)$, the true support set $S^*$ can be exactly recovered without any falsely discovered entry or group for all $ t > C\log n$.

{
We now show a sharper estimation rate by using Theorem 2.1 in \cite{HK12}, that is, for all $t >0$, we have 
$$
\begin{aligned}
    P\Bigg(  &\left\| (X_{S^*}^\top X_{S^*})^{-1}X_{S^*}^\top \xi  \right\|_2^2 \\
    &\quad\ge \text{tr}((X_{S^*}^\top X_{S^*})^{-1} )+
     2\sqrt{\text{tr}\left ((X_{S^*}^\top X_{S^*})^{-1} (X_{S^*}^\top X_{S^*})^{-1}  \right)t} \\
     &\qquad+ 2 \left\| (X_{S^*}^\top X_{S^*})^{-1} \right\|_2 t \Bigg)
     \le e^{-t}.
\end{aligned}
$$
Based on DSRIP$(s,s_0, \delta)$ condition, we have $\Lambda_i\left ((X_{S^*}^\top X_{S^*})^{-1} \right)\le \frac1{n(1-\delta)}$,
which leads
\begin{align*}
\text{tr}((X_{S^*}^\top X_{S^*})^{-1} ) &\le \frac{ss_0}{n(1-\delta)},\\
\text{tr}\left ((X_{S^*}^\top X_{S^*})^{-1} (X_{S^*}^\top X_{S^*})^{-1}  \right) &\le  \frac{ss_0}{n^2(1-\delta)^2}, \\
\left\| (X_{S^*}^\top X_{S^*})^{-1} \right\|_2 & \le  \frac{1}{n(1-\delta)}.
\end{align*}
Therefore, for every $\epsilon \in (0,1)$, by taking $t = \log(1/\epsilon)$, we have 
$$
  P\Bigg( \frac1{\sigma^2} \left\|\tilde \beta^* - \beta^* \right\|_2^2 \le \frac{2ss_0}{n(1-\delta)} + \frac{3\log(1/\epsilon)}{n(1-\delta)}  \Bigg)
     \ge 1-\epsilon,
$$
combining with \eqref{eq: error in Th2.2} we complete the proof of Theorem \ref{T8.5}.
}
%\end{proof}

\subsection{Proof of Theorem \ref{T9}}
%\begin{proof}[Proof of Theorem \ref{T9}]
By Cram\'er--Wold Theorem, we only need to prove the asymptotic normality for $\sqrt n g^\top (\tilde\beta^t_{S^*} - \beta^*_{S^*})$, where $g \in \mathbb R^{|S^*| \times 1}$ is an arbitrary vector with bounded Euclid norm.
By taking  $t >  C_\delta \log n$, the output $\tilde\beta^t$ of Algorithm \ref{scaledIHT} leads that
\begin{align*}
    \left|\sqrt n g^\top (\tilde\beta^t_{S^*} - \beta^*_{S^*}) 
-\sqrt n  g^\top (\tilde\beta^*_{S^*} - \beta^*_{S^*}) \right|
%=\sqrt n \left|g^\top (\tilde\beta^t_{S^*} - \tilde\beta^*_{S^*}) \right|
\le&\sqrt n \|g\|_2 \cdot  \left\| \tilde\beta^t_{S^*} - \tilde\beta^*_{S^*} \right\|_2\\
\le& \|g\|_2 \frac{\sigma}{\sqrt n}
\to 0, 
\end{align*}
as $n,~ p(=d\times m) \to \infty$, where the last inequality follows from Step 4 in the proof of Theorem \ref{T8.5}.
Therefore we only need to focus on the term $\sqrt ng^\top (\tilde\beta^*_{S^*} - \beta^*_{S^*})$. By definition, we get 
\begin{equation} 
\sqrt n g^\top (\tilde\beta^*_{S^*} - \beta^*_{S^*})
= \sum_{k=1}^n  \sigma\sqrt n \cdot g^\top (X_{S^*}^\top X_{S^*})^{-1} X_{S^*}^{(k)} \xi_k
:= \sum_{k=1}^n \zeta_k,
\end{equation}
where we denote by $ X_{S^*}^{(k)} \in \mathbb R^{|S^*| \times 1}$ the $k$-th observation of the covariates on the support $S^*$.

Additionally, by DSRIP condition, we have 
\begin{equation}
    \sum_{k=1}^n\text{Var} \left (  \zeta_k \right ) 
    = \sigma^2 nc_\xi g^\top (X_{S^*}^\top X_{S^*})^{-1} g
\ge \frac{\sigma^2 c_\xi \|g\|_2^2}{ 1+\delta },
\end{equation}
where recall we define $c_\xi := \text{Var}(\xi_k)$. 
We also have 
{
\begin{equation}
\begin{aligned}
 \sum_{k=1}^n \mathbf E\left( |\zeta_k|^3 \right)
=& \sigma^3 n^{3/2}\sum_{k=1}^n \left| g^\top (X_{S^*}^\top X_{S^*})^{-1} X_{S^*}^{(k)}\right|^3 
   \mathbf E\left( |\xi_k|^3 \right)\\
\overset{(i)}\le& C_1 \sigma^3 n^{3/2} \| g\|_2^3  \left\| (X_{S^*}^\top X_{S^*})^{-1}\right\|_2^3
   \cdot\sum_{k=1}^n\left\|X_{S^*}^{(k)}\right\|_2^3 \\
\overset{(ii)}\le& \frac{C_1 \sigma^3 \| g\|_2^3 }{n^{3/2} (1-\delta)^3} 
   \cdot n B_{S^*}^3\\ 
= & \frac{C_1 \sigma^3 \|g\|_2^3 B_{S^*}^3 }{n^{1/2} (1-\delta)^3} ,
\end{aligned}
\end{equation}
where inequality (i) follows from the subGaussian property, i.e.,
$$
\mathbf E |\xi_k|^3 = \int_{t\ge0} \mathbf P\left( |\xi_k|^3 \ge t \right)\mathrm d t
\le \int_{t\ge 0} 6t^2 e^{-t^2/2}\mathrm d t 
= 3 \sqrt{2 \pi} =: C_1,
$$
and inequality (ii) follows from the assumption $B_{S^*} = \max_{i \in [n]} \| X_{S^*}^{(i)} \|_2$ in Theorem \ref{T9}.

Therefore, we have 
\begin{equation}
\begin{aligned}
\frac{1}{\left ( \sum_{k=1}^n\text{Var} \left( \zeta_k \right) \right )^{3/2} }
\sum_{k=1}^n \mathbf E\left( |\zeta_k|^3 \right)
\le & \frac{C_1 \sigma^3 \| g\|_2^3 B_{S^*}^3 (1+\delta)^{3/2}}
{n^{1/2} (1-\delta)^3 \sigma^3 c_\xi^{3/2} \|g\|_2^3  }\\
=& \frac{C_1 (1+\delta)^{3/2} }
{c_\xi^{3/2} (1-\delta)^3 } \cdot \frac{B_{S^*}^3}{\sqrt n}
\to 0,
\end{aligned}
\end{equation}
as $B_{S^*}^3 = o(\sqrt n )$ and $n\to \infty$.}
Finally, by Lyapunov's central limit theorem, we conclude that 
$$
\frac{\sqrt n g^\top (\tilde\beta^*_{S^*} - \beta^*_{S^*})}
{\sqrt{\sigma^2 c_\xi g^\top \left( \frac1n X_{S^*}^\top X_{S^*} \right)^{-1} g }} \to N(0,1),
$$
as $n, d, m \to \infty$ and $B_{S^*}^3 = o_p(\sqrt n )$, which completes the proof of Theorem \ref{T9}. 
%\end{proof}

{
\subsection{Proof of Theorem \ref{upperbound} }
}
The proof of Theorem \ref{upperbound} is quite similar to Proposition \ref{T8}, with fewer scale techniques. 
We first introduce some abbreviations used in this subsection:
\begin{equation}\label{eq: th2.1 relation}
    C_\lambda =\sqrt{40} \cdot \frac{\kappa+ (\sqrt3 -1)\delta}{\kappa-\delta},
\quad A = \frac{8\delta^2 }{ (\kappa-\delta)^2},
\quad B := \left( 1- \sqrt{10}/C_\lambda \right)^2 \cdot \frac{8 \kappa^2}{ (\kappa-\delta)^2},
\end{equation}
where recall $\delta \in (0,1)$ is the parameter in DSRIP$\left((1+2A)s, ~\frac{1+4A}{1+2A}s_0, \delta \right)$ condition, and $
\kappa \in \left(\delta,~1\wedge (\delta+2\sqrt2 \delta) \right)$ is a tuning parameter.
We then introduce the event
$$
\mathcal E_{2.1} := \left\{ \max_{S \in \mathcal S(s,As_0) \bigcup  \mathcal S(As,s_0) }~\sum_{(i,j) \in S} \Xi_{ij}^2 < 10A \cdot \frac{\sigma^2 ss_0\Delta(s,s_0)}{n} \right\}.
$$
%recall that we denote by $\mathcal S(as,bs_0)$ the set collects all s
Folowing Lemma \ref{good event}, we learn that $\mathbf P(\mathcal E_{2.1}) \ge 1- 2 \exp \left( -\frac A3 ss_0 \Delta(s,s_0) \right)$.  

Under this event, we next prove that
\begin{equation}\label{eq: dsresult}
\begin{aligned}
    \|\hat \beta^t - \beta^* \|_2 \le&~ \sqrt{B ss_0 } \lambda_{t} \quad,\\
\hat S_{OG}^{t} := (S_{G^*})^c 
 \cap \hat S^{t} \in& ~\mathcal S (As,s_0) \quad,\\
\hat S_{IG}^{t} := S_{G^*} \cap \hat S^{t } \cap (S^*)^c \in& ~\mathcal S (s, A s_0) \quad,
\end{aligned}
\end{equation}
where $\hat S_{OG}^{t}, \hat S_{IG}^{t}$ represent similar meaning as $\tilde S_{OG}^{t}$ and $\tilde S_{IG}^{t}$ in Table \ref{sphericcase} (the superscript $\hat{}$ denotes the analysis of the first-step algorithm, and the superscript $\tilde{}$ denotes the analysis of the second-step algorithm).

When $t=0$, we have $\hat \beta^0 = \mathbf 0_p$ and $\hat S^0 = \emptyset$, therefore all three conclusions in \eqref{eq: dsresult} hold by choosing proper $\lambda_{(0)}$ satisfying {$\|\beta^*\|_2 \le \frac{3(1+ \sqrt{2})}{2}\sqrt{ss_0} \lambda_{(0)}$ (which will be discussed in the last of this subsection).}
Then, by using mathematical induction, we first assume all three results hold in the $t$-th iteration $(t\ge0)$, and aim to prove that all of them still hold in the $(t+1)$-th iteration. 

\paragraph{Step 1 (Control the shape of $ \hat S_{OG}^{t+1}= (S_{G^*})^c 
 \cap \hat S^{t+1}$).}  
By contradiction, we initially assume that more than $As$ groups are falsely discovered in the $(t+1)$-th iteration.
Then, we can choose arbitrary $As$ false discovered groups and construct an exact $(As ,s_0)$-shape $S_{OG}' \in \mathcal S (As,s_0)$ (the process of constructing $S_{OG}'$ is the same as the step 1 in the proof of Proposition \ref{T8}). 
Based on the element-wise thresholding and the group-wise thresholding operators, for any group $G_j$ selected in $S_{OG}'$, it holds that $\| \hat \beta_{G_j \cap S_{OG}'}^{t+1} \|_2^2 \ge s_0 \lambda_{(t)}^2 $, which yields that
\begin{equation}\label{OGeasy}
    \begin{aligned}
    \sqrt{A ss_0 }\lambda_{(t+1)}  
    \le& \sqrt{\sum_{(i,j) \in S_{OG}'} \langle\Phi_{(i,j)}^\top, \beta^* - \hat \beta^{t} \rangle^2 } 
        + \sqrt{\sum_{(i,j) \in S_{OG}'} \Xi_{ij}^2 }\\
    \overset{(i)}{<} & \delta \|\hat \beta^t - \beta^* \|_2 + \sqrt{\frac{10A \sigma^2 ss_0\Delta(s,s_0)}{n} }\\
    \overset{(ii)}{\le} & \delta \sqrt{Bss_0} \lambda_{(t)} + \frac{\sqrt{10A}}{C_\lambda} \sqrt{ss_0} \lambda_{(\infty)} \\
    \overset{(iii)}{\le} & \left( \frac{\delta \sqrt B}{\kappa} + \frac{\sqrt{10A}}{C_\lambda}\right) \sqrt{ ss_0 }\lambda_{(t+1)}\\
    = & \sqrt{ A ss_0 }\lambda_{(t+1)},
    \end{aligned}
\end{equation}
where inequality (i) follows from the DSRIP$\left((1+2A)s, ~\frac{1+4A}{1+2A}s_0, \delta \right)$ condition and event $\mathcal E_{2.1}$, inequality (ii) follows from the induction hypothesis \eqref{eq: dsresult} in the t-th iteration, and also follows from the relationship \eqref{eq: th2.1 relation}.
Inequality (iii) follows from $\lambda_{(t+1)} = (\kappa\lambda_{(t)} ) \vee \lambda_{\infty}$, and the final equality follows from the relationship \eqref{eq: th2.1 relation} again.
Thus, we find an absurdity, demonstrating that only fewer than $As$ groups can be falsely discovered in the $(t+1)$-th iteration.

Again, by contradiction, we can prove that fewer than $Ass_0$ entries are falsely discovered in the falsely discovered groups.
If not so, we can construct a set $S_{OG}'' \subseteq \hat S_{OG}^{t+1}$ satisfying $|S_{OG}''|=Ass_0$ and $S_{OG}'' \in \mathcal S(As,s_0)$, which leads an absurdity as the same as \eqref{OGeasy} again.

\paragraph{Step 2 (Control the shape of $ \hat S_{IG}^{t+1}= S_{G^*} 
 \cap (S^{*})^c \cap \hat S^{t+1}$ ).} 
By contradiction, we initially assume that more than $Ass_0$ entries are falsely discovered in $S_{G^*} \cap (S^*)^c$. 
Then we can construct a index set $S_{IG}'$ satisfying $S_{IG}' \subseteq \hat S_{IG}^{t+1}, S_{IG}' \in \mathcal S(s, As_0) $ and $|S_{IG}''|=Ass_0$. 
Therefore, similar to \eqref{OGeasy}, we can find the absurdity again.

\paragraph{Step 3 ($\ell_2$ error inequality of $\hat \beta^{t+1}$).} 
Now we prove that the $\ell_2$ error bound in the $(t+1)$-th iteration.
Note that
$$
    \hat \beta_{ij}^{t+1} - \beta_{ij}^* = - \hat H_{ij}^{t+1} \cdot \mathbf1 \left( (i,j) \notin \hat S^{t+1} \right) + \langle\Phi_{(i,j)}^\top, \beta^* - \hat \beta^{t} \rangle + \Xi_{i,j},
$$
where
$$
 \hat H_{ij}^{t+1} = \beta_{ij}^* + \langle\Phi_{(i,j)}^\top, \beta^* - \hat \beta^{t} \rangle + \Xi_{ij}.
$$

We then have
\begin{equation} \label{eq: 2.1 l21}
    \begin{aligned}
& \left\|\hat \beta^{t+1} - \beta^* \right\|_2 \\
\le& \left[ \sum_{(i,j) \in S^*} \left\{\langle\Phi_{(i,j)}^\top, \beta^* - \hat \beta^{t} \rangle
   +\Xi_{ij} - \hat H_{ij}^{t+1} \cdot \mathbf1 \left( (i,j) \notin \hat S^{t+1} \right) \right\}^2 \right. \\
&\quad + \left.\sum_{(i,j) \in \hat S^{t+1}\setminus S^*} 
   \left\{\langle\Phi_{(i,j)}^\top, \beta^* - \hat \beta^{t} \rangle+\Xi_{ij} \right\}^2 \right]^{1/2}\\
\le& \sqrt{ \sum_{(i,j) \in \hat S^{t+1}\cup S^*} \left\{\langle\Phi_{(i,j)}^\top, \beta^* - \hat \beta^{t} \rangle +\Xi_{ij} \right\}^2}\\
& + \left\{ \sum_{(i,j) \in S^*} \left( \hat H_{ij}^{t+1} \right)^2 \mathbf1 \left( \left|\hat H_{ij}^{t+1}\right| <\lambda_{(t+1)}\right)  \right.\\
&\quad\quad + \sum_{(i,j) \in S^*} \left( \hat H_{ij}^{t+1} \right)^2 
   \mathbf1\left( \left|\hat H_{ij}^{t+1}\right|\ge\lambda_{(t+1)}\right)\\
&\quad\quad\quad\quad\quad \left. \times \mathbf1 \left(\sum_{k: (k,j)\in S^* }\left(\hat H_{kj}^{t+1}\right)^2 
   \mathbf1\left( \left|\hat H_{kj}^{t+1}\right|\ge\lambda_{(t+1)} \right) < s_0\lambda_{(t+1)}^2 \right) \right\}^{1/2}\\
\le & \delta \left\|\hat \beta^{t}-\beta^* \right\|_2 
   + \sqrt{\sum_{(i,j)\in S^*} \Xi_{ij}^2+ \sum_{(i,j)\in \hat S_{IG}^{t+1}} \Xi_{ij}^2 + \sum_{(i,j)\in \hat S_{OG}^{t+1}} \Xi_{ij}^2 }
   + \sqrt{2ss_0} \lambda_{(t+1)},
\end{aligned}
\end{equation}
where the second inequality follows from H\"{o}lder inequality and the definition of the double sparse thresholding operator $\mathcal T_{\lambda_{(t+1)},s_0}$ in the ($t+1$)-th iteration.
The last inequality follows from the DSRIP$\left((1+2A)s, ~\frac{1+4A}{1+2A}s_0, \delta \right)$ condition, since by the induction hypothesis \eqref{eq: dsresult} and the first two step proofs, we have $S^* \cup \hat S^{t} \cup \hat S^{t+1} \in \mathcal S \left((1+2A)s, ~\frac{1+4A}{1+2A}s_0\right)$.
Consequently, by the event $\mathcal E_{2.1}$ and \eqref{eq: dsresult}, we conclude that
\begin{equation} \label{eq: 2.1 l22}
 \begin{aligned}
\left\|\hat \beta_{S^*}^{t+1} - \beta_{S^*}^* \right\|_2 
\le & \delta \left\|\hat \beta^{t}-\beta^* \right\|_2 
   + \sqrt{\frac{30A\sigma^2 s s_0 \Delta(s,s_0)}{n}}
   + \sqrt{2ss_0} \lambda_{(t+1)}\\
\le & \left(\frac{\delta \sqrt B}{\kappa} + \frac{\sqrt{30A}}{C_{\lambda}} + \sqrt2 \right) \sqrt{ss_0} \lambda_{(t+1)}\\
=& \sqrt{Bss_0 } \lambda_{(t+1)},
\end{aligned}
\end{equation}
where the last equality follows from the relationship \eqref{eq: th2.1 relation}.
 
Therefore, we prove that all three results in \eqref{eq: dsresult} still hold in the $(t+1)$-th iteration.

\paragraph{Step 4 (Feasible initial thresholding parameter).}
We finally end the proof of Theorem \ref{upperbound} by providing a suitable initial thresholding parameter 
$$
\lambda_{(0)} := \frac{ \|X^\top Y /n\|_\infty + \sqrt{10\sigma^2 (\log p)/n}}{ \sqrt2 \kappa}.
$$
By the relationship \eqref{eq: th2.1 relation}, we learn that $C_\lambda > \sqrt{40}$ and $\sqrt{B} \ge \frac{\sqrt 2 \kappa}{\kappa- \delta} \ge \frac{\sqrt 2 \kappa}{1- \delta}$. 
Consequently,
$$
\begin{aligned}
\sqrt{Bss_0} \lambda_{(0)} \ge & \sqrt{ss_0} \cdot \frac{\sqrt 2 \kappa}{1- \delta} \cdot \left( \frac{ \|X^\top Y /n\|_\infty + \sqrt{10\sigma^2 (\log p)/n}}{ \sqrt2 \kappa} \right)\\
\overset{(i)}\ge & \frac{1}{1-\delta} \cdot \left(\sqrt{ss_0} \left\| \frac1n X_{S^*}^\top (X_{S^*}  \beta_{S^*}^* + \sigma \xi) \right\|_\infty + \| \Xi_{S^*} \|_2 \right)\\
\ge & \frac{1}{1-\delta} \cdot \left( \left\| \frac1n X_{S^*}^\top X_{S^*}  \beta_{S^*}^* \right\|_2  -  \left\| \frac\sigma n X_{S^*}^\top \xi \right\|_2 + \| \Xi_{S^*} \|_2 \right)\\
\ge & \frac{1}{1-\delta} \left\| \frac1n X_{S^*}^\top X_{S^*} \right\|_2 \cdot \left\|\beta_{S^*}^* \right\|_2  \\
\ge & \left\|\beta^* \right\|_2 ,
\end{aligned}
$$
where inequality (i) follows from Lemma \ref{good event}, holding with probability greater than $1- \exp(ss_0 \Delta(s,s_0)/3)$, and the last inequality follows from DSRIP condition.

Therefore, combining Step 1-4 we complete the proof of Theorem \ref{upperbound}. 

{
\subsection{Miscellaneous}\label{sec: random}}
This subsection provides an in-depth discussion of two assumptions of the design matrix: the DSRIP condition and the rate $B_{S^*} = \max_{i\in[n]}\|X_{i,S^*}\|_2 $.

{
\subsubsection{Connection between DSRIP and double sparse Riesz condition}}
We first introduce the double sparse Riesz condition.
\begin{definition}\label{def: DSRC}
    We say that $X \in \mathbb R^{n\times p}$ satisfies the Double Sparse Riesz Condition $DSRC(as,bs_0, C_U, C_L)$ if and only if 
$$
C_L\|u\|_2^2  \le \frac1n\bigl\|X_{S} ~u\bigr\|_2^2 \le C_U\|u\|_2^2,
 \text{ for every } S\in\mathcal S(as,\,bs_0) \text{ and } u \in  \mathbb{R}^{|S|}\setminus \{ \mathbf 0_{|S|}\},
$$
where $C_U\ge C_L>0$ are two arbitrary constants and $a,b >0$ depend on $C_U$ and $C_L$. 
\end{definition}
The sparse Riesz condition is a well-known structural assumption in sparse regression \citep{BickelTsy09lassodantzig, yuan18grad}, and by Definition \ref{def: DSRC}, we extend this condition to the double sparse space.
Now we assume $X$ satisfies DSRC $\left( (1+2A)s,\frac{1+4A}{1+2A}s_0, C_U, C_L \right)$, and rewrite the gradient descent \eqref{eq: grad learning} with a learning rate $\gamma$ as 
\begin{equation}\label{eq: new grad learning}
\begin{aligned}
\tilde H^{t+1} &=  \tilde \beta^{t } + \frac\gamma nX^\top \left(Y-X\tilde \beta^{t } \right) \\
&=  \tilde \beta^{t } + \frac\gamma nX^\top \left(X \tilde \beta^* + \sigma \tilde \xi -X\tilde \beta^{t } \right) \\
&= \tilde \beta^* +  \left( \frac\gamma nX^\top X -I_p\right)(\tilde\beta^*- \tilde \beta^t) + \gamma\tilde \Xi.
\end{aligned}
\end{equation}
By solving the system of inequalities in $\gamma$ and $\delta$
$$
\left\{\begin{matrix}
|\gamma C_U -1|\le  \delta,\\
|\gamma C_L -1|\le  \delta,
\end{matrix}\right.
$$
we obtain the feasible region of learning rate $\gamma$ is $(0, 2/C_U)$, leading 
$$
\left\| \frac\gamma n X_S^\top X_S -I_p\right\|_2\le \underbrace{ \max\Big\{ |\gamma C_U-1|,~|1- \gamma C_L| \Big\}}_{=: ~\delta'} <1
$$
for every set $S \in \mathcal S \left( (1+2A)s,\frac{1+4A}{1+2A}s_0 \right)$.
%a feasible $\delta = \max\left\{ |\gamma C_U-1|,~|1- \gamma C_L| \right\} <1$.
%and subsequently determine the corresponding equivalent DSRIP constant $\delta = \frac{C_U - C_L}{C_U+ C_L} \in (0,1)$.
By incorporating the revised decomposition \eqref{eq: new grad learning}, we can rescale inequality \eqref{tech} exactly as under the original DSRIP condition (with parameter $\delta'$), and thus the proof and conclusions of Theorem \ref{upperbound} and Proposition \ref{T8} remain valid, requiring only minor scaling adjustments involving $\delta'$, $\gamma$, $C_U$, and $C_L$.

The above analysis establishes the equivalence of DSRIP and DSRC while relaxing our original DSRIP assumption.  
Fundamentally, both conditions bound the deviation of the sample covariance matrix from the identity matrix $I_p$, thereby enabling a (Gaussian) location model perspective \cite{CM18, MN19, LZ22} on the double sparse regression.

%以上的分析证实了dsrip和dsrc的等价性，并放宽了我们原有的dsrip条件.此外，我们揭示了这些条件的本质：通过控制样本协方差阵与单位阵的差异，将对回归模型的分析转变为对截距模型的分析。

%基于现在的分解式\eqref{eq: new grad learning}以及delta的形式，定理\ref{upperbound}和命题\ref{T8}的证明内容和结果仍然成立——只不过需要基于$\delta, \gamma, C_U, C_L$ 进行少量scaling调整。
%This equivalence weakens our original DSRIP assumption and guarantees that our proposed procedure enjoys rate-optimal theoretical properties under some general random-design settings. 
{
\subsubsection{Realizability of double sparse Riesz condition}\label{subsec: realize DSRC}
}
We next consider the realizability of the double sparse Riesz condition in the sub-Gaussian setting:
For each $i \in [n]$, we assume the $i$-th observation $X^{(i)} \overset{d}{=} \Sigma^{1/2} Z^{(i)} $, where $\Sigma \in \mathbb{R}^{p \times p}$ is the population covariance matrix and $Z^{(1)},\cdots, Z^{(n)} \in \mathbb R^p$ are i.i.d. centered 1-sub-Gaussian random vectors such that $\mathbf E ( Z^{(i)} Z^{(i)\top} ) = I_p$.
Assume the $\Sigma$ has bounded eigenvalues as $C_L' \le \Lambda_j (\Sigma )  \leq C_U'$ for every $j \in [p]$, where $C_U'\ge C_L'>0$ are two absolute constants.

\begin{proposition}\label{prop: dsrc}
  Under the above sub-Gaussian setting, assume
\begin{equation}\label{eq: sample size}
    n \ge \left( \frac{C+ \sqrt{2/c}}{1\wedge (C_L'/2)} \right)^2 \times \Big\{ (2A+1)s\log (em/s) + (4A+1)ss_0 \log(ed/s_0) \Big\}.
\end{equation}
Then, with a probability greater than $1-  2e^{ -(2A+1)s\log (em/s) - (4A+1)ss_0 \log(ed/s_0) }$, the design matrix $X \in \mathbb R^{ n \times p}$ satisfies DSRC $\left( (1+2A)s,\frac{1+4A}{1+2A}s_0, C_U, C_L \right)$, where
$$
A = \frac{8\delta^2}{(\kappa-\delta)^2}, 
\quad \delta = \frac{C_U'}{C_U'+ C_L'}, 
\quad \kappa \in (\delta,1),
\quad C_U = C_U' + \frac{C_L'}2,
\quad C_L  = \frac{C_L'}2,
$$
and $C,c>0$ are two absolute constants.

\end{proposition}

\begin{proof}[proof of Proposition \ref{prop: dsrc}]
For any set $S$ satisfying $S \in \mathcal S \left( (1+2A)s,\frac{1+4A}{1+2A}s_0 \right)$ and $|S| = (1+4A)ss_0$, by Remark 5.40 in \cite{vershynin2010introduction}, we have
$$
\mathbf P \left( \left\|\frac1n X_S^\top X_S- \Sigma_{SS} \right\|_2 > \max(\iota , \iota^2)  \right) \le 2e^{-ct^2},
$$
where $\iota = C\sqrt{\frac{(1+4A)ss_0}n } +\frac{t}{\sqrt n}$ and $C,c$ are two fixed positive constants.

Then, by taking $t =\sqrt{\frac2c}  \cdot \sqrt{ (2A+1)s\log (em/s) + (4A+1)ss_0 \log(ed/s_0)}$ and following the sample size assumption \eqref{eq: sample size}, we get $\iota \vee \iota^2 = \iota \le C_L'/2 $, which yields that
\begin{equation}\label{eq: maxop}
\begin{aligned}
&\mathbf P \left(  \max_{S \in \mathcal S \left( (1+2A)s,~\frac{1+4A}{1+2A}s_0 \right)} \left\|\frac1n X_S^\top X_S- \Sigma_{SS} \right\|_2 > \frac{C_L'}{2}\right)\\
\le& \mathbf P \left( \bigcup_{\substack{
    S\in\mathcal S \left( (1+2A)s,~\frac{1+4A}{1+2A}s_0 \right), \\
    \lvert S\rvert=(1+4A) ss_0 }} \left\{ \left\|\frac1n X_S^\top X_S- \Sigma_{SS} \right\|_2 > \iota \vee \iota^2 \right\}\right)\\
\le& 2\binom m{ (1+2A)s} \binom{(1+2A)sd}{ (1+4A)ss_0} \exp\left(-ct^2 \right)\\
\le& 2 \exp\Big\{ -(2A+1)s\log (em/s) - (4A+1)ss_0 \log(ed/s_0) \Big\},
\end{aligned}
\end{equation}
where the last inequality follows from $\binom yx \le (ey/x)^x$ for every $0<x<y$.
Therefore, with a probability greater than $1-  2e^{ -(2A+1)s\log (em/s) - (4A+1)ss_0 \log(ed/s_0) }$, for every $S \in \mathcal S \left( (1+2A)s,~\frac{1+4A}{1+2A}s_0 \right)$, we have
$$
\left\|\frac1n X_S^\top X_S \right\|_2 
~\le~ \|  \Sigma_{SS} \|_2 + {C_L'}/{2}
~\le~ C_U' +  {C_L'}/{2} ~=:~ C_U, 
$$
and 
$$
\left\|\frac1n X_S^\top X_S \right\|_2 
~\ge~ \|  \Sigma_{SS} \|_2 - {C_L'}/{2}
~\ge~  {C_L'}/{2}~ =:~ C_L. 
$$
Hence, by choosing learn rate $\gamma = \frac{2}{C_U + C_L} $ and $\delta = \frac{C_U - C_L}{C_U + C_L}$, we complete the proof of Proposition \ref{prop: dsrc}.
\end{proof}

\paragraph{Some examples satisfying DSRC}
Here we provide some illustrative examples satisfying DSRC.
Consider an exponential-decay Toeplitz covariance matrix $\Sigma$ with entries $\Sigma_{ij}=\rho^{|i-j|}$ for a constant $\rho \in [0,1)$.
This covariance structure is ubiquitous in high-dimensional studies \cite{JMLR:v11:raskutti10a, fan2014fcp, zhao2022high}.
It can be proved that
$$
\Lambda_k(\Sigma)
=\frac{1-\rho^2}
{1 -2\rho\cos\bigl(\tfrac{\pi k}{p+1}\bigr)
+\rho^2}, \text{ for every } k \in [p],
$$
leading 
$$
 C_L' = \frac{1-\rho}{1+ \rho} \le \Lambda_{k}(\Sigma) \le \frac{1+ \rho}{1-\rho} = C_U'.
$$
Therefore, this $\Sigma$ satisfies our DSRC assumption.

Additionally, consider the special case of an i.i.d. Gaussian design $X_{ij}\sim N(0,1)$ (so $\rho=0$ and $C_U'=C_L'=1$), which is standard in compressed sensing \cite{Ndaoud2020NCS} and high-dimensional sparse regression \cite{ARW2025Zhu}.
Then the design matrix $X$ even satisfies the DSRIP condition with high probability as soon as $n \gtrsim s \log(em/s) + ss_0 \log(ed/s_0)$.

{
\subsubsection{A high probability bound of $B_{S^*}$}\label{subsec: B}
}
Here we provide the rate of $B_{S^*}$ (appears in Theorem \ref{T9}) in the context of sub-Gaussian design.
For each $i \in [n]$, we consider the vector $X^{(i)} \overset{d}{=} \Sigma^{1/2} Z^{(i)} \in \mathbb{R}^{p}$ as in Appendix \ref{subsec: realize DSRC}
, and assume the submatrix $\Sigma_{S^*, S^*} \in \mathbb{R}^{|S^*| \times |S^*|}$ has bounded spectral norm as $ \Lambda_{\max}(\Sigma_{S^*, S^*} )  \leq C_{S^*}$, where $C_{S^*}$ is an absolute constant.
Then, %each $X^{(i)}_{S^*} \overset{d}{=} \left(\Sigma^{1/2}\right)_{S^*,\cdot} Z^{(i)} \in \mathbb R^{|S^*|}$.
by the Hanson-Wright inequality \citep{vershynin13HW}, for a fixed $i \in [n]$, we have
$$
\mathbf{P}\left\{ \left\|X_{S^*}^{(i)} \right\|_2^2 \ge tr(\Sigma_{S^*, S^*}) +K^2 \left( \sqrt{\frac tc} \|\Sigma_{S^*,S^*}\|_F + \frac tc \|\Sigma_{S^*,S^*}\|_{2} \right) \right\} \leq e^{-t}
$$
with two constants $K>0, c \in (0,1)$.
We take $t = 2 \log n$, by the inequalities
$$
\begin{aligned}
    &tr(\Sigma_{S^*, S^*}) +K^2 \left( \sqrt{\frac tc} \|\Sigma_{S^*,S^*}\|_F + \frac tc \|\Sigma_{S^*,S^*}\|_{2} \right)\\
    \le & C_{S^*} ss_0 + \frac{C_{S^*} K^2}{c} \left( \sqrt{ss_0 t} + t\right)\\
    \le & \frac32 C_{S^*} \left( 1+ \frac{K^2}{c} \right) \left( ss_0 + t\right),
\end{aligned}
$$
we have
\begin{equation}\label{eq: hw}
\begin{aligned}
  & \mathbf{P}\left[\bigcup_{i \in [n]} \left\{ \left\|X_{S^*}^{(i)} \right\|_2^2 \ge \frac32 C_{S^*} \left( 1+ \frac{K^2}{c} \right) \left( ss_0 + 2 \log n\right) \right\} \right] \\
  \le& \sum_{i \in [n]}\mathbf{P} \left\{ \left\|X_{S^*}^{(i)} \right\|_2^2 \ge \frac32 C_{S^*} \left( 1+ \frac{K^2}{c} \right) \left( ss_0 + 2 \log n\right) \right\} \\
  \le & n e^{-2 \log n} = n^{-1}.
\end{aligned}
\end{equation}
With a probability greater than $1-1/n$, \eqref{eq: hw} leads $B_{S^*} \lesssim \sqrt{ss_0 + \log n}$.
Since $\log n \prec n^{1/3}$ as $n \to \infty$, it follows that
$$
\left\{ B_{S^*}^3 = o_p(\sqrt n)\right\} \Leftarrow \left\{ (ss_0)^3 + \log^3 n = o_p(n)\right\} \Leftrightarrow \left\{ ss_0 = o_p(n^{1/3})\right\},
$$
which provides a clear understanding of the technical assumption in Theorem \ref{T9}.

\section{Proof of the minimax lower bounds}\label{lowerproof}
First, we consider the minimax lower bound for signal estimation and transform the minimax risk into Bayesian risk for in-depth analysis. This analytical framework is inspired by \cite{MN19}.

\begin{lemma}\label{lemma:minimaxori}
For any $1\le s<m,~1\le s_0<d$, any subset ${\Theta} \subseteq \mathbb R^p$ and for any prior probability distribution $\pi$ on $\mathbb R^p$, as we assume the design matrix $X \in \mathbb R^{n \times p}$ is fixed, we have
\begin{equation}\label{lemma1.ineq}
\begin{aligned} 
&\inf_{\hat \beta} \sup_{\beta^* \in {\Theta}} \underset{Y \sim P_{\beta^*}}{\mathbf{E}}\|\hat \beta - \beta^* \|_2^2 \\
\ge & \inf_{\hat T  } \underset{\beta^* \sim \pi}{\mathbf{E}}~ 
      \underset{Y \sim P_{\beta^*}}{\mathbf{E}} \|\hat T(Y,X) - \beta^* \|_2^2\\
& - 2\underset{\beta^* \sim \pi}{\mathbf {E}}~ \underset{Y \sim P_{\beta^*}}{\mathbf{E}}
      \Bigg\{\bigg(\underset{\beta^{ {\Theta}} | Y}{\mathbf{E} }\Big(\left\|\beta^{ {\Theta}}\right\|_2^2 \big| Y\Big)+\|\beta^*\|_2^2\bigg) \mathbf{1}\left(\beta^* \notin {\Theta} \right)\Bigg\},
\end{aligned}
\end{equation}
where $\inf_{\hat \beta}, \inf_{\hat T  }$ are taken over all estimator of $\beta^*$, and $\beta^{ {\Theta}}:= \beta \mathbf1(\beta \in {\Theta})$.
We denote by $\underset{\beta^{ {\Theta}} | Y}{\mathbf{E} }(\cdot|Y)$ an expectation based on the conditional distribution $\frac{ P(Y|\beta) \pi_{\Theta}(\beta) }{ \int  P(Y|\beta ) \pi_{\Theta}(\beta ) \mathrm d \beta}$, where $\pi_{\Theta}$ denotes the probability measure $\pi$ conditioned by the event $\{ \beta \in {\Theta} \}$.
\end{lemma}

\subsection{Proof of Theorem \ref{allselectorD1D2}}\label{sec: pi12}
%\subsection{Lower bounds based on Bernoulli prior}\label{Bernoulli}
Recall that we mainly focus on two subsets of ${\Theta}_e(s,s_0,a)$, which can be described as 
$$
\begin{aligned}
\Theta_{e,1} &:= \left\{ \beta \in \Theta_e \left(s,s_0, a \right) ~\left|~
\begin{aligned}
    & G^*(\beta) = [s], \\
    & \beta_{ij} = a ~\text{ for every } (i,j) \in \operatorname{supp}(\beta)
\end{aligned} \right.\right\},\\
\Theta_{e,2} &:= \left\{ \beta \in \Theta_e\left(s,s_0, a \right) ~\left|~
\beta_{ij} =\begin{cases}
 a, & \text{ if } i \in [s_0] \text{ and } j \in G^*(\beta) \\
 0, & \text{ otherwise}
\end{cases} \right.\right\}.
\end{aligned}
$$

\begin{comment}
$$
\begin{aligned}
\Theta_1 &:= \left\{ \beta \in \Theta\left(s,s_0, a, a \sqrt{s_0/10} \right) ~\left|~
\begin{aligned}
    & G^*(\beta) = [s], \\
    & \| \beta_{G_j} \|_0 \ge {s_0}/{10}~ \text{ for every } j \in [s],\\
    & \beta_{ij} = a ~\text{ for every } (i,j) \in \operatorname{supp}(\beta)
\end{aligned} \right.\right\},\\
\Theta_2 &:= \left\{ \beta \in \Theta\left(s,s_0, a, a \sqrt{s_0/10} \right) ~\left|~
\beta_{ij} =\begin{cases}
 a, & \text{ if } i \in [s_0] \text{ and } j \in G^*(\beta) \\
 0, & \text{ otherwise}
\end{cases} \right.\right\}.
\end{aligned}
$$
\end{comment}
In ${\Theta}_{e,1}$, information is limited to the location of support groups, omitting details on their support entries. 
The subspace ${\Theta}_{e,2}$ provides insight into the support entries within each support group, but it lacks information on the location of support groups.
Given the settings above, we frame the signal estimation as a support identification problem, i.e., to estimate the decoder $\eta^* = \{\mathbf{1}(\beta^*_{ij} \ne 0) \}_{ij}\in \{0,1\}^{p}$. 
%Recall that we assume that $X$ satisfies DSRIP$(s,s_0,\delta)$ condition with some constant $\delta \in (0,1)$. This assumption will be used in the proof of Theorem \ref{allselectorD1D2} and \ref{LB2priors}.
Plus, we define two functions:
\begin{equation}\label{aux2fun}
\begin{aligned}
    \psi(d,s_0,a,\sigma) &:=(d-s_0){\Phi}\left(-\frac{t(a,d,s_0,\sigma)}{\sigma}\right) + s_0 {\Phi}\left(-\frac{a\sqrt n-t(a,d,s_0,\sigma)}{\sigma}\right),\\
    t(a,d,s_0,\sigma) &:= \frac{a\sqrt n}{2} + \frac{\sigma^2}{a\sqrt n} \log \frac{d-s_0}{s_0}.
\end{aligned}
\end{equation}

%We use two Bernoulli prior distributions to lower bound the losses on ${\Theta}_1$ and ${\Theta}_2$ respectively. 
%Specifically, based on $\eta \to \beta \to Y$ and $Y \sim P_{\beta}$ for the fixed matrix $X$,
%({here we assume each $\eta \in \{0,1\}^{p}$ corresponds only one $\beta \in \mathbb R^{p}$ and one distribution of Y, as $P_\eta$})
%we can get estimator $\hat \eta = \hat \eta(Y,X) \in \mathbb R^{p}$, and call $\hat \eta_{ij}$ is the estimator (or selector) of $\eta_{ij} = \mathbf1(\beta_{ij} \ne 0)$, i.e., the sign of whether the $i$-th variable in the $j$-th group is in the support set.
 
Now, we define two prior distributions $\pi_1$ and $\pi_2$ corresponding to parameter subsets ${\Theta}_{e,1}$ and ${\Theta}_{e,2}$ respectively.
For $\pi_1$, suppose that for each $(i,j) \in [d] \times [s]$ (by the definition of ${\Theta}_{e,1}$, the support index must be in $[d] \times [s]$), the number of support entries is from a binomial distribution, i.e., $\sum_{(i,j) \in [d] \times [s]} \eta_{ij}^* \sim Bin (ds, s_0'/d)$, where $ s_0' $ satisfies $1\le s_0' < s_0$ is an integer determined later. 
And assume that $\eta_{G_j}^* \equiv \mathbf 0_d$ for all $ j >s$.

For $\pi_2$, suppose that the number of support groups is from a binomial distribution $Bin(m, s^{\prime} / m)$, where $ s' < s$ is also an integer determined later.
In each support group, only the first $s_0$ entries are support entries. Therefore $\eta_{ij}^* \equiv 0$ for all $i > s_0, j \in [m]$.

Note that $\pi_1$ and $\pi_2$ are two prior distributions of $\eta^*$, while $\beta^* = a\cdot\eta^*$ in ${\Theta}_{e,1}$ and ${\Theta}_{e,2}$.
%By letting $\beta = a \eta_k$, we can also get the prior distributions of $\beta$ for $k=1,2$ respectively. 
Therefore, the term $a \hat \eta$ can also be an estimator of $\beta^*$, and we replace $\beta^*$ and $\hat \beta$ (in Lemma \ref{lemma:minimaxori}) by $\eta^*$ and $\hat \eta$ with prior $\pi_k$ (for $k=1,2$).
Then we have 
\begin{equation}\label{eq:priork}
\begin{aligned}
&\inf_{\hat \eta \in \{0,1\}^p} \sup_{\beta^* \in  {\Theta}_{e,k}} 
    \mathbf{E}_{Y\sim P_{\beta^*} } \|\hat \eta - \eta^*\|_2^2 \\
\ge &\inf_{\hat \eta \in [0,1]^p} \sup_{\beta^* \in  {\Theta}_{e,k}} 
    \mathbf{E}_{Y\sim P_{\beta^*} } \|\hat \eta - \eta^*\|_2^2 \\
\ge &\inf_{\hat T } \underset{\eta^* \sim \pi_k}{\mathbf E} \underset{Y \sim P_{\beta^*} }{\mathbf E} \|\hat T(Y,X) - \eta^*\|_2^2 \\
    &-2 \underset{\eta^* \sim \pi_k}{\mathbf E} \underset{Y \sim P_{\beta^*} }{\mathbf E}  \Bigg\{ \bigg( \underset{\eta^{ {\Theta}_{e,k}}|Y}{\mathbf{E}}   \Big( \|\eta^{ {\Theta}_{e,k}} \|_2^2 \big| Y \Big) +
    \| \eta^* \|_2^2  \bigg) \mathbf1( a\eta^* \notin {\Theta}_{e,k}) \Bigg\} \\
\ge &\inf_{\hat T } \underset{\eta^* \sim \pi_k}{\mathbf E} \underset{Y \sim P_{\beta^*} }{\mathbf E}  \|\hat T(Y,X) - \eta^* \|_2^2 \\
    &-2 ss_0 \underset{\eta^* \sim \pi_k}{\mathbf P} \left( a\eta^* \notin {\Theta}_{e,k} \right) 
    -2 \underset{\eta^* \sim \pi_k}{\mathbf E} \big( \| \eta^*\|_2^2 \mathbf1( a\eta^* \notin {\Theta}_{e,k}) \big),
\end{aligned}
\end{equation}
where $\inf_{\hat T}$ denotes the infimum over all estimators $\hat T \in [0,1]^p$, and the last inequality follows from $\|\eta^{ {\Theta}_{e,k}} \|_2^2 =  \|\eta \|_2^2\cdot \mathbf1( a\eta \in  {\Theta}_{e,k} )\le ss_0$.

\begin{proof}[Proof of \eqref{allselectorD1}]

Recall that by $\beta^* \in {\Theta}_{e,1}$, only the first $s$ groups are support groups.
And based on the prior $\pi_1$, the total number of support entries $v: = \sum_{j\in [m]} v_j$ is from a binomial distribution $Bin(ds, s_0' / d)$, whence we have 
\begin{equation}\label{eq:23term}
   \begin{aligned}
&2ss_0 \mathbf P_{\eta^* \sim \pi_1} \Big( a\eta^* \notin {\Theta}_{e,1} \Big) + 2\mathbf E_{\eta^* \sim \pi_1}  \Big( \|\eta^*\|_2^2 \mathbf1(a\eta^* \notin {\Theta}_{e,1}) \Big)\\
=& 2ss_0 \mathbf P \left( v > ss_0\right) + 2\mathbf E \Big\{ v \cdot \mathbf1(v  >ss_0)  \Big\} \\
\le& 2s s_0 \exp\left(-\frac{3s(s_0-s_0')^2}{2(s_0+2s_0')}\right) 
    +2ss_0' \exp\left(-\frac{3s(s_0-s_0')^2}{2(s_0+2s_0')}\right),
\end{aligned}
\end{equation}
where the first equality follows from the definition of $\Theta_1$, and the last inequality follows from Lemma \ref{berncen} and the bound
\begin{equation*}
\begin{aligned}
&\mathbf E \Big\{ v \cdot \mathbf1(v  >ss_0)  \Big\}\\
=&\sum_{i=ss_0+1}^{ds} i \binom{ds}{i} \left( \frac{s_0' }{d} \right)^i \left(1- \frac{s_0'}{d} \right)^{ds-i}\\
=& ss_0' \sum_{i-1=ss_0}^{ds-1} \frac{(ds-1)!}{(i-1)!~[(ds-1)-(i-1)]!} \left( \frac{s_0'}{d} \right)^{i-1} \left(1- \frac{s_0'}{d} \right)^{(ds-1)-(i-1)}\\
=& ss_0' P\left( Bin(ds-1 , \frac{s_0'}{d} ) \ge ss_0\right)\\
\le& s s_0' P\left( Bin(ds , \frac{s_0'}{d} ) \ge s s_0\right)\\
\le& ss_0' \exp\left(-\frac{3s(s_0-s_0')^2}{2(s_0+2s_0')}\right) . 
\end{aligned}
\end{equation*}

\begin{comment}
Under prior $\pi_1$, for the last two terms in \eqref{eq:priork}, {let $v_j := \sum_{i \in [d]} \eta_{ij}^* \sim Bin(d, s_0'/d)$ and $v: = \sum_{j\in [m]} v_j \sim Bin(ds, s_0'/d)$, whence we have 
\begin{equation}\label{eq:23term}
   \begin{aligned}
&2ss_0 \mathbf P_{\eta^* \sim \pi_1} \Big( a\eta^* \notin {\Theta}_1 \Big) + 2\mathbf E_{\eta^* \sim \pi_1}  \Big( \|\eta^*\|_2^2 \mathbf1(a\eta^* \notin {\Theta}_1) \Big)\\
=& 2ss_0 \Big\{ \mathbf P \left( v > ss_0\right) + \mathbf P \left( v \le ss_0, \text{ and } v_j < s_0 /10 \text{ for some } j \in [s]\right)\Big\}\\
& + 2\mathbf E \Big\{ v \Big[ \mathbf1(v  >ss_0)  + \mathbf1(v \le ss_0, \text{ and } v_j < s_0 /10 \text{ for some } j \in [s]) \Big]  \Big\} \\
\le& 2ss_0 \mathbf P \left( v > ss_0\right) + 2\mathbf E \left\{ v \mathbf1(v  >ss_0) \right\}  + 4ss_0 \sum_{j\in[s]} \mathbf P\left( v_j < s_0/10 \right)\\
\le& 2s s_0 \exp\left(-\frac{3s(s_0-s_0')^2}{2(s_0+2s_0')}\right) 
    +2ss_0' \exp\left(-\frac{3s(s_0-s_0')^2}{2(s_0+2s_0')}\right)\\
  &  +4s^2s_0 \exp\left(-\frac{15(s_0'-s_0/10)^2}{40s_0' - s_0}\right),
 %+ (2s s_0+ 2ss_0') \exp\left(-\frac{s(s_0-s_0')^2}{2s_0}\right) ,
\end{aligned}
\end{equation}
}
\end{comment}

For the first term in \eqref{eq:priork}, define
\begin{equation*}
\tilde Y_{ij}:= Y - \sum_{(k,\ell) \ne (i,j)} X_{(k\ell)} ~ \beta_{k\ell}^* = X_{(ij)} \beta_{ij}^* + \sigma \xi.
\end{equation*}
Then we have
\begin{align*}
& \inf_{\hat T }\underset{\eta^* \sim \pi_1}{\mathbf E}
    \underset{Y \sim P_{\beta^*} }{\mathbf E} \|\hat T(Y,X) - \eta^* \|_2^2 \\
\overset{(i)}{\ge} & \sum_{j=1}^s \sum_{i=1}^d \underset{\eta_{\setminus (ij)}^* }{\mathbf E}
    \left\{ \inf_{\hat T_{ij} }
    \underset{\eta_{ij}^* }{\mathbf E} \underset{Y}{\mathbf E}
    \left( (\hat T_{ij}(Y,X) - \eta_{ij}^*)^2 \Big| \eta_{\setminus (ij)}^* \right) \right\} \\
\overset{(ii)}{\ge}& \sum_{j=1}^s \sum_{i=1}^d \underset{\eta_{\setminus (ij)}^* }{\mathbf E} \left\{ \inf_{\hat T_{ij} }
    \underset{\eta_{ij}^* }{\mathbf E} \underset{Y}{\mathbf E}
    \left( (\hat T_{ij}(\tilde Y_{ij},X) - \eta_{ij}^*)^2 \Big| \eta_{\setminus (ij)}^* \right) \right\} \\
%= &  \sum_{j=1}^s \sum_{i=1}^d \underset{\eta_{\setminus (ij)}}{\mathbf E}  \left\{ \inf_{\hat T_{ij} } \underset{\eta_{ij}}{\mathbf E} \left( \mathbf1(\eta_{ij}=0)
%    \underset{\tilde Y_{ij}|\eta_{ij}=0}{\mathbf E} |\hat T_{ij}(\tilde Y_{ij},X)|^2  
%    + \mathbf1(\eta_{ij}=1) \underset{\tilde Y_{ij}|\eta_{ij}=1}{\mathbf E}  |1-\hat T_{ij}(\tilde Y_{ij},X)|^2  \right) \right\}\\
=& \sum_{j=1}^s \sum_{i=1}^d  \inf_{\hat T_{ij}} \left\{
    \left(1-\frac{s_0'}{d}\right)
    \underset{\tilde Y_{ij}\sim N(\mathbf 0,\sigma^2 \mathbf I_n)}{\mathbf E}
     \left( \hat T_{ij}(\tilde Y_{ij},X) \right)^2\right. \\
    &  \qquad\qquad \qquad \left.+ \frac{s_0'}{d} \left(
    \underset{\tilde Y_{ij} \sim  N(aX_{ij},\sigma^2 \mathbf I_n)}{\mathbf E} \left( 1 -\hat T_{ij}(\tilde Y_{ij},X) \right)^2
    \right) \right\}.
\end{align*}
where inequality (i) only involves the first $s$ groups since we can always set $\hat T_{ij} = 0$ for all $ j>s, i \in [d]$ by the construction of $\pi_1$.
Inequality (ii) is based on that under independent prior distributions of the entries of $\eta^*$, the oracle selector of a given component $\eta_{ij}^*$ does not depend on the rest of the components. 
And the infimum in the last equality can be achieved by the selector 
\begin{equation*}
\hat T_{ij}^{*}(\tilde Y_{ij},X) 
:= \frac{1}{1+ \frac{(d-s_0')\boldsymbol \varphi_{\sigma} ( \tilde Y_{ij})  }
    {s_0' \boldsymbol \varphi_{\sigma} (\tilde Y_{ij}-aX_{ij} ) } }
=\frac{1}{1+ \frac{d-s_0'}{s_0'}
   \frac{\boldsymbol \varphi_{0,\sigma}(\tilde Y_{ij})}{\boldsymbol \varphi_{a,\sigma}(\tilde Y_{ij}) } },
\end{equation*}
where we abbreviate $\boldsymbol \varphi_{0,\sigma}({\tilde Y_{ij}}):= \boldsymbol \varphi_{\sigma} ( \tilde Y_{ij})$, $\boldsymbol \varphi_{a,\sigma} (\tilde Y_{ij} ):= \boldsymbol \varphi_{\sigma} (\tilde Y_{ij}-aX_{ij} )$, where $\boldsymbol \varphi_{ \sigma}(y) = \frac{1}{\sqrt{2\pi}\sigma} \exp \left( - \frac{y^2}{2\sigma^2} \right)$.
Then, define $\mathcal A := \left\{\mathbf y \in \mathbb R^n : ~\frac{d-s_0'}{s_0'} \frac{\boldsymbol \varphi_{0,\sigma}(\mathbf y)}{\boldsymbol \varphi_{a,\sigma}(\mathbf y) } >1 \right\}$, after some simple calculation we have 
\begin{equation}\label{D1T}
\begin{aligned}
&\inf_{\hat T\in[0,1]^p}\underset{\eta^* \sim \pi_1}{\mathbf E} 
   \underset{Y \sim P_{\beta^*} }{\mathbf E} \|\hat T(Y,X) - \eta^* \|_2^2 \\
%\ge& sd \left\{ \int_{{\bf y} \in \mathbb R^n} \frac{d-s_0'}{d} \left( \frac{1}{1+ \frac{d-s_0'}{s_0'}
 %  \frac{\boldsymbol \varphi_{0,\sigma}({\bf y})}{\boldsymbol \varphi_{a,\sigma}({\bf y})}}
%   \right)^2 \boldsymbol \varphi_{0,\sigma}({\bf y})
%   + \frac{s_0'}{d}\left( \frac{\frac{d-s_0'}{s_0'} \frac{\boldsymbol\varphi_{0,\sigma}({\bf y})}{\boldsymbol \varphi_{a,\sigma}({\bf y})}}
%   {1+\frac{d-s_0'}{s_0'} \frac{ \boldsymbol \varphi_{0,\sigma}({\bf y})}{\boldsymbol \varphi_{a,\sigma}({\bf y})}}\right)^2
%   \boldsymbol \varphi_{a,\sigma}({\bf y}) \mathrm d {\bf y} \right\} \\
\ge& sd\int_{{\bf y}\in \mathcal A} \frac{s_0'}{d} \frac{\frac{d-s_0'}{s_0'} \frac{\boldsymbol\varphi_{0,\sigma}({\bf y})}{\boldsymbol \varphi_{a,\sigma}({\bf y})}}
   {1+\frac{d-s_0'}{s_0'} \frac{ \boldsymbol \varphi_{0,\sigma}({\bf y})}{\boldsymbol \varphi_{a,\sigma}({\bf y})}}
   \boldsymbol \varphi_{a,\sigma}({\bf y}) \mathrm d {\bf y}
    + sd\int_{{\bf y}\in \mathcal A^c} \frac{s_0'}d \frac{\frac{d-s_0'}{s_0'} \frac{\boldsymbol\varphi_{0,\sigma}({\bf y})}{\boldsymbol \varphi_{a,\sigma}({\bf y})}}
   {1+\frac{d-s_0'}{s_0'} \frac{ \boldsymbol\varphi_{0,\sigma}({\bf y})}{\boldsymbol \varphi_{a,\sigma}({\bf y})}}
   \boldsymbol \varphi_{a,\sigma}({\bf y}) \mathrm d {\bf y}\\
\ge&sd \int_{{\bf y}\in \mathcal A} \frac{s_0'}{d} \cdot \frac12 \boldsymbol \varphi_{a,\sigma}({\bf y})\mathrm d {\bf y}
    +sd \int_{{\bf y}\in \mathcal A^c} \frac{d-s_0'}{d} \cdot\frac12 \boldsymbol \varphi_{0,\sigma}({\bf y})\mathrm d {\bf y}\\
=& \frac{sd}{2} \left\{ \frac{s_0'}{d} {\Phi}\left(-\frac{a \sqrt n}{2\sigma}+\frac{\sigma}{a\sqrt n} \log\frac{d-s_0'}{s_0'}\right) \right.\\
  &\qquad \left.  + \frac{d-s_0'}{d} {\Phi}\left(-\frac{a \sqrt n}{2\sigma}-\frac{\sigma}{a\sqrt n} \log\frac{d-s_0'}{s_0'}\right) \right\}\\
=& \frac{s}{2} \psi(d, s_0', a, \sigma),
\end{aligned}
\end{equation}
where the function $\psi$ follows from \eqref{aux2fun}.

Combining \eqref{eq:priork}, \eqref{eq:23term} and \eqref{D1T} together, we have
\begin{equation}\label{Theta1}
\begin{aligned}
& \inf_{\hat \eta \in \{ 0,1\}^p} \sup_{\beta^* \in {\Theta}_{e,1} } \mathbf{E}_{Y\sim P_{\beta^*} } \| \hat \eta(Y,X) - \eta \|_2^2 \\
\ge& \frac{s}{2} \psi(d, s_0', a, \sigma)  -  2s( s_0+s_0') \exp\left(-\frac{3s(s_0-s_0')^2}{2(s_0+2s_0')}\right)  \\
{\ge}& \frac{ss_0'}{2s_0}\psi(d, s_0, a, \sigma)  -  2s( s_0+s_0') \exp\left(-\frac{3s(s_0-s_0')^2}{2(s_0+2s_0')}\right) ,
\end{aligned}
\end{equation}
\begin{comment}
    \begin{equation}\label{Theta1}
\begin{aligned}
& \inf_{\hat \eta} \sup_{\beta \in {\Theta}_1} \mathbf{E}_{Y\sim P_\beta} \| \hat \eta(Y,X) - \eta \|_2^2 \\
%\ge& \inf_{\hat T\in[0,1]^p} \underset{\beta \sim \pi_1}{\mathbf E} \underset{Y \sim P_\beta}{\mathbf E} \|\hat T(Y,X) - \eta\|_2^2 
%    -2 ss_0 \underset{\beta \sim \pi_1}{\mathbf E} \big[ \mathbf1(\beta \notin {\Theta}_1) \big]  -2 \underset{\beta \sim \pi_1}{\mathbf E} \big[ \|\eta\|_1 \mathbf1(\beta \notin {\Theta}_1) \big]\\
%\ge& \frac{sd}{2} \left\{ \frac{d-s_0'}{d} {\Phi}\left(-\frac{a \sqrt n}{2\sigma}-\frac{\sigma}{a\sqrt n} \log\frac{d-s_0'}{s_0'}\right) 
%    + \frac{s_0'}{d} {\Phi}\left(-\frac{a \sqrt n}{2\sigma}+\frac{\sigma}{a\sqrt n} \log\frac{d-s_0'}{s_0'}\right) \right\} \\
%    &-(2s s_0+ ss_0') \exp\left(-\frac{s(s_0-s_0')^2}{2s_0}\right) \\
\ge& \frac{s}{2} \psi(d, s_0', a, \sigma)\\ 
    & -  2s( s_0+s_0') \exp\left(-\frac{3s(s_0-s_0')^2}{2(s_0+2s_0')}\right) 
    - 4s^2s_0 \exp\left(-\frac{15(s_0'-s_0/10)^2}{40s_0' - s_0}\right) \\
{\ge}& \frac{ss_0'}{2s_0}\psi(d, s_0, a, \sigma) \\ 
    & -  2s( s_0+s_0') \exp\left(-\frac{3s(s_0-s_0')^2}{2(s_0+2s_0')}\right) 
    - 4s^2s_0 \exp\left(-\frac{15(s_0'-s_0/10)^2}{40s_0' - s_0}\right), \\
\end{aligned}
\end{equation}
\end{comment}
where the last inequality follows from Lemma \ref{mdPsi}.
\end{proof}

\begin{proof}[Proof of \eqref{allselectorD2}]
Similar to \eqref{eq:priork}, by using the Bayesian risk we have 
\begin{equation}\label{eq:prior2}
\begin{aligned}
&\inf_{\hat \eta_G \in \{0,1\}^m} \sup_{\beta^* \in  {\Theta}_{e,2} } 
    \mathbf{E}_{Y\sim P_{\beta^*} } \|\hat \eta_G - \eta^*_G\|_2^2 \\
\ge &\inf_{\hat T_G \in [0,1]^m} \underset{\eta^* \sim \pi_2}{\mathbf E} \underset{Y \sim P_{\beta^*} }{\mathbf E} \|\hat T_G(Y,X) - \eta^*_G\|_2^2 \\
    &-2 \underset{\eta^* \sim \pi_2}{\mathbf E} \underset{Y \sim P_{\beta^*} }{\mathbf E}  \Bigg\{ \bigg( \underset{\eta^{ {\Theta}_{e,2}}|Y}{\mathbf{E}}   \Big( \|\eta^{ {\Theta}_{e,2}}_G \|_2^2 \big| Y \Big) +
    \| \eta^*_G \|_2^2  \bigg) \mathbf1( a\eta^* \notin {\Theta}_{e,2}) \Bigg\} \\
\ge &\inf_{\hat T_G \in [0,1]^m} \underset{\eta^* \sim \pi_2}{\mathbf E} \underset{Y \sim P_{\beta^*} }{\mathbf E} \|\hat T_G(Y,X) - \eta^*_G\|_2^2\\
    &-2 s \underset{\eta^* \sim \pi_2}{\mathbf P} \left( a\eta^* \notin {\Theta}_{e,2} \right) 
    -2 \underset{\eta^* \sim \pi_2}{\mathbf E} \big( \| \eta^*_G \|_2^2 \mathbf1( a\eta^* \notin {\Theta}_{e,2}) \big),
\end{aligned}
\end{equation}
where for a $\eta^* \sim \pi_2$, we have $\| \eta_G^*\|_2^2 = \sum_{j\in[m]}(\eta_G^*)_j \sim Bin(m,s'/m)$.
By the construction of $\pi_2$ and ${\Theta}_{e,2}$, similar to \eqref{eq:23term}, we have 
$$
2 s \underset{\eta^* \sim \pi_2}{\mathbf P} \left( a\eta^* \notin {\Theta}_{e,2} \right) 
    + 2 \underset{\eta^* \sim \pi_2}{\mathbf E} \big( \| \eta^*_G \|_2^2 \mathbf1( a\eta^* \notin {\Theta}_{e,2}) \big) 
\le2(s+s')\exp\left(-\frac{3(s-s')^2}{2(s+2s')} \right).
$$

For the first term of \eqref{eq:prior2}, we have
\begin{align*}
& \inf_{\hat T_G \in [0,1]^m} \underset{\eta^* \sim \pi_2}{\mathbf E} \underset{Y \sim P_{\beta^*} }{\mathbf E} \|\hat T_G(Y,X) - \eta^*_G\|_2^2\\
\ge &\sum_{j=1}^m \underset{(\eta_{G}^*)_{\setminus j} }{\mathbf E}  \left\{ \inf_{\hat T_{G_j} }~
    \underset{(\eta_{G}^*)_j }{\mathbf E}\underset{Y }{\mathbf E} \left(   |\hat T_{G_j}(Y,X) - (\eta_G^*)_j|^2 \Big| (\eta_{G}^*)_{\setminus j}\right) \right\}\\
\ge &\sum_{j=1}^m \underset{(\eta_{G}^*)_{\setminus j} }{\mathbf E}  \left\{ \inf_{\hat T_{G_j} }~
    \underset{(\eta_{G}^*)_j }{\mathbf E}\underset{Y }{\mathbf E} \left(   |\hat T_{G_j}(\tilde Y_{G_j},X) - (\eta_G^*)_j|^2 \Big| (\eta_{G}^*)_{\setminus j}\right) \right\}\\ 
= &  \sum_{j=1}^m \underset{(\eta_{G}^*)_{\setminus j} }{\mathbf E} 
    \left\{ \inf_{\hat T_{G_j} }~ \left[ \frac{m-s'}{m}
    \underset{\tilde Y_{G_j}|(\eta_G^*)_j = 0}{\mathbf E} \left(
    \hat T_{G_j}^2 \Big| (\eta_{G}^*)_{\setminus j} \right)  \right.\right. \\
    & \qquad\qquad\qquad\qquad \left.\left.+ \frac{s^\prime}{m}~\underset{\tilde Y_{G_j}| (\eta_G^*)_j = 1 }{\mathbf E}
    \left( (1-\hat T_{G_j})^2 \Big| (\eta_{G}^*)_{\setminus j} \right) \right] \right\}\\
= &  \sum_{j=1}^m \underset{(\eta_{G}^*)_{\setminus j} }{\mathbf E} 
    \left\{ \frac{m-s'}{m}
    \underset{\tilde Y_{G_j}|(\eta_G^*)_j = 0}{\mathbf E} \left(
    (\hat T_{G_j}^*)^2 \Big| (\eta_{G}^*)_{\setminus j} \right)  \right.  \\
    & \qquad\qquad \quad \left.+ \frac{s^\prime}{m}~\underset{\tilde Y_{G_j}| (\eta_G^*)_j = 1 }{\mathbf E}
    \left( (1-\hat T_{G_j}^*)^2 \Big| (\eta_{G}^*)_{\setminus j} \right)  \right\},
\end{align*}
where in the second inequality, we use
\begin{equation*}
\tilde Y_{G_j}:= Y - \sum_{k \ne j}\sum_{i \in[s_0]} \beta_{ik}^* X_{(ik)} = \sum_{i\in[s_0]} a \eta_{ij}^* X_{(ij)} + \sigma \xi \in \mathbb R^n
\end{equation*}
to represent the marginal observation under the information of the other groups.
And in the last equality, we achieve the infimum by using the selector
\begin{equation*} 
\hat T_{G_j}^*(\tilde Y_{G_j},X) 
:=  \frac1{ 1+ \frac{m-s'}{s'} \frac {\boldsymbol \varphi_{\sigma}(\tilde Y_{G_j})}
    {\boldsymbol \varphi_{\sigma} \left( \tilde Y_{G_j}-\sum_{i\in [s_0]}a X_{(ij)}\right)} }.
\end{equation*}

Therefore, leveraging a technique similar to \eqref{D1T}, we obtain
\begin{equation}\label{D2T}
\begin{aligned}
& \inf_{\hat \eta_G \in \{0,1\}^m} \sup_{\beta^* \in  {\Theta}_{e,2}} 
    \mathbf{E}_{Y\sim P_{\beta^*} } \|\hat \eta_G(Y,X) - \eta^*_G\|_2^2\\
\ge& \frac{s'}{2} \left\{ \frac{m-s'}{s'}{\Phi}\left(-\frac{a \|\sum_{i\in[s_0]}X_{(ij)}\|_2}{2\sigma}-\frac{\sigma \log(m/s'-1)}{a\|\sum_{i\in[s_0]}X_{(ij)}\|_2 } \right)  \right.\\
  & \quad\quad\quad \left.+ {\Phi}\left(-\frac{a \|\sum_{i\in[s_0]}X_{(ij)}\|_2 }{2\sigma} + \frac{\sigma \log(m/s'-1)}{a \|\sum_{i\in[s_0]}X_{(ij)}\|_2 } \right) \right\}\\ 
\ge & \frac{s'}{2} \left\{ \frac{m-s}{s}{\Phi}\left(-\frac{a \|\sum_{i\in[s_0]}X_{(ij)}\|_2}{2\sigma}-\frac{\sigma \log(m/s-1)}{a\|\sum_{i\in[s_0]}X_{(ij)}\|_2 } \right)  \right.\\
  & \quad\quad\quad \left.+ {\Phi}\left(-\frac{a \|\sum_{i\in[s_0]}X_{(ij)}\|_2 }{2\sigma} + \frac{\sigma \log(m/s-1)}{a \|\sum_{i\in[s_0]}X_{(ij)}\|_2 } \right) \right\} ,
\end{aligned}
\end{equation}
where the last inequality follows Lemma \ref{mdPsi}.

Combining \eqref{eq:prior2} and \eqref{D2T} together, we have
\begin{equation}\label{eq: thetae2 general lb}
\begin{aligned}
& \inf_{\hat \eta} \sup_{\beta \in {\Theta}_{e,2}} \mathbf{E}_{Y\sim P_\beta} \|\hat \eta(Y,X) - \eta \|_2^2 \\
\ge& \frac{ s'}{2s}\left\{  (m-s) {\Phi}\left(-\frac{a \|\sum_{i\in[s_0]}X_{(ij)}\|_2 }{2\sigma} -
     \frac{\sigma \log(m/s-1)}{a \|\sum_{i\in[s_0]}X_{(ij)}\|_2 } \right) \right.  \\
     & \qquad\quad \left. +s {\Phi}\left(-\frac{a \|\sum_{i\in[s_0]}X_{(ij)}\|_2 }{2\sigma} + 
     \frac{\sigma \log(m/s-1)}{a \|\sum_{i\in[s_0]}X_{(ij)}\|_2 } \right) \right\} \\
    &- 2(s+s')\exp\left(-\frac{3(s-s')^2}{2(s+2s')} \right).
\end{aligned}
\end{equation}
Therefore, we complete the proof of Theorem \ref{allselectorD1D2}.
\end{proof}

\subsection{Proof of Theorem \ref{LB2priors}}\label{proof of th 3.2}
We next prove \eqref{eq: elementwise LB 1} and \eqref{eq: elementwise LB 2}, respectively.

\paragraph{Proof of \eqref{eq: elementwise LB 1}.}
Assume { $ss_0 \ge 54$ and $a^2 = \frac{\sigma^2}{5n} \log(ds-ss_0)$.}
We stress that the constants appearing in this proof may not be optimal.
Regardless, the existence of these lower bounds is assured.
Define $B:= \frac{\log(ss_0)}{10}+ \log(d/s_0 -1)$, by $a\le \frac{\sqrt2\sigma}{\sqrt n} \sqrt{B+ \sqrt{B^2 - \log^2(d/s_0-1)}}$, we can calculate that
\begin{equation}\label{eq: y}
\begin{aligned}
& -\frac{a\sqrt n}{2\sigma} +\frac{\sigma\log(d/s_0-1)}{a\sqrt n} \\
\ge&  \frac{\sqrt{B- \sqrt{B^2 - \log^2(d/s_0-1)}} - \sqrt{B+ \sqrt{B^2 - \log^2(d/s_0-1)}}}{\sqrt2} \\
=&-\sqrt{B - \log(d/s_0-1) }\\
=& -\sqrt{\frac{\log(ss_0)}{10} },
\end{aligned}
\end{equation}
where the first inequality holds because the function $f(x) = -\frac{x}{2\sigma} +\frac{\sigma\log(d/s_0-1)}{x} $ is monotonically decreasing for $x > 0 $.

Take $s_0' = s_0/2$. Then, by combining \eqref{eq: y} with \eqref{allselectorD1}, we have
\begin{equation}
\begin{aligned}
& \inf_{\hat \eta\in \{0,1\}^p} \sup_{\beta^* \in {\Theta}_{e,1} } \mathbf{E}_{Y\sim P_{\beta^*} }\| \hat \eta(Y,X) - \eta^* \|_2^2 \\
\ge& \frac{ss_0 }{4}{\Phi}\left( -\frac{a \sqrt n}{2\sigma} +\frac{\sigma\log(d/s_0-1)}{a \sqrt n}\right) 
    -3 s s_0 \exp\left(-\frac{3s s_0 }{16}\right) \\
\overset{(i)}\ge& \frac{(ss_0)^{19/20}}{8} \frac{\sqrt{2/\pi}}{1+\sqrt{(\log(ss_0))/10}} 
    -3 (s s_0)^{4/5} \cdot (ss_0)^{1/5} \exp\left(-\frac{3s s_0 }{16}\right)\\
\overset{(ii)}\ge& \frac{1}{9} (ss_0)^{\frac45} 
    - \frac{1}{90} (ss_0)^{\frac45} \\
=& \frac{1}{10} (ss_0)^{\frac45},
\end{aligned}
\end{equation}
where inequality (i) follows from \eqref{eq: y} and ${\Phi}(-y) \ge \sqrt{\frac2\pi} \frac{\exp(-y^2/2)}{2(y+1)}$ for every $y>0$, and inequality (ii) follows from the assumption $ss_0 \ge 54$.

\paragraph{Proof of \eqref{eq: elementwise LB 2}.}
Using a technique similar to \eqref{eq: y}, we can also demonstrate that, if $ a^2 s_0 \le \frac{\sigma^2 \log(m-s)}{5(1+\delta)n}$, then
\begin{equation}\label{eq: lower bound subspace 2}
\begin{aligned}
&\inf_{\hat \eta_G \in \{0,1\}^m} \sup_{\beta^* \in  {\Theta}_{e,2}} 
    \mathbf{E}_{Y\sim P_{\beta^*} } \left\{ \sum_{j\in [m]} \left| (\hat \eta_G)_j(Y,X) - (\eta^*_G)_j \right| \right\} \\
\ge& \frac{s}{4} \Phi\left( -\sqrt{\frac{\log s}{10}}\right) - 3s \exp\left(-\frac{3s}{16} \right)\\
\ge& \frac{1}8s^{\frac7{10}} - s^{\frac7{10}} \times 3 s^{\frac3{10}} e^{-3s/16}\\
\ge& \frac1{20}s^{\frac7{10}},
\end{aligned}
\end{equation}
where the last two inequalities follow from $s \ge 25$.

Combining the results derived from ${\Theta}_{e,1}$ and ${\Theta}_{e,2}$, if
\begin{align*}
a^2 \le& \frac{\sigma^2}{10n} \left( \frac{\log(m-s)}{s_0(1+\delta)} + \log(sd-ss_0)\right)\\
\le& \max\left(\frac{\sigma^2\log(m-s)}{5 s_0 n(1+\delta)} , \frac{\sigma^2\log(sd-ss_0)}{5n} \right),
\end{align*}
then it is impossible for both element-wise and group-wise selection to be consistent---at least one level of selection is unattainable.
Therefore, we conclude the proof of Theorem \ref{LB2priors}.

{
\subsection{Proof of Theorem \ref{th: LB group}}
}
\paragraph{Preliminary}
The group-wise signal analysis is quite similar to the proof of Theorem \ref{allselectorD1D2} and \ref{LB2priors}. We construct two subspaces of $\Theta_g(s,s_0,b)$ as
$$
\begin{aligned}
\Theta_{g,1} &:= \left\{ \beta \in \Theta\left(s,s_0 \right) ~\left|~
\begin{aligned}
    & G^*(\beta) = [s], \\
    & \| \beta_{G_j} \|_0 \ge {s_0}/{10}~ \text{ for every } j \in [s],\\
    & \beta_{ij} = a ~\text{ for every } (i,j) \in \operatorname{supp}(\beta)
\end{aligned} \right.\right\},\\
\Theta_{g,2} &:= \left\{ \beta \in \Theta\left(s,s_0 \right) ~\left|~
\beta_{ij} =\begin{cases}
 b/\sqrt{s_0}, & \text{ if } i \in [s_0] \text{ and } j \in G^*(\beta) \\
 0, & \text{ otherwise}
\end{cases} \right.\right\}.
\end{aligned}
$$
where in $\Theta_{g,1}$, the parameter $a$ satisfies that $\| \beta_{G_j}\|_2= \sqrt{a^2 \| \beta_{G_j}\|_0} \ge b$ for every $j \in [s]$.
The proof of Theorem \ref{th: LB group} uses Lemma \ref{lemma:minimaxori} with the priors $\pi_1$ and $\pi_2$ defined in Appendix \ref{sec: pi12}. 

\paragraph{Proof of \eqref{eq: elementwise LB 4}.}
We first concentrate on the lower bound for group-wise support recovery over $\Theta_{g,2}$. 
Since $\Theta_{g,2}$ and $\Theta_{e,2}$ share the same structural properties, the arguments in \eqref{eq:prior2}, \eqref{D2T}, \eqref{eq: thetae2 general lb}, and \eqref{eq: lower bound subspace 2} carry over directly.
In particular, taking $s' = s/2 \ge 12.5$ and $b^2 \le \frac{\sigma^2 \log(m-s)}{5(1+\delta)n} $, we obtain
\begin{equation*}\label{eq: lower bound subspace 3}
\inf_{\hat \eta_G \in \{0,1\}^m} \sup_{\beta^* \in  {\Theta}_{g,2}} \mathbf{E}_{Y\sim P_{\beta^*} } \left\{ \sum_{j\in [m]} \left| (\hat \eta_G)_j(Y,X) - (\eta^*_G)_j \right| \right\}  
\ge \frac1{20}s^{\frac7{10}}.
\end{equation*}

\paragraph{Proof of \eqref{eq: elementwise LB 3}.}
We then focus on the element-wise selection error in $\Theta_{g,1}$. 
With the prior $\pi_1$ and $1\le s_0'\le s_0$, we have 
\begin{equation}
\begin{aligned}
&\inf_{\hat \eta \in \{0,1\}^{d\times m}} \sup_{\beta^* \in  {\Theta}_{g,1}} 
    \mathbf{E}_{Y\sim P_{\beta^*} } \|\hat \eta - \eta^*\|_2^2 \\ 
\ge &\inf_{\hat T } ~\underset{\eta^* \sim \pi_1}{\mathbf E}~ \underset{Y \sim P_{\beta^*} }{\mathbf E}  \|\hat T(Y,X) - \eta^* \|_2^2 \\
    &-2 ss_0 \underset{\eta^* \sim \pi_1}{\mathbf P} \left( a\eta^* \notin {\Theta}_{g,1} \right) 
    -2 \underset{\eta^* \sim \pi_1}{\mathbf E} \big( \| \eta^*\|_2^2 \mathbf1( a\eta^* \notin {\Theta}_{g,1}) \big)\\
\ge & \frac{ss_0'}{2s_0}\left\{ (d-s_0) {\Phi}\left(-\frac{a \sqrt n}{2\sigma}-\frac{\sigma\log(d/s_0-1)}{a\sqrt n} \right) \right.\\ 
 & \qquad~~ \left. + s_0 {\Phi}\left(-\frac{a \sqrt n}{2\sigma}+\frac{\sigma \log(d/s_0-1)}{a\sqrt n} \right) \right\} \\
    & -  2s( s_0+s_0') \exp\left(-\frac{3s(s_0-s_0')^2}{2(s_0+2s_0')}\right) - 4s^2s_0 \exp\left(-\frac{15(s_0'-s_0/10)^2}{40s_0' - s_0}\right),
\end{aligned}
\end{equation}
where the first inequality follows from \eqref{eq:priork}, and in the last inequality, the first term (Bayesian risk) is derived from \eqref{D1T} and Lemma \ref{mdPsi}, and the last two terms are from the following:
\begin{equation*} 
   \begin{aligned}
&2ss_0 \mathbf P_{\eta^* \sim \pi_1} \Big( a\eta^* \notin {\Theta}_{g,1} \Big) + 2\mathbf E_{\eta^* \sim \pi_1}  \Big( \|\eta^*\|_2^2 \mathbf1(a\eta^* \notin {\Theta}_{g,1}) \Big)\\
=& 2ss_0 \Big\{ \mathbf P \left( v > ss_0\right) + \mathbf P \left( v \le ss_0, \text{ and } v_j < s_0 /10 \text{ for some } j \in [s]\right)\Big\}\\
& + 2\mathbf E \Big\{ v \Big[ \mathbf1(v  >ss_0)  + \mathbf1(v \le ss_0, \text{ and } v_j < s_0 /10 \text{ for some } j \in [s]) \Big]  \Big\} \\
\le& 2ss_0 \mathbf P \left( v > ss_0\right) + 2\mathbf E \left\{ v \mathbf1(v  >ss_0) \right\}  + 4ss_0 \sum_{j\in[s]} \mathbf P\left( v_j < s_0/10 \right)\\
\le& 2s (s_0+s_0') \exp\left(-\frac{3s(s_0-s_0')^2}{2(s_0+2s_0')}\right)
   +4s^2s_0 \exp\left(-\frac{15(s_0'-s_0/10)^2}{40s_0' - s_0}\right).
 %+ (2s s_0+ 2ss_0') \exp\left(-\frac{s(s_0-s_0')^2}{2s_0}\right) ,
\end{aligned}
\end{equation*}

We now take $s_0'=2s_0/3$. Following the similar proof technique used in the first part of Appendix \ref{proof of th 3.2}, as $a^2 = \frac{\sigma^2 \log(ds-ss_0)}{10n}$ (thus leading $b^2\le \frac{\sigma^2 s_0 \log(ds-ss_0)}{100n}$), $ss_0 \ge 87$, and $s^{6/5} \le \frac7{200} s_0^{-1/5} \exp(0.1876s_0)$, we have
\begin{equation*}
\begin{aligned}
& \inf_{\hat \eta\in \{0,1\}^{d\times m}} \sup_{\beta^* \in {\Theta}_{g,1} } \mathbf{E}_{Y\sim P_{\beta^*}} \| \hat \eta(Y,X) - \eta^* \|_2^2 \\
\ge& \frac{(ss_0)^{39/40}}{6} \frac{\sqrt{2/\pi}}{1+\sqrt{(\log(ss_0))/20}} 
    - \frac{10}3s s_0\exp\left(-\frac{s s_0 }{14}\right)
    - 4s^2 s_0\exp\left(-0.1876s_0 \right)\\
\ge& \frac{1}{6} (ss_0)^{\frac45} - 0.016 (ss_0)^{\frac45} - 0.14(ss_0)^{\frac45} \\
> & \frac{1}{100} (ss_0)^{\frac45}.
\end{aligned}
\end{equation*}

Therefore, we complete the proof of Theorem \ref{th: LB group}.

\section{Auxiliary Lemmas for the oracle properties}\label{appC}
Firstly, we introduce a useful lemma from Theorem 2.1 in \cite{HK12}.
The conclusion of this lemma is not limited to the true sparsity level $(s,s_0)$; in fact, it can be applied to any $0< s' < m $ and $0<s_0' < d/e$. 
\begin{lemma}\label{good event}
Assume that $X\in \mathbb R^{n \times p}$ satisfies { DSRIP$( s',s_0', \delta )$ with $\delta  \in (0,1) $.}
For all $ k \in [n]$, assume that each $\xi_k$ is independent sub-Gaussian random variable with zero mean and $\|\xi\|_{\psi_2}^2 \le 2$. 
%Recall $X_{(ij)}\in \mathbb R^n$ denotes the observation vector of the $i$-th variable in the $j$-th group.% and satisfies $\|X_{(ij)}\|_2 = \sqrt n$.
Define $\Xi_{ij} = \frac\sigma n X_{(ij)}^\top \xi$ and assume that $d \ge e s_0'$.
Then the event
\begin{equation*}
   \mathcal E(s',s_0'):= \left\{ \text{For all }S \in \mathcal S(s',s_0'),~\sum_{(i,j) \in S} \Xi_{ij}^2 < \frac{10\sigma^2s's_0'\Delta(s',s_0')}n \right\}
\end{equation*}
holds with probability greater than $1-\exp\left(-\frac13 s's_0' \Delta(s',s_0')\right)$, where
\begin{equation*}
    \Delta(s',s_0') := \frac1{s_0'} \log \frac{{\rm e} m}{s'} + \log \frac{{\rm e} d}{s_0'}.
\end{equation*}
\end{lemma}

\begin{proof}[{Proof of Lemma \ref{good event}}]
%\begin{proof}[Proof of Lemma \ref{good event}]
%Based on the properties of sub-Gaussian random vector $\xi$, we have $\| \Xi_{ij}\|_{\psi_2}^2 \le \frac{2\sigma^2}{n}$.
From Theorem 2.1 in \cite{HK12}, for all $t >0$ and all $ S \in \mathcal S(s',s_0')$, we have 
    \begin{equation}\label{eq:Hsu}
        P\left( \|X_S^\top\xi\|_2^2 \ge \text{tr}(X_S X_S^\top)+
     2\sqrt{\text{tr}(X_S X_S^\top X_S X_S^\top)t} + 2 \|X_S X_S^\top \|_2 t \right)
     \le e^{-t},
    \end{equation}
    where we also use $\| \cdot \|_2$ to denote the spectral norm of a matrix.
    Based on DSRIP$(s,s_0, \delta)$ condition, we have $
   n(1-\delta) \le \left\| X_S^\top {X}_{S} \right\|_2 \le n(1+\delta)$,
which leads
\begin{align*}
\text{tr}(X_S X_S^\top) &= \text{tr}(X_S^\top X_S ) \le 2s's_0'n,\\
\text{tr}(X_S X_S^\top X_S X_S^\top) &= \text{tr}(X_S^\top X_S \cdot X_S^\top X_S) \le  |S| \cdot \|X_S^\top X_S \|_2^2 \le 4 n^2 s's_0', \\
\|X_S X_S^\top \|_2 &=\|X_S^\top X_S \|_2 \le 2n.
\end{align*}

Back to \eqref{eq:Hsu}, we have
     \begin{equation}\label{ineq:t}
         P\left( \|X_S^\top\xi\|_2^2 \ge 2s's_0'n +
     4n \sqrt{s's_0't} + 4n t \right) \le e^{-t}.
    \end{equation}
Then take $t =\frac43 s's_0'\Delta(s',s_0')$, we obtain
\begin{align*}
        2s's_0'n +
     4n \sqrt{s's_0't} + 4n t  < 10ns's_0' \Delta(s',s_0').
\end{align*}
Therefore, for a fixed $ S \in \mathcal S (s',s_0')$ satisfying $|S| = s's_0'$, we have
    \begin{equation*}
         P\left( \sum_{(i,j) \in S}\Xi_{ij}^2 \ge \frac{ 10 \sigma^2}{n}s's_0'\Delta(s',s_0') \right)
        \le \exp \left( -\frac43 s's_0'\Delta(s',s_0') \right).
    \end{equation*}

Finally, by the probability union bound, we have 
\begin{align*}
& P\left( \forall S \in \mathcal S (s',s_0'), ~\sum_{(i,j)\in S} \Xi_{ij}^2 < \frac{10 \sigma^2s's_0'\Delta(s',s_0') }n \right) \\
= & 1 - P \left( \bigcup_{S \in \mathcal S (s',s_0'):~ |S|= s's_0} \left\{  ~ \sum_{(i,j)\in S} \Xi_{ij}^2 \ge \frac{ 10 \sigma^2s's_0'\Delta(s',s_0') }n \right\} \right)\\
{\ge} & 1 - \binom m{s'} \binom{ds'}{s's_0'} \exp \left( -\frac43 s's_0' \Delta(s',s_0') \right)\\
\ge& 1 - \exp \left( -\frac13 s's_0'\Delta(s',s_0') \right),\\
\end{align*} 
where the last inequality follows from $\binom yx \le (ey/x)^x$ for every $0<x<y$.
Thus, we complete the proof of Lemma \ref{good event}.
\end{proof}

\begin{comment}
Note that 
    \begin{equation}\label{covering}
        |\mathcal S (s',s_0')| = \binom{m}{s'} \binom{ds'}{s's_0'}
        \le \left( \frac{{\rm e} m}{s'}\right)^{s'} \left( \frac{{\rm e}ds'}{s's_0'}\right)^{s's_0'} 
        =\exp\left( s's_0'\Delta(s',s_0') \right),
    \end{equation}
\end{comment}

\subsection{Lemma \ref{exactOGChi2} and its proof}
{
Now we focus on some inequalities used in the proof of Proposition \ref{T8}. 
Specifically, the following lemma bounds the squared sum of $\tilde \Xi$ on some $S_{OG}\in \mathcal S(As,s_0)$.
Recall that $\tilde \Xi := \frac\sigma n X^\top \left(\mathbf I_n - X_{S^*} (X_{S^*}^\top X_{S^*})^{-1}X_{S^*}^\top \right) \xi \in \mathbb R^p$, $A = \frac{8\delta^2}{(\kappa- \delta)^2} $, $\mu' = \sigma  \sqrt{ \frac{10}{n} \left[ \Delta(1,s_0) + \log(ss_0) \right] }$, and $\mu \ge \sqrt{4+ \frac{12 \delta^2}{(1-\delta)^2} } \mu'$.

\begin{lemma}\label{exactOGChi2}
	Under conditions of Proposition \ref{T8}, for any falsely discovered set $S_{OG} \in \mathcal S(As,s_0)$ which does not contain any support group, we have 
	\begin{align*}
		& P\left(  \sum_{(i,j) \in S_{OG}} \tilde \Xi_{ij}^2 \mathbf1\Big\{ \mathcal T_{\mu,s_0} \big(\tilde H^{t+1}\big)_{ij}\ne 0\Big\} 
		\le   \frac{(1-\delta)^2}{3} \|\tilde \beta^t - \tilde \beta^* \|_2^2 \right) \\
		\ge& 1- \exp\left( -\frac13 s_0 \Delta(1,s_0) \right).
	\end{align*} 
\end{lemma}

\begin{proof}[Proof of Lemma \ref{exactOGChi2}]
	Based on the definition of the double sparse operator $\mathcal T_{\mu,s_0}$ and $\mathcal T_{\mu',s_0}$, we have 
	\begin{equation*} 
		\begin{aligned}
			&   \sum_{(i,j) \in S_{OG}} \tilde \Xi_{ij}^2 \mathbf1\Big\{ \mathcal T_{\mu,s_0} 
			\big(\tilde H^{t+1}\big)_{ij}\ne 0\Big\} \\
			= &   \sum_{(i,j) \in S_{OG}} \tilde\Xi_{ij}^2 \mathbf1\Big\{\mathcal T_{\mu',s_0} \big(\tilde \Xi \big)_{ij}\ne 0, ~ \mathcal T_{\mu,s_0} \big(\tilde H^{t+1}\big)_{ij}\ne 0\Big\}\\
			& + \sum_{(i,j) \in S_{OG}} \tilde \Xi_{ij}^2 \mathbf1\Big\{\mathcal T_{\mu',s_0} 
			\big(\tilde \Xi \big)_{ij}= 0,~ \mathcal T_{\mu,s_0} \big(\tilde H^{t+1}\big)_{ij}\ne 0\Big\} \\
			\le & \sum_{(i,j) \in S_{OG}} \tilde \Xi_{ij}^2 \mathbf1\Big\{\mathcal T_{\mu',s_0} \big( \tilde\Xi \big)_{ij}\ne 0 \Big\} \\
			&+ \sum_{(i,j) \in S_{OG}} \tilde \Xi_{ij}^2\mathbf1\Big\{|\tilde\Xi_{ij}|< \mu',
			~ \mathcal T_{\mu,s_0} \big(\tilde H^{t+1}\big)_{ij}\ne 0\Big\} \\ 
			& + \sum_{(i,j) \in S_{OG}} \tilde \Xi_{ij}^2\mathbf1\Bigg\{|\tilde \Xi_{ij}|\ge \mu', \\
			& \qquad \qquad \qquad \qquad \sum_{k\in[d]} \tilde \Xi_{kj}^2 \mathbf1 (|\tilde \Xi_{kj}|\ge \mu')< s_0\mu'^2,~
			\mathcal T_{\mu,s_0} \big(\tilde H^{t+1}\big)_{ij}\ne 0\Bigg  \}.  \\ 
		\end{aligned}
	\end{equation*}
	
	Then, by combining Lemma \ref{exactOGChi2:1st}, \ref{OGChi2:2nd} and \ref{OGChi2:3rd} together, we complete the proof of Lemma \ref{exactOGChi2}.    
\end{proof}

\begin{lemma}\label{exactOGChi2:1st}
	Under all conditions of Proposition \ref{T8}, we have
	\begin{equation*} 
		P \left( \sum_{(i,j) \in [d] \times [m]} \mathbf1\Big\{\mathcal T_{\mu',s_0} \big(\tilde \Xi \big)_{ij} \ne 0 \Big\} = 0 \right) \ge 1- \exp\left( -\frac13 s_0 \Delta(1,s_0) \right),
	\end{equation*}
	that is, no element or group of $\tilde \Xi $ can be selected by the operator $ \mathcal T_{\mu',s_0} = \mathcal T_{\mu',s_0}^{(2)} \circ \mathcal T_{\mu'}^{(1)}$.
\end{lemma}

\begin{proof}[Proof of Lemma \ref{exactOGChi2:1st}]
	Firstly note that $\| \tilde \Xi_{ij} \|_{\psi_2}^2 \le \frac{2\sigma^2}{n} $.
	Then we take $s' = 1$, $s_0'=s_0$ and conclude that event $\mathcal E(1,s_0)$ (defined in Lemma \ref{good event}) holds with probability greater than $1- \exp\left(-\frac13 s_0 \Delta(1,s_0) \right)$.
	
	%We use $G_{OG}$ to represent the group index set of the falsely discovered set $S_{OG}$. 
	Based on the event $\mathcal E(1,s_0)$, we next prove that no group in $\tilde \Xi $ can be discovered under $ \mathcal T_{\mu',s_0}$ by using contradiction. 
	Specifically, if there exists a group $G_{j_0}$ satisfies $\mathcal T_{\mu',s_0}(\tilde \Xi_{G_{j_0}} ) \ne \mathbf 0_d $, we can separate the proof into two cases:
	\begin{enumerate}
		\item If more than $s_0$ entries in $\tilde \Xi_{G_{j_0}}$ are discovered, we can select arbitrary $s_0$ discovered entries from $\tilde \Xi_{G_{j_0}}$, and use $S'$ to denote their index set.
		Then, by the definition of the element-wise thresholding operator $\mathcal T_{\mu'}^{(1)}$ we obtain
		\begin{equation*} 
			\sum_{(i,j) \in S'} \tilde \Xi_{ij}^2  
			\ge  s_0 \mu'^2 
			>\frac{10\sigma^2 s_0 \Delta(1,s_0)}{n},
		\end{equation*}
		which contradicts the event $\mathcal E(1,s_0)$.
		
		\item If only less than $s_0$ entries in $\tilde \Xi_{ G_{j_0}}$ is discovered, we select all discovered entries from $\tilde \Xi_{ G_{j_0}}$, and use $S''$ to denote their index set.
		Then, by the definition of the group-wise thresholding operator $\mathcal T_{\mu',s_0}^{(2)}$ we obtain
		\begin{equation*} 
			\sum_{(i,j) \in S''} \tilde \Xi_{ij}^2  
			\ge  s_0 \mu'^2 
			>\frac{10\sigma^2 s_0 \Delta(1,s_0)}{n},
		\end{equation*}
		which contradicts the event $\mathcal E(1,s_0)$ again.
	\end{enumerate}
	
	Therefore, we prove that no group in $\tilde \Xi$ can be selected, which completes the proof of Lemma \ref{exactOGChi2:1st}.
\end{proof}

\begin{lemma}\label{OGChi2:2nd}
	Under conditions of Proposition \ref{T8}, for any falsely discovered set $S_{OG} \in \mathcal S(As,s_0)$ which does not contain any support group
	%, as $\Delta(s,s_0), \frac{ss_0}{\Delta(s,s_0)} \to \infty$ and $\lambda_1^2 = \frac{8\sigma^2}{n}\Delta(s,s_0)$
	, we have 
	\begin{equation}\label{eq:exactOGChi2:2nd}
		\sum_{(i,j) \in S_{OG}} \tilde \Xi_{ij}^2\mathbf1\Big\{| \tilde\Xi_{ij}|< \mu',
		~ \mathcal T_{\mu,s_0} \big(\tilde H^{t+1}\big)_{ij}\ne 0\Big\}
		\le  \frac{(1-\delta)^2}{6} \left\|  \tilde \beta^* -\tilde \beta^t \right\|_2^2.
	\end{equation} 
\end{lemma}

\begin{proof}[Proof of Lemma \ref{OGChi2:2nd}]
	
	%\begin{proof}[\textbf{Proof of Lemma \ref{OGChi2:2nd}}]
	By the decomposition $\tilde H_{ij}^{t+1} = \langle{\Phi}_{(i,j)}^\top, \tilde \beta^* - \tilde\beta^{t} \rangle + \tilde \Xi_{ij}$ for every $ (i,j) \in (S^*)^c$ and the DSRIP$\left( (1+2A)s,\frac{1+4A}{1+2A}s_0,\delta \right)$ condition, we have
	\begin{equation}
		\begin{aligned}\label{ElementXifailHsuccess}
			& \sum_{(i,j) \in S_{OG}} \tilde \Xi_{ij}^2\mathbf1\Big\{|\tilde \Xi_{ij}|< \mu',
			~ \mathcal T_{\mu,s_0} \big(\tilde H^{t+1}\big)_{ij}\ne 0\Big\}\\
			\le & \sum_{(i,j) \in S_{OG}} \mu'^2\mathbf1\Big\{|\tilde \Xi_{ij}|< \mu',  ~ |\tilde \Xi_{ij}+ \langle{\Phi}_{(i,j)}^\top, \tilde \beta^* - \tilde\beta^{t} \rangle| \ge \mu \Big\}\\
			\le & \sum_{(i,j) \in S_{OG}} \mu'^2\mathbf1\Big\{|\tilde \Xi_{ij}|< \mu',  ~ |\langle{\Phi}_{(i,j)}^\top, \tilde \beta^* - \tilde\beta^{t} \rangle| \ge \mu -\mu' \Big\}\\
			\overset{(i)}\le & \sum_{(i,j) \in S_{OG}} \mu'^2 \mathbf1\left\{ \mu' \le \frac{1-\delta}{\sqrt6 \delta} |\langle{\Phi}_{(i,j)}^\top, \tilde \beta^* - \tilde\beta^{t} \rangle|  \right\}\\
			\le &\frac{(1-\delta)^2}{6\delta^2} \sum_{(i,j) \in S_{OG} } \langle {\Phi}_{(i,j)}^\top, \tilde \beta^* -\tilde \beta^t \rangle^2
            ~\le~\frac{(1-\delta)^2}6 \left\|  \tilde \beta^* -\tilde \beta^t \right\|_2^2,
		\end{aligned}
	\end{equation}
where inequality (i) follows from $\mu > \left(1+ \frac{\sqrt6 \delta}{1-\delta}\right)\mu'$.
Therefore we complete the proof of Lemma \ref{OGChi2:2nd}. 
\end{proof}

\begin{lemma}\label{OGChi2:3rd}
	Under conditions of Proposition \ref{T8}, for any falsely discovered set $S_{OG} \in \mathcal S(As,s_0)$ which does not contain any support group, we use $G_{OG}$ to represent the group index set of $S_{OG}$. 
	Then we have
	\begin{align*}\label{eq:exactOGChi2:3rd}
		\sum_{(i,j) \in S_{OG}} \tilde \Xi_{ij}^2\mathbf1 &\Bigg\{|\tilde \Xi_{ij}|\ge \mu',
		~ \sum_{k\in [d]} \tilde \Xi_{kj}^2 \mathbf1 (|\tilde \Xi_{kj}|\ge \mu')< s_0\mu'^2,~ \mathcal T_{\mu,s_0} \big(\tilde H^{t+1}\big)_{ij}\ne 0 \Bigg\}\\ 
		\le&  \frac{(1-\delta)^2}6 \left\| \tilde \beta^t - \tilde \beta^* \right\|_2^2. 
	\end{align*} 
\end{lemma}

\begin{proof}[Proof of Lemma \ref{OGChi2:3rd}]
	%\begin{proof}[\textbf{Proof of Lemma \ref{OGChi2:3rd}}]
	For any non-support group $G_j$ satisfies that
	$$\mathcal T_{\mu',s_0} \big( \tilde \Xi_{ G_j}\big)=\mathbf 0_d,
	\quad \mathcal T_{\mu,s_0} \big(\tilde H^{t+1}_{S_{OG} \cap G_j}\big) \ne \mathbf 0,$$
	by the double sparse operator $\mathcal T_{\mu,s_0}$ and $\mu \ge\sqrt{ 4+ \frac{12 \delta^2}{(1-\delta)^2}} \mu'$, %(which can be proved by the construction of step 1 in Proposition \ref{T6}), 
	we have 
	\begin{equation}\label{OGgroupscaled}
		\begin{aligned}
			s_0\mu^2  
	\le& \sum_{k: (k,j) \in S_{OG}} \Big( \underbrace{\tilde \Xi_{kj} + \langle {\Phi}_{(k,j)}^\top, \tilde \beta^* -\tilde \beta^t \rangle}_{\tilde H_{kj}^{t+1}} \Big)^2  \mathbf1\Big( |\tilde \Xi_{kj} + \langle {\Phi}_{(k,j)}^\top, \tilde \beta^* -\tilde \beta^t \rangle| \ge \mu \Big)\\
	\le& \sum_{k: (k,j) \in S_{OG}} 2 \langle {\Phi}_{(k,j)}^\top, \tilde \beta^* -\tilde \beta^t \rangle^2 \mathbf1\Big( |\tilde \Xi_{kj}+\langle {\Phi}_{(k,j)}^\top, \tilde \beta^* -\tilde \beta^t \rangle |\ge\mu\Big)\\
			& +\sum_{k: (k,j) \in S_{OG}} 2 \tilde \Xi_{kj}^2 \mathbf1\Big( |\tilde \Xi_{kj}+\langle {\Phi}_{(k,j)}^\top, \tilde \beta^* -\tilde \beta^t \rangle |\ge\mu\Big) \\
	\le& 2 \sum_{k: (k,j) \in S_{OG}} \langle {\Phi}_{(k,j)}^\top, \tilde \beta^* -\tilde \beta^t \rangle^2 + 2\sum_{k: (k,j) \in S_{OG}}  \tilde \Xi_{kj}^2 \mathbf1\Big( |\tilde \Xi_{kj}|\ge\mu'\Big)\\
			&+ \sum_{k: (k,j) \in S_{OG}} 2 \tilde \Xi_{kj}^2 \mathbf1\Big( |\tilde \Xi_{kj}|<\mu' \le \frac{1-\delta}{\sqrt6 \delta}|\langle {\Phi}_{(k,j)}^\top, \tilde \beta^* -\tilde \beta^t \rangle|\Big)\\
	\le & 2s_0 \mu'^2 + 2\left( 1+ \frac{(1-\delta)^2}{6\delta^2}\right)\sum_{k: (k,j) \in S_{OG}} \langle {\Phi}_{(k,j)}^\top, \tilde \beta^* -\tilde \beta^t \rangle^2,
		\end{aligned}
	\end{equation}
	where the last inequality follows from 
	$$\sum_{k\in[d]} \tilde \Xi_{kj}^2 \mathbf1 (|\tilde \Xi_{kj}|\ge \mu')< s_0\mu'^2$$
	by using $\mathcal T_{\mu',s_0} \big( \tilde \Xi_{G_j}\big)=\mathbf 0_d$.
	Therefore, based on \eqref{OGgroupscaled}, we conclude that 
	$$\frac{6\delta^2}{(1-\delta)^2}s_0 \mu'^2 \le  \sum_{k: (k,j) \in S_{OG}} \langle {\Phi}_{(k,j)}^\top, \tilde \beta^* -\tilde \beta^t \rangle^2. $$
	Then, we get the upper bound as
	\begin{equation}\label{GroupXifailHsuccess}
		\begin{aligned}
			&  \sum_{(i,j) \in S_{OG}} \tilde \Xi_{ij}^2\mathbf1\left\{|\tilde \Xi_{ij}|\ge \mu', 
            %\right. \\ & \qquad\qquad \left.~
            \sum_{k\in [d]} \tilde \Xi_{kj}^2 \mathbf1 (|\tilde \Xi_{kj}|\ge \mu' )< s_0\mu'^2,~  \mathcal T_{\mu,s_0} \big(\tilde H^{t+1}\big)_{ij}\ne 0 \right\} \\
			=&\sum_{ j \in G_{OG}}~ \sum_{i:(i,j) \in S_{OG}} \tilde \Xi_{ij}^2 \mathbf1\Big\{|\tilde \Xi_{ij}|\ge \mu'\Big\}\\
			&\qquad\times  \mathbf1\left\{\sum_{k\in [d]} \tilde \Xi_{kj}^2 \mathbf1 (|\tilde \Xi_{kj}|\ge \mu')< s_0\mu'^2,~  \mathcal T_{\mu,s_0} \big(\tilde H^{t+1}\big)_{ij}\ne 0 \right\} \\
			\le & \sum_{j\in G_{OG}} s_0\mu'^2 \mathbf1 \left\{s_0 \mu'^2 \le  \frac{(1-\delta)^2}{6\delta^2} \sum_{k: (k,j) \in S_{OG}} \langle {\Phi}_{(k,j)}^\top, \tilde \beta^* -\tilde \beta^t \rangle^2 \right\} \\
			\le & \frac{(1-\delta)^2}{6\delta^2} \sum_{(i,j) \in S_{OG}} \langle {\Phi}_{(i,j)}^\top, \tilde \beta^*-\tilde\beta^t \rangle^2
            ~ \le ~ \frac{(1-\delta)^2}6 \left\| \tilde \beta^t - \tilde \beta^* \right\|_2^2 .
		\end{aligned}
	\end{equation}
	Therefore, we complete the proof of Lemma \ref{OGChi2:3rd}.
\end{proof}
}

\subsection{Lemma \ref{exactsupporterror} and its proof}
Finally, we analyze one essential error term associated with the true support set $S^*$.

\begin{lemma}\label{exactsupporterror}
{Under conditions of Proposition \ref{T8}, recall that the element-wise beta-min condition and the group-wise beta-min condition 
		\begin{equation*}
			\begin{aligned}
				&\min_{(i,j)\in S^*}|\beta^*_{ij}| \ge \left( 2 + \frac{\sqrt6 \delta}{1-\delta} \right) \mu,\\
				&\min_{j \in G^*}\|\beta^*_{G_j}\|_2 \ge 
				 \left( 2 + \frac{\sqrt6 \delta}{1-\delta} \right) \sqrt{s_0} \mu.
			\end{aligned}
		\end{equation*}
		%where $\mu_2 = 2\sigma \left( \sqrt{ \frac{6}{n} \Delta(1,s_0) } + \sqrt{ \frac{3}{n} \log(ss_0) } \right)$.
	}
	Then, we have
	\begin{align*} 
		&P\left( \sum_{(i,j) \in S^*} \left( \tilde H_{ij}^{t+1}\right)^2 \mathbf1\left\{ (i,j) \notin \tilde S^{t+1}\right\}
		< \frac{(1-\delta)^2}{3} \left\| \tilde\beta^t - \tilde \beta^* \right\|_2^2   \right)\\
		\ge& 1-O\left( e^{-\frac13 \left[ \Delta(1,s_0)+ \log(ss_0) \right] } \right).
	\end{align*}
\end{lemma}

\begin{proof}[Proof of Lemma \ref{exactsupporterror}]
	{
		Recall $s_j = |{G_j} \cap S^*|$ for all $ j \in G^*$.
		Then we have
		\begin{equation}\label{eq:exactsupp}
			\begin{aligned}
				& \sum_{ (i,j) \in S^* } \left( \tilde H_{ij}^{t+1}\right)^2 
				\mathbf1\left( (i,j) \notin \tilde S^{t+1}\right) \\
				\le &  \sum_{(i,j) \in S^*} \left( \tilde H_{ij}^{t+1}\right)^2 
				\mathbf1\left( |\tilde H_{ij}^{t+1}| < \mu \right)  \\
				&~+  \sum_{(i,j) \in S^*} \left( \tilde H_{ij}^{t+1}\right)^2 \cdot
				\mathbf1\left( |\tilde H_{ij}^{t+1}| \ge \mu, 
				~\sum_{k\in [d]}\left( \tilde H_{kj}^{t+1}\right)^2  \mathbf1 ( |\tilde H_{kj}^{t+1}| \ge \mu) < s_0 \mu^2 \right) \\
				\le &  \sum_{(i,j) \in S^*} \mu^2 \cdot \mathbf1\left( |\tilde H_{ij}^{t+1}| <\mu\right) \\
				&~+  \sum_{j \in G^*}  s_0 \mu^2 \cdot 
				\mathbf1\left( \sum_{k: (k,j) \in S^*}\left( \tilde H_{kj}^{t+1}\right)^2 \mathbf1 (|\tilde H_{kj}^{t+1}| \ge \mu)< s_0 \mu^2 \right) \\
				\le &  \sum_{(i,j) \in S^*} \mu^2 \cdot
				\mathbf1\left( |\tilde H_{ij}^{t+1}| < \mu \right) \\
				&~+ \sum_{j \in G^*} s_0 \mu^2 \cdot 
				\mathbf1\left( \sum_{k: (k,j) \in S^*}\left( \tilde H_{kj}^{t+1}\right)^2 < (s_j+s_0) \mu^2 \right).
			\end{aligned}
		\end{equation}
	}
	{
		Note that for all $ (i,j) \in S^*$, we have $\tilde H_{ij}^{t+1} = \tilde \beta^* + \langle{\Phi}_{(i,j)}^\top,\tilde\beta^*-\tilde\beta^{t}\rangle$ holds (since $\tilde \Xi_{ij} =0$ for all $ (i,j) \in S^*$).
		Therefore, for the first term in \eqref{eq:exactsupp}, 
		%with probability greater than $1 - 2 (ss_0)^{-1/2}$, 
		we have 
		\begin{equation}\label{eq:exactsupport1}
			\begin{aligned}
				\mathbf1\left( |\tilde H_{ij}^{t+1}| < \mu \right) 
				\le & \mathbf 1\Big( |\tilde \beta_{ij}^*| - | \langle{\Phi}_{(i,j)}^\top,\tilde\beta^*-\tilde\beta^{t}\rangle | < \mu \Big) \\ 
				\overset{(i)}{\le} & \mathbf 1\left( |\beta_{ij}^*|-\frac\mu3-|\langle{\Phi}_{(i,j)}^\top, \tilde\beta^*-\tilde\beta^{t}\rangle| < \mu \right)\\
				{\le} & \mathbf 1\left( \mu  < \frac{1-\delta}{\sqrt6 \delta} |\langle{\Phi}_{(i,j)}^\top, \tilde\beta^*-\tilde\beta^{t}\rangle| \right), 
			\end{aligned}
		\end{equation}
		where inequality (i) follows from the third inequality in \eqref{eq:exactprob}, and the last inequality follows from the element-wise beta-min condition $\min_{(i,j) \in S^*} |\beta_{ij}^*| \ge \left(2 + \frac{\sqrt6 \delta}{1-\delta} \right)\mu$.
	}
	
	{
		Additionally, for the second term in \eqref{eq:exactsupp}, by both element-wise and group-wise beta-min conditions, we conclude that $\| \beta_{G_j}^*\|_2 \ge \left( 2 + \frac{\sqrt6 \delta}{1-\delta} \right) \sqrt{s_j \vee s_0} \mu$. 
		Then, we have 
		\begin{equation}\label{eq:exactsupport2}
			\begin{aligned}
				&\mathbf1\left( \sum_{k: (k,j)\in S^*}\left( \tilde H_{kj}^{t+1}\right)^2 < (s_j+s_0) \mu^2 \right) \\
				\le & \mathbf1\left(  \sqrt{\sum_{k: (k,j)\in S^*} \left( \tilde \beta_{kj}^* \right)^2 }- \sqrt{\sum_{k: (k,j)\in S^*}\langle{\Phi}_{(k,j)}^\top,\tilde\beta^*-\tilde\beta^{t}\rangle^2} 
				< \sqrt{s_j+s_0} \mu \right) \\
				\le & \mathbf1\Bigg( \sqrt{\sum_{k: (k,j)\in S^*} \left( \beta_{kj}^* \right)^2 }-
				 \sqrt{ \sum_{k: (k,j)\in S^*} \left( \tilde \beta^*_{kj} - \beta^*_{kj}  \right)^2 }
				 - \sqrt{2 (s_j \vee s_0)} \mu\\
				 &\qquad\qquad  < \sqrt{\sum_{k: (k,j)\in S^*} \langle{\Phi}_{(k,j)}^\top,\tilde\beta^*-\tilde\beta^{t}\rangle^2} \Bigg)\\
				\le & \mathbf1\left( \sum_{k: (k,j)\in S^*} \langle{\Phi}_{(k,j)}^\top,\tilde\beta^*-\tilde\beta^{t}\rangle^2 
				>\frac{6 \delta^2}{(1-\delta)^2} (s_j \vee s_0 ) \mu^2 > \frac{6 \delta^2}{(1-\delta)^2} s_0 \mu^2  \right),
			\end{aligned}
		\end{equation}
		where the last inequality follows from the third inequality in \eqref{eq:exactprob}. 
		
		By applying \eqref{eq:exactsupport1} and \eqref{eq:exactsupport2} into \eqref{eq:exactsupp}, we derive that 
		\begin{align*}
			& \sum_{ (i,j) \in S^* } \left( \tilde H_{ij}^{t+1}\right)^2 
			\mathbf1\left( (i,j) \notin \tilde S^{t+1}\right) \\
			\le &  \sum_{(i,j) \in S^*} \mu^2 \cdot \mathbf 1\left( \mu  < \frac{1-\delta}{\sqrt6 \delta} |\langle{\Phi}_{(i,j)}^\top, \tilde\beta^*-\tilde\beta^{t}\rangle| \right)\\ 
			&+ \sum_{j \in G^*} s_0 \mu^2 \cdot 
			\mathbf1\left( \sum_{k: (k,j)\in S^*} \langle{\Phi}_{(k,j)}^\top,\tilde\beta^*-\tilde\beta^{t}\rangle^2 
			> \frac{6 \delta^2}{(1-\delta)^2} s_0 \mu^2  \right)\\
			\le &  \frac{(1-\delta)^2}{3 \delta^2} \sum_{(i,j) \in S^*}  \langle{\Phi}_{(i,j)}^\top, \tilde\beta^*-\tilde\beta^{t}\rangle^2 \\ 
			\le &  \frac{(1-\delta)^2}{3}  \left\|\tilde\beta^*-\tilde\beta^{t} \right\|_2^2,
		\end{align*}
		which completes the proof of Lemma \ref{exactsupporterror}.
	}
\end{proof}

\section{Auxiliary Lemmas for the minimax lower bounds}\label{appD}

This appendix provides some lemmas used in Appendix \ref{lowerproof}.

\begin{proof}[Proof of Lemma \ref{lemma:minimaxori}]
To facilitate the calculation, we use the double index $(i,j)$ to locate the $i$-th variable in the $j$-th group $G_j$.
Recall $\pi_{\Theta}(A) = \frac{\pi(A)}{\pi({\Theta})}$. 
Then, based on the property of expectation, we have
\begin{equation}\label{1siineq}
 \begin{aligned}
\inf_{\hat \beta} \sup_{\beta^* \in  {\Theta}} 
    \mathbf{E}_{Y\sim P_{\beta^*}} \|\hat \beta - \beta^*\|_2^2 
\ge & \inf_{\hat \beta} \mathbf E_{\beta^* \sim \pi_{\Theta} } 
    \mathbf{E}_{Y\sim P_{\beta^*} } \|\hat \beta - \beta^*\|_2^2 \\
\ge & \sum_{i=1}^d\sum_{j=1}^m \inf_{\hat \beta_{ij}} \mathbf E_{\beta^* \sim \pi_{\Theta} } 
    \mathbf{E}_{Y\sim P_{\beta^*} } \left( \hat \beta_{ij}- \beta_{ij}^{*{\Theta} }\right)^2 \\
\overset{(i)}{=}& \sum_{i,j} \inf_{\hat \beta_{ij} } 
    \mathbf E_{Y} \Bigg\{ \underset{\beta^{*{\Theta}}|Y}{\mathbf{E}}
    \left( \left( \hat \beta_{ij} - \beta_{ij}^{*{\Theta}} \right)^2 \Big| Y \right)\Bigg\}\\
\overset{(ii)}{ = }& \sum_{i,j} \mathbf E_{Y } \Bigg\{
    \underset{\beta^{*{\Theta}} |Y}{\mathbf{E}} \left( \left( 
    B_{ij}^{\Theta} - \beta_{ij}^{*{\Theta}} \right)^2 \Big| Y \right)\Bigg\}\\
=& \underset{\beta^* \sim \pi_{\Theta}}{\mathbf E}~\underset{Y \sim P_{\beta^*} }{\mathbf E}  
    \left( \sum_{i,j} \left(B_{ij}^{\Theta} - \beta_{ij}^{*{\Theta}} \right)^2 \right),
\end{aligned}
\end{equation}
where equality $(i)$ uses $\underset{\beta^{*{\Theta}}|Y}{\mathbf{E}}$ to indicate that the expectation is with respect to $\beta^{*{\Theta}}$ conditioned on $Y$, that is, the conditional distribution $\frac{ P(Y|\beta^*) \mathrm d \pi_{\Theta}(\beta^*)  }{ \int P(Y|\beta) \mathrm d \pi_{\Theta}(\beta^*) }$. 
And in inequality (ii) we define $B_{ij}^{\Theta} := \underset{\beta^{{\Theta}}|Y}{\mathbf{E} } (\beta_{ij}^{{\Theta}}|Y)$, where recall that the distribution of $\beta_{ij}^{{\Theta}}$ is independently identical with that of $\beta_{ij}^{*{\Theta}}$. 
According to Theorem 1.1 and Corollary 1.2 in Chapter 4 in \cite{LE06}, $B_{ij}^{\Theta}$ achieves the infimum.
Additionally, we have 
$(B_{ij}^{\Theta})^2 \le 
\underset{\beta^{{\Theta}}|Y}{\mathbf{E}} \left( \left(\beta_{ij}^{{\Theta}}\right)^2\Big| Y \right)$.

Next, by taking the infimum over all possible estimator $\hat T(Y,X)$, we obtain
\begin{equation}\label{3partori}
\begin{aligned}
&\inf_{\hat T } \underset{\beta^* \sim \pi}{\mathbf E}~\underset{Y \sim P_{\beta^*} }{\mathbf E} \|\hat T(Y,X) - \beta^*\|_2^2 \\
\le & \underset{\beta^* \sim \pi}{\mathbf E}~\underset{Y \sim P_{\beta^*} }{\mathbf E} \sum_{i,j} \left(B_{ij}^{\Theta}  - \beta_{ij}^* \right)^2 \\
\le & \underset{\beta^* \sim \pi}{\mathbf E}~\underset{Y \sim P_{\beta^*} }{\mathbf E} \left( \sum_{i,j} \left(B_{ij}^{\Theta} - \beta_{ij}^* \right)^2 \mathbf1(\beta^* \in {\Theta}) \right)\\
   & + \underset{\beta^* \sim \pi}{\mathbf E}~\underset{Y \sim P_{\beta^*} }{\mathbf E} \left( \sum_{i,j}\left(B_{ij}^{\Theta} - \beta_{ij}^* \right)^2 \mathbf1(\beta^* \notin {\Theta}) \right) \\
%\le& \underset{\beta^* \sim \pi_{\Theta}}{\mathbf E}~\underset{Y \sim P_{\beta^*} }{\mathbf E}  \left( \sum_{i,j} \left(B_{ij}^{\Theta} - \beta_{ij}^{*{\Theta}} \right)^2 \right)+ 2\underset{\beta^* \sim \pi}{\mathbf E}~\underset{Y \sim P_{\beta^*} }{\mathbf E} \left\{  \sum_{i,j}\left( \left( B_{ij}^{\Theta}\right)^2 + \beta_{ij}^2  \right) \mathbf1(\beta^* \notin {\Theta}) \right\} \\
\le& \underset{\beta^* \sim \pi_{\Theta}}{\mathbf E}~\underset{Y \sim P_{\beta^* } }{\mathbf E}  
    \left( \sum_{i,j} \left(B_{ij}^{\Theta} - \beta_{ij}^{*{\Theta}} \right)^2 \right)\\
    &+ 2\underset{\beta^* \sim \pi}{\mathbf E}~\underset{Y \sim P_{\beta^*} }{\mathbf E} 
    \left\{  \sum_{i,j} \left( \underset{\beta^{ {\Theta}}|Y}{\mathbf E} 
    \left( \left(\beta_{ij}^{ {\Theta}}\right)^2 \big| Y \right)  
    + (\beta_{ij}^*)^2  \right) \mathbf1(\beta^* \notin {\Theta}) \right\}.
\end{aligned}
\end{equation}
%where the final inequality uses $(B_{ij}^{\Theta})^2 \le  \underset{\beta|Y}{\mathbf{E}^{\Theta}} \left( \left(\beta_{ij}^{\Theta} \right)^2\Big| Y \right)$.

Combining \eqref{3partori} and \eqref{1siineq} together, we get the lower bound of the minimax risk, that is,
\begin{equation*} 
\begin{aligned}
&\inf_{\hat \beta} \sup_{\beta^* \in  {\Theta}} 
    \mathbf{E}_{Y\sim P_{\beta^*} } \|\hat \beta - \beta^* \|_2^2 \\
\ge& \mathbf E_{\beta^* \sim \pi_{\Theta}} \mathbf E_{Y \sim P_{\beta^* } } 
    \left( \sum_{i,j} \left(B_{ij}^{\Theta} - \beta_{ij}^{\Theta} \right)^2 \right)\\
\ge &\inf_{\hat T } \left\{ \underset{\beta^* \sim \pi}{\mathbf E}~ \underset{Y \sim P_{\beta^*} }{\mathbf E} \|\hat T(Y,X) - \beta^* \|_2^2 \right\} \\
    & -2 \underset{\beta^* \sim \pi}{\mathbf E} ~ \underset{Y \sim P_{\beta^*} }{\mathbf E}
    \Bigg\{ \bigg( \underset{\beta^{\Theta} | Y}{\mathbf{E}} \Big( \|\beta^{{\Theta}} \|_2^2\big| Y \Big) + \| \beta^* \|_2^2  \bigg) 
    \mathbf1(\beta^* \notin {\Theta}) \Bigg\},\\
\end{aligned}
\end{equation*}
which completes the proof of Lemma \ref{lemma:minimaxori}.
\end{proof}

The next lemma is a concentration inequality of the binomial distribution in Appendix D.4 in \cite{mohri}.
\begin{lemma}[Bernstein's inequality]\label{berncen}
    Suppose $u \sim Bin(n,p),~ p \in (0,1)$, then, for every $\lambda >0$, 
    \begin{equation*}
       P \left(\frac un - p \ge  \lambda \right) \le \exp \left( -\frac{n \lambda^2}{2p(1-p) + \frac{2\lambda}{3} } \right).
    \end{equation*}
\end{lemma}

The following lemma demonstrates the monotonicity of a function, which will be utilized in the proof of Theorem \ref{allselectorD1D2}.
\begin{lemma}\label{mdPsi}
Recall that 
\begin{align*}
  t(d,s_0,a,\sigma ) &:= \frac{a\sqrt n}{2} + \frac{\sigma^2}{a\sqrt n} \log \frac{d-s_0}{s_0},\\
   \psi(d,s_0,a,\sigma) & :=(d-s_0){\Phi}\left(-\frac{t(d,s_0,a,\sigma )}{\sigma}\right) + s_0 {\Phi}\left(-\frac{a\sqrt n-t(d,s_0,a,\sigma )}{\sigma}\right).
\end{align*}
Then for the fixed $d,~s_0, ~a,~\sigma$, we have
\begin{equation*}
\frac{\psi(d,s_0',a,\sigma)}{s_0'} \ge \frac{\psi(d,s_0,a,\sigma)}{s_0}, \text{ for every }  s_0' \in (0, s_0].
\end{equation*}
\end{lemma}

\begin{proof}[Proof of Lemma \ref{mdPsi}]
%\begin{proof} %(Proof of Lemma \ref{mdPsi})
For ease of display, for the fixed $d,~a,~\sigma$ and for every $r\in (0,d)$, we define
\begin{equation*}
A(r) :=  \frac{\psi(d, r, a,\sigma)}{r}
= \frac{d-r}{r} {\Phi}\left( -\frac{ t(d, r, a, \sigma) }\sigma \right) 
    + {\Phi}\left( -\frac{ a\sqrt{n} - t(d, r, a, \sigma) }\sigma \right).
\end{equation*}
Thus we have 
\begin{equation*}
A'(r) = -\frac{d}{r^2} {\Phi}\left( -\frac{ t }\sigma \right)
  +\frac1\sigma \frac{\partial t }{\partial r} \cdot
   \left\{ -\frac{d-r}{r} \varphi\left( -\frac{ t }\sigma \right)
    + \varphi\left( -\frac{ a\sqrt{n} - t }\sigma \right)  \right\},
\end{equation*}
% A'(r) = -\frac{d}{r^2} {\Phi}\left( -\frac{ t(d, r, a, \sigma) }\sigma \right)\\& +\frac1\sigma \frac{\partial t(d, r, a, \sigma) }{\partial r}\left\{ -\frac{d-r}{r} \varphi\left( -\frac{ t(d, r, a, \sigma) }\sigma \right) + \varphi\left( -\frac{ a\sqrt{n} - t(d, r, a, \sigma) }\sigma \right)  \right\} 
where $\varphi(\cdot)$ is the probability distribution function of the standard Gaussian distribution.

Note that
\begin{equation*}\label{eq:equivalence}
\begin{aligned}
&\frac{d-r}{r} \varphi\left( -\frac{ t }\sigma \right)\\
=& \frac{1}{\sqrt{2\pi}}\frac{d-r}r \exp\left(-\frac1{2\sigma^2} \left[\frac{a^2 n}4
   + \frac{\sigma^4}{a^2 n} \log^2\frac{d-r}{r}+ \sigma^2 \log\frac{d-r}{r}\right] \right)\\
=& \frac{1}{\sqrt{2\pi}} \exp\left( -\frac{ 1 }{2 \sigma^2} \left[ \frac{a^2 n}4
   + \frac{\sigma^4}{a^2 n} \log^2\frac{d-r}{r} - \sigma^2 \log\frac{d-r}{r}\right]  \right)\\
=& \varphi\left( -\frac{ a\sqrt{n}- t }\sigma \right),
\end{aligned}
\end{equation*}
which leads $A'(r) = -\frac{d}{r^2} {\Phi}\left( -\frac{t}\sigma \right) <0$ for all $ r\in (0,d)$. 
Therefore, we complete the proof of Lemma \ref{mdPsi}.
\end{proof}

\end{appendix}

\end{document}